\documentclass{article}
\usepackage[utf8]{inputenc}

\usepackage{amsmath,amssymb,amsthm,amsfonts,amscd}
\usepackage[all]{xy}
\usepackage{color}
\usepackage[english,russian]{babel}
\usepackage{graphicx}

\definecolor{grey}{rgb}{0.5,0.5,0.5}
\definecolor{red}{rgb}{1,0,0}

\newcommand\grey[1]{  }

\newtheorem{theorem}{Theorem}
\newtheorem*{main*}{Main Theorem}
\newtheorem*{theorem*}{Theorem}
\newtheorem{lemma}[theorem]{Lemma}
\newtheorem{proposition}[theorem]{Proposition}

\theoremstyle{definition}
\newtheorem{definition}[theorem]{Definition}
\newtheorem{remark}[theorem]{Remark}
\newtheorem*{remark*}{Remark}
\newtheorem*{example*}{Example}

\def\Z{{\Bbb Z}}
\def\R{{\Bbb R}}
\def\RP{{\Bbb R}\!{\rm P}}
\def\N{{\bf N}}
\def\NN{{\bf NN}}

\def\C{{\Bbb C}}
\def\c{{c}}
\def\k{\mathfrak{k}}
\def\eps{{\varepsilon}}

\def\D{{\bf D}}
\def\Q{{\bf Q}}
\def\N{{\bf N}}
\def\i{{\bf i}}
	\def\y{{\bf y}}
\def\aa{\dot{a}}
\def\j{{\bf j}}
\def\int{{\rtimes}}

\def\J{{\bf J}}
	\def\JJ{{\dot{\bf J}}}

\def\I{{\bf I}}
\def\a{{\bf{{a}}}}
\def\b{{\bf b}}
\def\bb{{\dot{b}}}
\def\e{{\bf e}}
\def\II{{\dot{\bf I}}}
\def\f{{\bf f}}
	\def\x{{\bf x}}
\def\d{{\bf d}}
\def\dd{\dot{d}}
\def\H{{\bf H}}

\def\dist{{\operatorname{dist}}}

\def\dim{{\operatorname{dim}}}
\def\dist{\operatorname{dist}}

\def\eps{{\varepsilon}}

\def\diag{\mathrm{diag}}

\renewcommand\kappa\varkappa

\def\eps{\varepsilon}

\renewcommand\Re{\mathrm{Re}\,}
\renewcommand\Im{\mathrm{Im}\,}

		\def\Z{{\Bbb Z}}
		\def\R{{\Bbb R}}
		\def\RP{{\Bbb R}\!{\rm P}}
		\def\N{{\bf N}}
		
		\def\C{{\Bbb C}}
		\def\c{{\bf c}}
			\def\x{{\bf x}}
				\def\y{{\bf y}}

		\def\D{{\bf D}}
		\def\Q{{\bf Q}}
			\def\QQ{{\dot{\mathbf Q}}}
		\def\N{{\bf N}}
		\def\i{{\bf i}}
		\def\j{{\bf j}}

		\def\J{{\bf J}}
		
		\def\I{{\bf I}}
		\def\a{\bf a}
		\def\b{{\bf b}}
		\def\bb{{\dot{b}}}
		\def\e{{\bf e}}
		\def\II{{\dot{\bf I}}}
		\def\f{{\bf f}}
		\def\d{{\bf d}}
		\def\H{{\bf H}}

		\def\dist{{\operatorname{dist}}}
		
		\def\dim{{\operatorname{dim}}}
		\def\dist{\operatorname{dist}}

		\def\eps{{\varepsilon}}

		\def\diag{\mathrm{diag}}

		\renewcommand\kappa\varkappa

		\def\eps{\varepsilon}
		
		\renewcommand\Re{\mathrm{Re}\,}
		\renewcommand\Im{\mathrm{Im}\,}

		\newcommand\eqq{\begin{equation}}
			\newcommand\eeq{\end{equation}}

\begin{document}
		\title{Geometric approach to stable
			homotopy groups of spheres II;  
			Arf-Kervaire Invariants}
		\author{P.M. Akhmet'ev}
		
\date{July, 2025}

\maketitle

\begin{abstract}
	The Kervaire Invariant 1 Problem  until recently was an open problem in algebraic topology. 
	Hill-Hopkins-Ravenel theorem clams a negative solution of the problem for all dimensions $n=2^l-2$, $l \ge 8$.  We prove the statement of Hill-Hopkins-Ravenel theorem for all dimensions $2^l-2$,  $l \ge l_0$, where $l_0$ is a sufficiently great positive integer.

	The proof is based on the Hirsh control principle and the Compression theorem by the author. A notion internal symmetry: of Abelian (for skew-framed immersions), bi-cyclic (for $\Z/2^{[3]}$-framed immersions) and quaternion-cyclic structure (for $\Z/2^{[4]}$-framed immersions) are introduced. 
	
\end{abstract}

\section{Self-intersections of generic immersions and the Kervaire invariant; the problem statement}

The Kervaire Invariant 1 Problem  until recently was an open problem in algebraic topology. Presently, an algebraic solution was proposed in  \cite{H-H-R}.
In the  paper (part II) an alternative geometrical solution is proposed.  Our approach develops an approach by L.S.Pontryagin \cite{P},
and by V.A.Rokhlin \cite{Ro}.

A new idea concerns a conception of "`internal symmetry"' which, briefly, means that a framed immersion, which is represented an element of stable homotopy group by the
Pontryagin construction, is generic and in a very special position. The monodromy of the canonical covering over  self-intersection manifolds of the immersion admits a special reduction and contains "`internal characteristic numbers"'. There are two applications: geometrical solution of Hopf invariant problem and Kervaire invariant problem.
Let us skip the first problem and let us investigate Kervaire invariant. 

A geometrical overlook of the Kervaire problem is presented in \cite{Akh1}. A main idea of this paper is the following. Until the 	Hill-Hopkins-Ravenel result many 
authors tray to prove that infinite many dimensions with Kervaire invariants exist. As a new variant of Novikov's idea to investigate cobordisms with additional structuring group of normal bundle, the author develops investigation of Eccles's  stable framed cobordism, which is very closed to classical case of Pontryagin's framed cobordism.
Geometrical tools toward the generalized Kervaire problem is similar to the Kervaire problem, but the result is quite different: there are infinitely many dimensions with stably-framed manifold of Kervaire invariant one. The author assume that investigations of a relationship between Generalized Kervaire problem and Kervaire problem will give a new 
results in Stable homotopy groups of spheres, in present we know very few results concerning infinite series of invariants. 

The conception of internal symmetry has an  physical analog in MHD for magnetic fields in liquid conductive medium. The Kolmogorov energy spectrum in two-scale MHD model could be visualized by "`internal symmetry"' of magnetic lines:  Massey quadruple (generalized) invariant of randomly colored magnetic lines  gives a geometrical visualization of fractal structure for magnetic fields with Kolmogorov's energy distribution \cite{Akh2}.  This conception at present gives no new physical applications, but gives many interesting problems for investigation.

Let us start with precise definition to make a conception of "`internal symmetry"' clear.
Let us consider a smooth generic immersion 
$f: M^{n-1} \looparrowright \R^n$,
$n= 2^{\ell} -2$, $\ell>1$ of the codimension $1$. Denote by 
$g: N^{n-2} \looparrowright \R^n$ the immersion of self-intersection manifold of $f$.

Let us recall a definition of the cobordism group
$Imm^{sf}(n-k,k)$, a particular case $k=1$  is better known:
$Imm^{sf}(n-1,1)$.
The cobordism group is defined as equivalent classes of triples up to the standard 
cobordism relation, equipped with a disjoint union operation; 
\[  \]

$\bullet$ $f: M^{n-k} \looparrowright \R^n$ is a codimension $k$ immersion;

$\bullet$  $\Xi: \nu(f) \cong k \kappa_M^{\ast}(\gamma)$ is a bundle map, which is invertible,
an invertible bundle map is a fibrewized isomorphism; 

$\bullet$ $\kappa_M \in H^1(M^{n-k};\Z/2)$ is a prescribed cohomology class, which is 
a mapping $M^{n-k} \to \RP^{\infty}=K(\Z/2,1)$.

By $\nu(f)$ is denoted the normal bundle of the immersion $f$, by $\kappa_M^{\ast}(\gamma)$ is denoted the pull-back of
the universal line bundle $\gamma$ over $\RP^{\infty}$ by the mapping $\kappa_M$, 
by $k \kappa_M^{\ast}(\gamma)$ is denoted the Whitney sum of the $k$ copies of the line bundles,
below for short we will write $k \kappa_M$. The isomorphism $\Xi$ is called a skew-framing
of the immersion $f$.

The Kervaire invariant is an invariant of cobordism classes,  which is homomorphism
\begin{eqnarray}\label{1a}
\Theta^{sf}: Imm^{sf}(n-k,k) \to  \Z/2.
\end{eqnarray}
Let us recall the homomorphism $(\ref{1a})$ in the case $k=1$.

The normal bundle
$\nu_g$ of the immersion  $g:
N^{n-2} \looparrowright \R^n$ 
is a $2$-dimensional bundle over
$N^{n-2}$, which is equipped by a $\D$--framing $\Xi$, where  $\D$ is the dihedral
group of the order $8$. The classifying map of this bundle (and also the corresponding characteristic class) is denoted by
$\eta_N: N^{n-2} \to K(\D,1)$.  
Triples
$(g, \eta_N, \Xi)$ up to a cobordism relation represents an element in the group
$Imm^{\D}(n-2,2)$. The correspondence $(f,\kappa,\Psi) \mapsto (g, \eta_N, \Xi)$
defines a homomorphism
\begin{eqnarray}\label{2a}
\delta^{\D}: Imm^{sf}(n-1,1) \to Imm^{\D}(n-2,2).
\end{eqnarray}

\begin{definition}
The Kervaire invariant of an immersion
$f$ is defined by the formula: 
\begin{eqnarray}\label{arf}
\Theta_{sf}(f) = \langle \eta_N^{\frac{n-2}{2}}; [N^{n-2}] \rangle.
\end{eqnarray}
\end{definition}

It is not difficult to prove that the formula
$(\ref{arf})$ determines a homomorphism 
\begin{eqnarray}\label{3a}
\Theta^{\D}: Imm^{\D}(n-2,2) \to \Z/2.
\end{eqnarray}
The homomorphism
($\ref{3a}$) is called the Kervaire invariant of a
$\D$--framed immersion. 
The composition of the homomorphisms $(\ref{2a})$, $(\ref{3a})$ 
determines the homomorphism  $(\ref{1a})$.

\subsection*{Main Result}

 Assume an element $x \in Imm^{sf}(n - 1, 1)$
admits a compression of the order $7$ (see Definition \ref{Def5}).
Then in  the case $n= 2^{\ell} -2$, $\ell \ge 8$, the homomorphism $\Theta^{sf}$ (\ref{arf}) on $x$ is trivial: $\Theta^{sf}(x)=0$.

\begin{remark}

Main Result with an assumption $\ell \ge 8$  is a corollary of \cite{H-H-R}. If an arbitrary element $x$, $\ell \ge 8$, admits a compression of the order $7$ is unknown. 
 Main Result and the following Compression Theorem gives an alternative proof of Hill-Hopkins-Ravenel theorem for all $n \ge 256$, probably, except a finite number of exceptional cases.
\end{remark}

\subsubsection*{A short proof of Main Result} 
Main Result is proved in Theorem \ref{Arf}, where the Definition \ref{pure} of standardized $\Z/2^{[4]}$-framed immersion is used. The standardized immersion is a result of a Smale-Hirsh control principle  \cite{Hi} for iterated self-intersection points of a standardized $\D$-framed immersion (Subsection \ref{dihedstand}). The Arf-invariant for 
standardized $\Z/2^{[4]}$-framed immersion is determined by a homology class, described in 
Lemma \ref{dd} and then by the the middle (exceptional) class in
Definition \ref{pure}. \qed

\subsubsection*{Compression Theorem \cite{Akh}}
For an arbitrary positive integer $d$ there exists $\ell = \ell(d)$, such that an arbitrary
element in $Imm^{sf}(n - 1, 1)$, $n = 2^{\ell}-2$, admits a compression of the order $d$.
\[ \]

Let us explain steps in the proof Main Theorem. 
Definition of the cobordism group
$Imm^{\D}(n-2,2)$ of immersions with a dihedral framing  is standard and analogous to the above definition of the skew-framed cobordism group $Imm^{sf}(n-k,k)$  for $k=1$. This group also admits a standard generalization as the cobordism group $Imm^{\D}(n-2k,2k)$ with a parameter $k \ge 1$. 
 Each element in 
$Imm^{\D}(n-2k,2k)$ is represented by a triple $(g,\Psi,\eta_N)$, where
$g: N^{n-2k} \looparrowright \R^n$ is an immersion, 
$\Psi$-- is a dihedral framing in the codimension  $2k$, 
i.e. an isomorphism
\begin{eqnarray}\label{isom}
\Xi: \nu_g \cong k \eta_N^{\ast}(\tau), 
\end{eqnarray}
where $\eta_N: N \to K(\D,1)$ is the classifying mapping, the universal $2$-dimensional $\D$-bundle over $K(\D,1)$ is denoted by 
$\tau$. 
The mapping 
$\eta_N$ determines a dihedral framing of the normal bundle
$\nu_g$ by $\Xi$.

The Kervaire homomorphism is defined in a more general case using the following generalizations of the
homomorphisms 
($\ref{1a}$) and ($\ref{3a}$):
\begin{eqnarray}\label{4a}
\Theta^{sf}_{k}: Imm^{sf}(n-k,k) \to \Z/2, \qquad \Theta^{sf}_k: =
\Theta^{\D}_k  \circ \delta^{\D}_k.
\end{eqnarray}
\begin{eqnarray}\label{44}
\Theta^{\D}_{k} : Imm^{\D}(n-2k,2k) \to \Z/2, \qquad
\Theta^{\D}_{k}[(g,\Psi,\eta_N)] = \langle
\eta_N^{\frac{n-2k}{2}}; [N^{n-2k}] \rangle.
\end{eqnarray}
In the case
$k=1$ the homomorphism  ($\ref{44}$) coincides with 
the homomorphism
($\ref{3a}$) above, the following commutative diagram, in which the homomorphisms  $J^{sf}$, $J^{\D}$ 
are known, is well defined.

\begin{eqnarray}\label{5a}
\begin{array}{ccccc}
Imm^{sf}(n-1,1) & \stackrel  {\delta^{\D}}{\longrightarrow} &
Imm^{\D}(n-2,2) & \stackrel{\Theta^{\D}}{\longrightarrow} & \Z/2  \\
\downarrow J^{sf}_k & &  \downarrow J^{\D}_{k}  &  &  \vert \vert \\
Imm^{sf}(n-k,k) & \stackrel{\delta^{\D}_{k}}{\longrightarrow} &
Imm^{\D}(n-2k,2k) &
\stackrel{\Theta^{\D}_k}{\longrightarrow} & \Z/2.  \\
\end{array}
\end{eqnarray}


Let us remind, that an idea to investigate intersection points as in Theorem \ref{T11}
was communicated to me by A.V.Chernavskii (1996). S.A.Melikhov (2005) has noted that a quaternion analogue
of  the Chernavskii mapping as in Lemma \ref{standQ} is required. An idea to make a $\Z/2$-control as in the formula 
(\ref{KKin}) was realized after a talk by E.V.Scepin (2004).
As a conclusion, results was collected-out at the A.S.Mishchenko Seminar in  
2008-2024. Alexandr S. Mischenko and  Th. Yu. Popelenskii allow discussions, in particular, we will use Popelenskii's formulas in Section \ref{sec8}.

\section{First step}\label{sec1}

\subsection{The dihedral group $\Z/2^{[2]} = \D$}
In the present and the next sections the cobordism groups
$Imm^{sf}(n-k,k)$, $Imm^{\Z/2^{[2]}}(n-2k,2k)$ will be used. 
In the case the first argument in the bracket is strongly positive, the cobordism group is finite.

The dihedral group
$\Z/2^{[2]} = \D$ is defined by its corepresentation 
$$\{a,b \,\vert \, b^4 = a^2 = e, [a,b]=b^2\}.$$
This group is represented by rotations a subgroup in the group $O(2)$ of orthogonal transformation of the
standard plane. Elements transforms the base $\{\e_1,\e_2\}$ on the plane
$Lin(\e_1,\e_2)$ to itself, a non-ordered pair of coordinate lines on the plane 
are preserved by transformations. 
The element $b$ is represented by the rotation of the plane by the angle
$\frac{\pi}{2}$. The element $a$ is represented by the reflection of the plane 
relative to the straight line 
$l_1=Lin(\e_1 + \e_2)$ parallel to the vector  $\e_1 + \e_2$.

Let us consider a subgroup
$\I_{a \times \aa} = \I_a \times \II_a \subset \Z/2^{[2]}$ in the dihedral group,
which is generated by the elements  $\{a,b^2a\}$. This is an elementary $2$-group of the rank $2$. Transformations  of this group keep each line  $l_1$, $l_2$ with the base vectors $\f_1=\e_1+\e_2$, $\f_2=\e_1-\e_2$ correspondingly. 
The cohomology group $H^1(K(\I_a \times
\II_a,1);\Z/2)$ contains two generators $\kappa_a, \kappa_{\aa}$.

Let us define the cohomology classes
\begin{eqnarray}\label{kappab}
\kappa_a \in H^1(K(\I_a \times \II_a,1);\Z/2)), \quad \kappa_{\aa} \in H^1(K(\I_a \times \II_a,1);\Z/2)).
\end{eqnarray}
Denote by $p_{a}: \I_a \times \II_a \to
\I_a$ a projection, the kernel of $p_a$ consists  the symmetry transformation with respect to
the bisector of the second coordinate angle and the identity.

Denote $\kappa_a = p_a^{\ast}(t_a)$, where $e \ne t_a \in H^1(K(\I_a,1);\Z/2)
\simeq \Z/2$. Let us denote by  $p_{\aa}: \I_a \times \II_a \to
\II_a$ the projection, the kernel of $p_{\aa}$ consists of the symmetry
with respect to  the bisector of the first coordinate angle and the identity.

Let us denote $\kappa_{\aa} =
p_{\aa}^{\ast}(t_{\aa})$, where $e \ne t_{\aa} \in H^1(K(\II_{a},1);\Z/2) \cong
\Z/2$.

\subsection{A standardized  immersion with a dihedral framing \label{dihedstand}} 

Let us consider a $\D$-framed immersion 
$(g,\Psi,\eta_N)$ of codimension $2k$. Let us assume that 
the image of $g$ contains in a regular neighbourhood $U(\RP^2)$ of a prescribed embedding
$\RP^2 \subset \R^n$. The following mapping
$\pi \circ \eta_N: N^{n-2k} \to K(\D,1) \to K(\Z/2,1)$ is well defined, where 
$\pi: K(\D,1) \to K(\Z/2,1)$ is the epimorphism with the kernel $\I_a \times \II_a \subset \D$.
It is required that this mapping coincides (up to homotopy) to the composition 
$i \circ PROJ \circ g$, where
$PROJ: U(\RP^2) \to \RP^2$ is the projection of the neighbourhood onto its central line, 
$i: \RP^2 \subset K(\Z/2,1)$ is the standard inclusion, which transforms the fundamental class into the generator. 


\subsubsection{Condition C1 for standardized  immersions }\label{cond11}

 Let us consider the submanifold $N_{sing}^{n-2k-2} \subset N^{n-2k}$, $N_{sing}^{n-2k-2} = 
PROJ^{-1}(\RP^0), \RP^0 \subset \RP^2$.  
Require that the mapping
$\eta_N$, restricted to
$N_{sing}$, be skipped trough the skeleton of $K(\I_a \times \II_a,1)$ of the dimension strictly less then $\frac{3(n-2k)}{4}$.

Geometrically, this means that the submanifold  $N_{sing}$ is a defect of a reduction 
of the structured mapping  of the normal bundle with a control of the subgroup
$\I_a \times \II_a \subset \D$ over $\RP^2$ to a mapping with a control of the same type over $\RP^1 \subset \RP^2$.

In the case 
$k \to 1+$ the structuring mapping on the defect manifold is skipped through  a polyhedron of the dimension
$\frac{3n}{4}$.

The complement $N \setminus N_{sing}$ to the defect we get an open manifold, for which the classifying normal bundle mapping is given by
$$\eta_N: N \setminus N_{sing} \to  K(\I_a \times \II_a) \rtimes_{\chi} \RP^1,$$
for the involution $\chi = \chi^{[2]}$, see section  \ref{sec7}, Definition  \ref{chi2}.

\begin{definition}\label{stand}
Let us say that a $\D$-framed immersion
$(g,\Psi,\eta_N;\Upsilon)$ of the codimension $2k$ is standardized by $\Upsilon$ with 
a defect $\delta < \approx  \frac{3n}{4}$, if the above condition $C1$, subsubsection \ref{cond11},  is satisfied. 

Quadruples $(g,\Psi,\eta_N;\Upsilon)$, up to cobordism relations  generate an abelian group
$Imm^{D;stand}(n-2k,2k)$; there is natural forgetful homomorphism: 
\begin{eqnarray}\label{for}
FG: Imm^{\D;stand}(n-2k,2k) \longrightarrow Imm^{\D}(n-2k,2k),
\end{eqnarray}
 after a control structure $\Upsilon$ is omitted. 
\end{definition}

\begin{definition}\label{Def5}
Assume $[(f,\Xi,\kappa_M)] \in Imm^{sf}(n-k,k)$,
$f: M^{n-k} \looparrowright \R^n$, $\kappa_M \in H^1(M^{n-k};\Z/2)$,
$\Xi$ is a skew-framing. Let us say that the pair  $(M^{n-k},\kappa_M)$
admits a compression of an order $q$, if the mapping
$\kappa_M :
M^{n-k} \to \RP^{\infty}$ is represented (up to homotopy) by the following composition: 
$\kappa = I \circ \kappa_M' : M^{n-k} \to \RP^{n-k-q-1} \subset
\RP^{\infty}$. Let us say that the element  $[(f, \Xi, \kappa_M)]$
admits a compression of an order $q$, if in its cobordism class exists a triple
$(f',
\Xi', \kappa_{M'})$, which admits a compression of the order $q$.
\end{definition}

Below we assume
\begin{eqnarray}\label{nsigma}
 n_{\sigma} = 14, \quad k = \frac{n-n_{\sigma}}{16}.
 \end{eqnarray}
  In particular, in the case $n = 254$ we get $k = 15$.

\begin{theorem}\label{Th2}


 Assume that $\ell \ge 8$ and that an element $x \in
Imm^{sf}(n-\frac{n-n_{\sigma}}{16},\frac{n-n_{\sigma}}{16}) = Imm^{sf}(n-k,k)$ admits a compression of the order
$q=7$  
By this additional assumption one may get an element $y = \delta_k^{\D} \in 
Imm^{\D}(n-2k,2k)$, \footnote{The homomorphism $Imm^{sf}(n-k,k) \to Imm^{\D;stand}(n-2k,2k)$ is not claimed} such that there exist an element 
$z \in Imm^{\D;stand}(n-2k;2k)$, $z=[(g,\eta_N; \Xi; \Upsilon)]$, $FG(z)=y$ by $(\ref{for})$, $\Upsilon$ is the normalized structure  with an additional condition:  a submanifold $N^{n_{\sigma}}_{a \times \aa} \subset N^{n-2k}$, 
 \begin{eqnarray}\label{Nsigma}
N^{n_{\sigma}}_{a \times \aa} = \eta_N^{-1}(K(\infty-(n-2k-n_{\sigma})) \subset K(\D,1)),
\end{eqnarray}
where $K(\infty-(n-2k-n_{\sigma})) \subset K(\D,1)$ is a skeleton of the even codimension $(n-2k-n_{\sigma})$ in the 
universal space $K(\D,1)$, determined by the Euler class of the bundle $\frac{(n-2k-n_{\sigma})}{2}\tau$ over $K(\D,1)$: the Whitney sum of $\frac{(n-2k-n_{\sigma})}{2}$ copies of the universal $2$-bundle $\tau$ 
(see $(\ref{isom})$),
admits a reduction of its structured subgroup  $\I_a \times \II_a \subset \D$ of the normal bundle 
(equivalently, the mapping $\pi \circ g: N^{n-2k} \looparrowright U(\RP^2) \to \RP^2$,
restricted on the submanifold $N^{n_{\sigma}}_{a \times \aa}$, homotopic to the constant mapping). 

\end{theorem}

\subsubsection{A preliminary construction for Theorem \ref{Th2}}\label{/N}

Let
\begin{eqnarray}\label{d2}
d^{(2)}: \RP^{n-k'} \times \RP^{n-k'} \setminus \RP^{n-k'}_{diag} \to \R^n \times \R^n
\end{eqnarray}
be an arbitrary $(T_{\RP^{n-k'} \times \RP^{n-k'}}, T_{\R^n \times \R^n})$--equivariant mapping, which is transversal along the diagonal
$\R^n_{diag} \subset \R^n \times \R^n$. 
The diagonal in the preimage is mapped into the diagonal of the image, by this reason the equivariant
mapping $d^{(2)}$ is defined on the open manifold outside of the diagonal (in the pre-image).
(In (\ref{d2})  $k'$ has to be defined
by the formula: $k'=k+q+1$ (parameters $k$,$q$ correspond to denotations of Theorem 
$\ref{Th2}$).)

Let us re-denote
$(d^{(2)})^{-1}(\R^n_{diag})/T_{\RP^{n-k'} \times \RP^{n-k'}}$ by $\N_{\circ}=\N(d^{(2)})_{\circ}$ for short, 
this polyhedron is called  a polyhedron of (formal) self-intersection of the equivariant mapping   $d^{(2)}$. 

The polyhedron  $\N_{\circ}$ is an open polyhedron, this polyhedron admits a compactification, which is denoted by $\N$ with a boundary $\partial \N$. The boundary consists of all critical (formal critical) points of the 
mapping $d^{(2)}$. One has:
$\N \setminus \partial \N = \N_{\circ}$. Let us denote by  $U(\partial \N)_{\circ}$ a 
thin regular neighbourhood of the boundary (= a diagonal)
$\partial \N$.

For a special mapping $d^{(2)}$ the polyhedron $\N_{\circ}$  admits the following lift: 
\begin{eqnarray}\label{abels}
\eta_{\circ;Ab}: \N_{\circ} \to K(\I_{a \times \aa} \int_{\chi^{[2]}} \Z,1).
\end{eqnarray}
The relation between   the structuring mapping  $\eta_{\circ}$ and its lift (\ref{abels}) is the following: 
\begin{eqnarray}\label{etacirc}
\eta_{\circ}: \N_{\circ} \stackrel{\eta_{\circ;Ab}}{\longrightarrow} K(\I_{a \times \aa} \int_{\chi^{[2]}}),\Z,1) \longrightarrow  K(\Z/2^{[2]},1).
\end{eqnarray}
On the polyhedron $U(\partial \N)_{\circ}$ the mapping  $ \eta_{Ab}$ 
gets the values into the following subcomplex:
$K(\I_{a \times \aa} \times 2\Z,1) \subset  K(\I_{a \times \aa} \int_{\chi^{[2]}} \Z,1)$ (the projection on $K(\Z,1)=S^1$ is the analogue of the projection of the Moebius band onto its central line).

\begin{definition}\label{hol} 

Let us call  a formal (equivariant) mapping
$d^{(2)}$, given by $(\ref{d2})$, is holonomic, if this mapping 
is the formal extension of a mapping 
\begin{eqnarray}\label{dmain}
d: \RP^{n-k'} \to \R^n.
\end{eqnarray}  

\end{definition}

\begin{definition}\label{abstru}
Assume a formal (equivariant) mapping $(\ref{d2})$ is holonomic. 
Let us say $d^{(2)}$ admits an Abelian structure, if the following two conditions are satisfied.

-- 1. On the open polyhedron $\N_{\circ}$ the mapping (\ref{abels})
is well defined,
which is a lift of the structured mapping (\ref{etacirc}),
i.e. the composition of the lift
$\eta_{Ab}$  with a natural mapping 
$K(\I_{a \times \aa}) \int_{\chi^{[2]}} \Z,1) \to K(\Z/2^{[2]},1)$
coincides with the structured mapping
(\ref{etacirc}).

--2. Let us consider the Moebius  band $M^2$
and represent the Eilenberg-MacLane space  $K(\I_{a \times \aa} \int_{\chi^{[2]}} \Z,1)$
as a skew-product  $K(\I_{a \times \aa},1) \tilde{\times} M^2 \to M^2$; the restriction of the bundle over
the boundary circle
$S^1 = \partial(M^2)$ is identified with the subspace  $K(\I_{a \times \aa} \times 2\Z,1) \subset K(\I_{a \times \aa} \int_{\chi^{[2]}} \Z,1)$. 
Let us include the space  $ K(\I_{a \times \aa},1) \tilde{\times} M^2$ into
$ K(\I_{a \times \aa},1) \tilde{\times} \RP^2$ by a gluing of the trivial bundle over 
the boundary $S^1$ by the trivial bundle over a small disk
with a central point $\x_{\infty} \in \RP^2$.

The 
resulting space is denoted by
$K(\I_{a \times \aa},1) \rtimes \RP^2$. The following mapping:
\begin{eqnarray}\label{abelsreg1}
\eta_{\circ;Ab}: (\N_{\circ}, U(\partial \N_{\circ}))  \to (K(\I_{a \times \aa},1) \rtimes \RP^2,K(\I_{a \times \aa},1) \times \x_{\infty} ),
\end{eqnarray}
is well defined, where the inverse image by  $\eta_{Ab}$ over a marked point 
$K(\I_{a \times \aa},1) \times \x_{\infty}   \subset K(\I_{a \times \aa},1) \rtimes \RP^2$ 
is dominated by a subpolyhedron of the dimension
$\frac{3(n-2k')}{4} = \frac{3\dim(\N_{\circ})}{4}$
(up to a small constant $d=2$), which is equipped by the structured mapping into $K(\I_{a \times \aa},1)$.
\end{definition}

\begin{remark}\label{r7}
The condition --2, Definition \ref{abstru}
can be formulated, alternatively, as following: in the polyhedron $\N_{\circ}$ there exists a subpolyhedron
$\N_{\circ;reg} \subset U(\partial \N_{\circ}) \subset \N_{\circ}$,
which is inside its proper thin neighbourhood, 
and there exists a polyhedron $P$ of the dimension (up to a small constant $d=2$)
$dim(P)=\frac{3\dim(\N_{\circ})}{4}=\frac{3(n-2k')}{4}$;
there exists a control mapping 
$\N_{\circ;reg} \to P$ and the structured mapping $\eta_P: P \to K(\I_{a \times \aa},1) \times S^1$,
such that the composition $\N_{\circ;reg} \to P \to K(\I_{a \times \aa},1) \times S^1 \to K(\I_{a \times \aa},1) \rtimes \RP^1 \subset K(\I_{a \times \aa},1) \rtimes \RP^2$ is the structured mapping
$\eta_{\circ;Ab}: \N_{\circ;reg} \to K(\I_{a \times \aa},1) \rtimes_{\chi^{[2]}} \RP^2 $ up to homotopy.
Consider the pair
$( U(\partial \N_{\circ}), U(\partial \N_{\circ}) \setminus  \N_{\circ;reg})$,
and the restriction of $\eta_{\circ;Ab}$  on this pair. The following mapping of pairs
exists:

\begin{eqnarray}\label{abelsreg}
\eta_{\circ;Ab}: (U(\partial \N_{\circ}), U(\partial \N_{\circ}) \setminus  \N_{\circ;reg}) \to (K(\I_{a \times \aa} \times 2\Z,1), K(\I_{a \times \aa},1)).
\end{eqnarray}

In the formula above $K(\I_{a \times \aa} \int_{\chi^{[2]}} 2\Z,1) = K(\I_{a \times \aa} \times 2\Z,1)$;
$K(\I_{a \times \aa},1) \subset K(\I_{a \times \aa} \int_{\chi^{[2]}} 2\Z,1)$ is the standard the inclusion.
\end{remark}

The following lemma is proved in  Section \ref{sec8}, Subsection \ref{lem7}.

\begin{lemma}\label{7}{Small Lemma}

For
\begin{eqnarray}\label{dimdim}
n-k' \equiv 1 \pmod{2}, \quad n \equiv 0 \pmod{2}
\end{eqnarray}
there exist a holonomic formal mapping
$d^{(2)}$ ($\ref{d2}$), which admits an Abelian structure, Definition 
$\ref{abstru}$.
\end{lemma}

\subsection{Mapping $d^{(2)}$ in Lemma  \ref{7}}

Let us consider the standard $2$-sheeted covering
$p: S^1 \to S^1$ over the circle. One may write-down:  $p: S^1 \to \RP^1$,
where $\Z/2 \times S^1 \to \RP^1$ is the standard antipodal action. 
Let us consider the join
of $\frac{n-k'+1}{2}=r$-copies of 
$\RP^1$ $\RP^1 \ast \dots \ast \RP^1 = S^{n-k'}$, 
which is homeomorphic to the standard 
$n-k'$-dimensional sphere. Let us define the join of $r$ copies of the mapping
$$ \tilde{d}: S^1 \ast \dots \ast S^1 \to \RP^1 \ast \dots \ast \RP^1.$$
In the preimage the group 
$\Z/2$ by the antipodal action is represented, this action is commuted with 
$\tilde{d}$. The required mapping  (\ref{dmain}) is the composition 
of the projection onto the factor
$\tilde{d}/_{\Z/2}: \RP^{n-k'} \to S^{n-k'}$ with the standard inclusion
$$ d: S^{n-k'} \subset \R^n. \qed $$

\subsubsection*{Proof of Theorem  $\ref{Th2}$}
Recall,
$k=\frac{n-n_{\sigma}}{16}$. Assume that an element  $x \in Imm^{sf}(n-k,k)$ is represented 
by a skew-framed immersion $(f,\Xi,\kappa_M)$, $f: M^{n-k}
\looparrowright \R^n$. By the assumption, there exist a compression 
$\kappa_M: M^{n-k} \to  \RP^{n-k-q-1}$, $q=\frac{n_{\sigma}}{2} = 7$, for which the composition
$M^{n-k} \to \RP^{n-k-q-1} \subset K(\I_d,1)$ coincides with
the mapping $\kappa_M: M^{n-k} \to K(\I_d,1)$. 

Define $k'$ as the biggest positive integer, which is less then
$k+q+1$, and such that $n-k' \equiv 1\pmod{2}$.
Because  $k = 1\pmod{2}$, $q=7$,
we may put  $k' = k+q+1 \ge  7$.

By Lemma $\ref{7}$ there exist an (equivariant) formal mapping  $d^{(2)}$,
which admits an Abelian structure, see Definition $\ref{abstru}$.


Let us construct a skew-framed immersion
$(f_1,\Xi_1,\kappa_1)$, for which the self-intersection manifold
is as required in Theorem
\ref{Th2}.


An immersion $f_{1}: M^{n-k} \looparrowright \R^n$, which is equipped by a skew-framing
$(\kappa_{1},\Xi_{1})$, is defined as a result of a small deformation of the composition 
$d \circ \kappa: M^{n-k} \to \RP^{n-k-q-1} \subset \RP^{n-k'} \to \R^n$ 
in the prescribed regular cobordism class of the immersion $f$. A calibre of the deformation
$d \circ \kappa \mapsto f_a$ define less then a radius of a regular neighbourhood of the polyhedron 
of intersection points of the mapping $d$. Because $f_1$ is regular homotopic to $f$, the
immersion $f_1$ is naturally equipped by a skew-framed.

Let us denote by 
$N^{n-2k}$ the self-intersection manifold of the immersion  $f_1$.
The following decomposition over a common boundary are well defined:
\begin{eqnarray}\label{Nregb}
N^{n-2k} = N^{n-2k}_{a \times \aa,\N(d^{(2)})} \cup_{\partial}
N^{n-2k}_{reg}.
\end{eqnarray}
In this formula $N^{n-2k}_{a \times \aa,\N(d^{(2)})}$ is a manifold with boundary, 
which is immersed into a regular (immersed) neighbourhood 
$U_{\N(d)}$ of the polyhedron  $\N(d)$ (with a boundary) of self-intersection points of the mapping $d$. 
The manifold
$N^{n-2k}_{reg}$ with boundary is immersed into an immersed regular neighbourhood of self-intersection points,
outside critical points; let us denote this neighbourhood by   
$U_{reg}$. A common boundary of manifolds $N^{n-2k}_{a \times \aa,\N(d^{(2)})}$, $N^{n-2k}_{reg}$
is a closed manifold of dimension $n-2k-1$, this manifold is immersed into the boundary $\partial(U_{reg})$
of the immersed neighbourhood  $U_{reg}$.

The parametrized 
$\Z/2^{[2]}$--framed immersion of the immersed submanifold  $(\ref{Nregb})$ denote by   $g_{a \times \aa}$.
Let us  prove that a
$\Z/2^{[2]}$--framing over the component  $(\ref{Nregb})$ is reduced to a framing 
with the structured group
$\I_{a \times \aa} \rtimes \Z/2$ with a control over
$\RP^2$. 

Define on the manifold 
$N^{n-2k}_{a \times \aa,\N(d^{(2)})}$ a mapping   
$$N^{n-2k}_{a \times \aa,\N(d^{(2)})}
\stackrel{\eta_{a \times \aa}}{\longrightarrow} K(\I_{a \times \aa},1) \int_{\chi^{[2]}} \rtimes \RP^2,$$
as the result of a gluing of the two mappings.

Define on the component
$N^{n-2k}_{a \times \aa,\N(d^{(2)})}$ in the formula $(\ref{Nregb})$ the following mapping:  
\begin{eqnarray}\label{etadddNd2}
\eta_{a \times \aa,\N(d^{(2)})}:
N^{n-2k}_{a \times \aa,\N(d^{(2)})} \to K(\I_{a \times \aa} \int_{\chi^{[2]}} \Z),1),
\end{eqnarray}
as the composition of the projection
$N^{n-2k}_{a \times \aa,\N(d^{(2)})} \to N(d)$ and the mapping  $\eta_{a \times \aa}: \N(d^{(2)}) \to
K((\I_{a \times \aa}) \int_{\chi^{[2]}} \Z,1)$, constructed in Lemma $\ref{7}$.

The restrictions of the mappings
$\eta_{a \times \aa,N(d^{(2)})}$, $\eta_{a \times \aa,N_{reg}}$ on a common boundary
$\partial N^{n-2k}_{a \times \aa,N(d^{(2)})}$, $\partial N^{n-2k}_{reg}$ are homotopy,
because the mapping $\eta_{\N(d^{(2)})}: \N(d^{(2)}) \to K(\I_a \times \II_a \int_{\chi^{[2]}} \Z),1)$
satisfies on $\partial \N(d^{(2)})$ right boundary conditions. 
Therefore, the required mapping
\begin{eqnarray}\label{etadtd}
\eta_{a \times \aa}: N^{n-2k}_{a \times \aa}
\to K(\I_{a \times \aa},1) \rtimes_{\chi^{[2]}}  \RP^2
\end{eqnarray}
is well defined as a result of a gluing of the mappings
$\eta_{a \times \aa,\N(d^{(2)})}$ and $\eta_{a \times \aa,N_{reg}}$.
The mapping $(\ref{etadtd})$ determines a reduction of  $\Z/2^{[2]}$--framing of the immersion
$g_{a \times \aa}$.

Let us prove that Condition $C1$, subsubsection \ref{cond11}, see Definition \ref{stand}, this is a corollary of Condition --2 in Definition \ref{abstru}. The structuring mapping on $\partial N^{n-2k}_{a \times \aa,N(d^{(2)})}$ in compressed into the skeleton of $K(\I_{a \times \aa},1)$
of the dimension $\frac{3(n-k-q-1)}{4}+d$. In the case $k=15$, $q=7$,  $\frac{3(n-k-q-1)}{4}+d
< \frac{3n}{4}$, if $d \le 7$. In Lemma \ref{7} $d\le 2$, this proves Condition --2, Definition \ref{abstru}.

By the construction, the image of the mapping $f_{2}: M^{n-8k} \looparrowright \R^n$
belongs to a regular neighbourhood of the image of the submanifold
$\RP^{n-8k-q-1} \subset \RP^{n-7k-k'}$ by the mapping  $d$.
Because $n-8k-q-1 = n- \frac{n}{2} + \frac{n_{\sigma}}{2} - \frac{n_{\sigma}}{2} -1= \frac{n}{2}-1$,
the immersed submanifold $d(\RP^{n-8k-q-1})$ is an embedded submanifold. 

Denote by $N_{2}^{n-16k}$ the manifold of self-intersection points of the immersion
$f_{2}$, this is a manifold of dimension  
$n-16k=m_{\sigma}$. The following inclusion is well defined:
\begin{eqnarray}\label{incl}
N_{2}^{n-16k} \subset N_{reg}^{n-2k},
\end{eqnarray}
where the manifold $N_{reg}^{n-2k}$ is defined by the formula  $(\ref{Nregb})$.

In particular, the following reduction of the classified mapping 
$\eta_{2}: N^{n-16k}_{2} \to K(\Z/2^{[2]},1)$ to a mapping
\begin{eqnarray}\label{red2a}
\eta_{Ab,2}: N^{n-16k}_{2} \to K(\I_{b \times \bb},1) 
\end{eqnarray}
with the image $K(\I_{b \times \bb},1) \subset K(\Z/2^{[2]},1)$ is well defined.

Therefore, without loss of a generality, for
$\sigma \ge 5$ one may assume that 
\begin{eqnarray}\label{emptyset}
N^{n-16k}_{2} \cap N^{n-2k}_{a \times \aa,\N(d)} = \emptyset. 
\end{eqnarray}


Theorem \ref{Th2} is proved. \qed

\section{Local coefficients and homology groups \label{localcoeff}}

Let us define the group
$(\I_a \times \II_a) \int_{\chi^{[2]}} \Z$ and the epimorphism 
$$(\I_a \times \II_a) \int_{\chi^{[2]}} \Z \to \Z/2^{[2]}.$$
Consider the automorphism
\begin{eqnarray}\label{chidd}
\chi^{[2]}: \I_a \times \II_a \to \I_a \times \II_a
\end{eqnarray}
of the exterior conjugation of the subgroup
$\I_a \times \II_a \subset \D$ by the element  $ba \in \D$, 
this element is represented by  the reflection of the plane with respect to the line $Lin(\e_2)$.
Let us define the automorphism (denotations are not changed)
\begin{eqnarray}\label{chi2}
\chi^{[2]}: \Z/2^{[2]} \to \Z/2^{[2]},
\end{eqnarray}
by permutations of the base vectors.
It is not difficult to check, that the inclusion
$\I_a \times \II_a \subset \Z/2^{[2]}$ commutes with automorphisms 
$(\ref{chidd})$, $(\ref{chi2})$ in the image and the preimage.

Define the group
\begin{eqnarray}\label{chi2Z}
\I_{a \times \aa} \int_{\chi^{[2]}} \Z.
\end{eqnarray}
Let us consider the quotient 
of the group $\I_{a \times \aa} \ast \Z$ (the free product of the group
$\I_{a \times \aa}$ and $\Z$) by the relation $zxz^{-1}=\chi^{[2]}(x)$, where $z \in \Z$
is the generator,  
$x \in \I_{a \times \aa}$ is an arbitrary element. 

This group is a particular example of a semi-direct product  $A \rtimes_{\phi} B$, $A=\I_{a \times \aa}$, $B=\Z$, by a homomorphism  $\phi: B \to Aut(A)$; the set $A \times B$
is equipped with a binary operation  $(a_1,b_1) \ast (a_2,b_2) \mapsto (a_1 \phi_{b_1}(b_2),b_1b_2)$.
Let us define the group
$(\ref{chi2Z})$ by this construction for  $A=\I_{a \times \aa}$, $B=\Z$,
$\phi=\chi^{[2]}$.

The classifying space  $K(\I_{a \times \aa} \int_{\chi^{[2]}} \Z,1)$ is a skew-product over the circle $S^1$ with $K(\I_{a \times \aa},1)$, where the shift mapping in the cyclic covering
$K(\I_{a \times \aa},1) \to K(\I_{a \times \aa},1)$ over $K(\I_{a \times \aa} \int_{\chi^{[2]}} \Z,1)$ is induced by the automorphism $\chi^{[2]}$.
The projection onto the circle is denoted by
\begin{eqnarray}\label{pbb}
p_{a \times \aa}: K(\I_{a \times \aa}\int_{\chi^{[2]}}\Z,1) \to S^1.
\end{eqnarray}
Take a marked point
$pt_{S^1}   \in S^1$ and define the subspace
\begin{eqnarray}\label{inclK}
K(\I_{a \times \aa},1) \subset K(\I_{a \times \aa} \int_{\chi^{[2]}} \Z,1) 
\end{eqnarray}
as the inverse image of the marked point
$pt_{S^1}$ by the mapping $(\ref{pbb})$.

A description of the standard base of the group
$H_{i}(K(\I_{a \times \aa} \int_{\chi^{[2]}} \Z,1);\Z)$
is sufficiently complicated and is not required.
The group
$H_{i}(K(\I_{a \times \aa},1);\Z)$ is described using the Kunneth formula:
\begin{eqnarray}\label{Kunnd}
\begin{array}{c}
0 \to \bigoplus_{i_1+i_2=i} H_{i_1}(K(\I_a,1);\Z) \otimes H_{i_2}(K(\II_a,1);\Z) \longrightarrow \\
\longrightarrow H_i(K(\I_{a \times \aa},1);\Z)
\longrightarrow \\
\longrightarrow \bigoplus_{i_1+i_2=i-1} Tor^{\Z}(H_{i_1}(K(\I_a,1);\Z),H_{i_2}(K(\II_a,1);\Z) \to 0.\\
\end{array}
\end{eqnarray}

The standard base of the group
$H_{i}(K(\I_{a \times \aa} \int_{\chi^{[2]}} \Z,1))$ contains the following elements: 
$$x\otimes y / (x \otimes y)-(y \otimes x),$$
where $x \in H_{j}(K(\I_a,1))$, $y \in H_{i-j}(K(\II_a,1))$ 
($\Z/2$--coefficients is the formulas are omitted).

In particular, for odd $i$ the group
$H_{i}(K(\I_{a \times \aa},1);\Z)$ 
contains elements, which are defined by the fundamental classes
of the following submanifolds:
$\RP^{i} \times pt \subset \RP^i \times \RP^i
\subset K(\I_a,1) \times K(\II_a,1) = K(\I_{a \times \aa},1)$, $pt \times \RP^{i} \subset \RP^i \times \RP^i
\subset K(\I_a,1) \times K(\II_a,1) = K(\I_{a \times \aa},1)$.
Let us denote the corresponding elements as following:
\begin{eqnarray}\label{tdtdd}
t_{a,i} \in H_{i}(K(\I_{a \times \aa},1);\Z),
t_{\aa,i} \in H_{i}(K(\I_{a \times \aa},1);\Z).
\end{eqnarray}

The following analogues of the homology groups 
$H_{i}(K(\I_{a \times \aa} \int_{\chi^{[2]}} \Z,1))$,
$H_{i}(K(\I_{a \times \aa} \int_{\chi^{[2]}} \Z,1);\Z)$
with local coefficients over the generator $\pi_1(S^1)$ 
is defined, the groups are denoted by
\begin{eqnarray}\label{Z2lok}
H_{i}^{loc}(K(\I_{a \times \aa} \int_{\chi^{[2]}} \Z,1);\Z/2),
\end{eqnarray}
\begin{eqnarray}\label{Zlok}
H_{i}^{loc}(K(\I_{a \times \aa} \int_{\chi^{[2]}} \Z,1);\Z).
\end{eqnarray}

The following epimorphism
\begin{eqnarray}\label{epid}
p_{a \times \aa}: \I_{a \times \aa} \int_{\chi^{[2]}} \Z \to \Z
\end{eqnarray}
is well defined by the formula:
$x \ast y \mapsto y$, $x \in \I_{a \times \aa}$, $y \in \Z$.
The following homomorphism 
\begin{eqnarray}\label{epid2}
\I_{a \times \aa} \int_{\chi^{[2]}} \Z \to \Z \to \Z/2,
\end{eqnarray}
is well defined by the formula
$p_{a \times \aa} \pmod{2}$. 

Let us define the group $(\ref{Zlok})$. Let us consider the group ring  $\Z[\Z/2] = \{a + bt\}$, $a,b \in \Z$, $t \in \Z/2$. The generator
$t \in \Z[\Z/2]$ is represented by the involution  
$$\chi^{[2]}: K(\I_{a \times \aa} \times 2\Z,1) \to
K(\I_{a \times \aa} \times 2\Z,1),$$ 
the restriction of the involution
on the subspace 
$K(\I_{a \times \aa},1) \subset K(\I_{a \times \aa} \times 2\Z,1)$ 
is the reflection, which is induced by the automorphism  
$\I_a \times \II_a \to \I_a \times \II_a$, which  permutes factors;
the restriction of the involution on the group $2\Z$ is the antipodal involution in the double covering $S^1 \stackrel{\times 2}{\longrightarrow} S^1$.
Because all non-trivial homology classes of the space
$K(\I_{a \times \aa},1)$ are of the order $2$, transformation of signs 
is not required. Nevertheless,  a local orientation of transformed simplex  is preserved.

Let us consider the local system of the coefficient
$\rho_t: \Z/2[\Z/2] \to Aut(K(\I_{a \times \aa},1) \times 2\Z)$,
using this local system: a chain
$(a + bt)\sigma$ with the support on a simplex $\sigma \subset K(\I_{a \times \aa} \times 2\Z,1)$ is transformed into a chain $(at + b)\chi^{[2]}(\sigma)$.
The group  $(\ref{Zlok})$ is well defined
by the tensor product of a chain $\Z[\Z/2]$-complex with the integers $\Z$ by the augmentation
$\Z[\Z/2] \to \Z$: $a+bt \mapsto a+b$.

 The group  $(\ref{Z2lok})$ is defined analogously.
A complete calculation of the groups $(\ref{Z2lok})$, $(\ref{Zlok})$ is not required.

Let us define the following subgroup:
\begin{eqnarray}\label{DZlok}
D_{i}^{loc}(\I_{a \times \aa};\Z) \subset H_{i}^{loc}(K(\I_{a \times \aa} \int_{\chi^{[2]}} \Z,1);\Z)
\end{eqnarray}
by the formula:
$D_{i}^{loc}(\I_{a \times \aa};\Z) = \mathrm{Im}(H_i^{loc}(K(\I_{a \times \aa},1);\Z) \to
H_{i}^{loc}(K(\I_{a \times \aa} \int_{\chi^{[2]}} \Z,1);\Z))$, 
where the homomorphism
$$H_i^{loc}(K(\I_{a \times \aa},1);\Z) \to
H_{i}^{loc}(K(\I_{a \times \aa} \int_{\chi^{[2]}} \Z,1);\Z)$$
is induced by the inclusion of the subgroup. 

From the definition, there exist a natural homomorphism:
\begin{eqnarray}\label{epilok}
H_{i}(K(\I_{a \times \aa},1);\Z[\Z/2]) \to D_{i}^{loc}(\I_{a \times \aa};\Z),
\end{eqnarray}
this homomorphism is an isomorphism.



A subgroup
\begin{eqnarray}\label{DZ2lok}
D_{i}^{loc}(\I_{a \times \aa};\Z/2) \subset  H_{i}^{loc}(K(\I_{a \times \aa} \int_{\chi^{[2]}} \Z,1);\Z/2)
\end{eqnarray}
is defined analogously as $(\ref{DZlok})$.
A description of this subgroup is more easy, because this group is generated by elements
$X + tY$,
$X = x \otimes y, Y = x' \otimes y'$, wherein $\chi^{[2]}_{\ast}(x \otimes y) = y \otimes x$.

In the case $i=2s$ the basis elements 
$y \in D_{i-1}^{loc}(\I_{a \times \aa};\Z)$ are the following:

1. $y=r$,  $r= t_{a,2s-1} +  t_{\aa,2s-1}$, where $t_{a,2s-1},t_{\aa,2s-1} \in H_{2s-1}(K(\I_{a \times \aa},1);\Z)$ are defined by the formula
$(\ref{tdtdd})$.

2. $y= z(i_1,i_2)$, where $z(i_1,i_2) = tor(r_{a,i_1}, r_{\aa,i_2}) + tor(r_{a,i_2}, r_{\aa,i_1})$, $i_1,i_2 \equiv 1 \pmod{2}$,
$i_1 + i_2 = 2s-2$,
$tor(r_{a,i_1}, r_{\aa,i_2}) \in Tor^{\Z}(H_{i_1}(K(\I_{a},1);\Z),H_{i_2}(K(\II_a,1);\Z))$.
The elements  $\{z(i_1,i_2)\}$ belong to the kernel of the homomorphism: 
\begin{eqnarray}\label{coeff}
D_{i-1}^{loc}(\I_{a \times \aa};\Z) \to D_{i-1}^{loc}(\I_{a \times \aa};\Z/2),
\end{eqnarray}
which is defined as the modulo 2 reduction of the coefficients:
$\Z \to \Z/2$.



\begin{lemma}\label{obstr}
The group $H_{2s}^{loc}(K(\I_{a \times \aa}\int_{\chi^{[2]}}\Z,1);\Z)$ contains a direct factor by the subgroup
$ D_{2s}^{loc}(\I_{a \times \aa};\Z)$. 
The elements from the subgroup  
$$\bigoplus_{i_1+i_2=2s} H_{i_1}(K(\I_{a},1);\Z) \otimes H_{i_2}(K(\II_a,1);\Z) \subset D_{2s}^{loc}(\I_a \times \II_a;\Z),$$ 
is mapped monomorphic into the image   $\mathrm{Im}(A)$  by the following homomorphism:

\begin{eqnarray}\label{A}
A: H_{2s}^{loc}(K(\I_{a \times \bb}\int_{\chi^{[2]}}\Z,1);\Z) \to 
\end{eqnarray}
$$H_{2s}^{loc}(K(\I_{a \times \bb}\int_{\chi^{[2]}}\Z,1);\Z/2),$$
which is the modulo 2 reduction of the coefficient system.
\end{lemma}

\subsubsection*{Proof of Lemma $\ref{obstr}$}

Assume an element $ x \otimes y \in H_{i_1}(K(\I_{a},1);\Z) \otimes H_{i_2}(K(\II_a,1);\Z)$
is a cycle in $H_{2s}^{loc}(K(\I_{a \times \aa}\int_{\chi^{[2]}}\Z,1);\Z)$, which is represented by a singular manifold 
$f: V^{2s} \longrightarrow K(\I_{a},1) \times K(\II_a,1)$. Take the boundary
$W^{2s+1}$, $\partial W^{2s+1}=V^{2s}$ and consider the cycle
$$F: (W^{2s+1},V^{2s}) \longrightarrow (K(\I_{a \times \aa}\int_{\chi^{[2]}}\Z,1),K(\I_{a},1) \times K(\II_a,1))$$
with a local coefficient system, which is given by the collection of a paths from centres of simplexes to the marked point in the subspace of the pair. We may assume, that the singular cycle
$F$ is such that the manifold is catted by the hypersubmanifold $U^{2s} \subset W^{2s+1}$, the
local path system at the  two components of the boundary  of
the catted submanifold $W^{2s-1} \setminus U^{2s}$ have no intersection points of paths with $U^{2s}$.
Then, the two copies $U_+^{2s}$, $U^{2s}_-$ of the boundary are represented by $\chi^{(2)}$-conjugated cycles and after the homotopy of $U^{2s} \subset W^{2s+1}$ into a subspace
$K(\I_{a},1) \times K(\II_a,1) \subset K(\I_{a \times \aa}\int_{\chi^{[2]}}\Z,1)$ we get
an ordinary singular cycle with 3 components of boundary: $V^{2s}$, $U^{2s}$ and $-U^{2s}$,
where $-\dots$ is denoted the cycle with the opposite orientation. Lemma \ref{obstr} is proved. \qed
\[  \]

Let us consider the diagonal subgroup:
$i_{\I_d, \I_{a \times \aa}}: \I_d \subset \I_{a} \times \II_a = \I_{a \times \aa}$.
This subgroup coincides to the kernel of the homomorphism
\begin{eqnarray}\label{omega2}
\omega^{[2]}: \I_a \times \II_a \to \Z/2,
\end{eqnarray}
which is defined by the formula
$(x \times y) \mapsto xy$.

Let us define the homomorphism
\begin{eqnarray}\label{Phi2}
\Phi^{[2]}: \I_{a \times \aa} \int_{\chi^{[2]}} \Z \to \Z/2^{[2]}
\end{eqnarray}
by the formula:
$\Phi^{[2]}(z)=ab$, $z \in \Z$ is the generator  (the element $ab$ is represented by the inversion of the first basis vector $\Z^{[2]} \subset O(2)$), 
$$\Phi^{[2]} \vert_{\I_{a \times \aa} \times \{0\}}: \I_{a \times \aa} \subset \D$$ 
is the standard inclusion;
$$\Phi^{[2]} \vert_{z^{-1}\I_a \times \II_a  z}: \I_{a \times \aa} \subset \D$$ 
is the conjugated inclusion, by the composition with the exterior automorphism
in the subgroup $\I_{a \times \aa} \subset \Z/2^{[2]}$.

Let us define
\begin{eqnarray}\label{arf}
(\Phi^{[2]})^{\ast}(\tau_{[2]}) = \tau_{a \times \aa},
\end{eqnarray}
where
$\tau_{a \times \aa} \in H^2(K(\I_{a \times \aa} \int_{\chi^{[2]}} \Z,1))$, $\tau_{[2]} \in H^2(K(\Z^{[2]},1))$.

\section{The fundamental class of the canonical covering over $D$-framed standardized immersion}

Let us consider 
a $\Z/2^{[2]}$--framed immersion  $(g,\Psi,\eta_N)$, $g: N^{n-2k} \looparrowright \R^n$, 
and assume that this immersion is standardized by $\Upsilon$, see subsubsection \ref{dihedstand}.

The mapping  $\eta_N$ admits a reduction, the image of this mapping belongs to the total bundle space over $\RP^2$ with the layer
$K(\I_{a \times \aa},1)$. Assume that the manifold
$N^{n-2k}$ is connected. Assume that a marked point on $\RP^2$ is fixed and denote:
\begin{eqnarray}\label{pro}
PROJ^{-1}(\x_{\infty}) \in N^{n-2k}.
\end{eqnarray}
 Assume that the immersion  $g$ 
translates a marked point into a small neighbourhood of the central point $\x_{\infty} \in \RP^2$,
over this point the defect is well defined.

The image of the fundamental class
$\eta_{N,\ast}([N^{n-2k}])$ belongs to $H_{n-2k}(\D;\Z)$. The standardization condition, see Definition \ref{stand}, 
implies an existence of a natural lift of this homology class
into the group $H_{n-2k}(\I_{a \times \aa} \int_{\chi^{[2]}} \Z;\Z)$, 
because the condition C1 implies that the submanifold
$N_{sing} \subset N^{n-2k}$, which is a layer over the point  $\x_{\infty} \in \RP^2$
determines zero-chain.

Let us define the lift of the characteristic class
$\eta_{N,\ast}$ 
and define the class 
\begin{eqnarray}\label{class4}
\eta_{N,\ast}([N^{n-2k}]) \in H_{n-2k}^{loc}((\I_{a \times \aa} \rtimes_{\chi^{[2]}} \Z,1);\Z).
\end{eqnarray}
Then let us prove that

\begin{theorem}\label{cla}
The characteristic class 	(\ref{class4}) satisfies the following equation:
\begin{eqnarray}\label{class5}
\eta_{N,\ast}([N^{n-2k}]) \in D_{n-2k}^{loc}(\I_{a \times \aa};\Z),
\end{eqnarray}
 see the inclusion (\ref{DZlok}). 
\end{theorem}





\subsubsection{Construction of the class (\ref{class4})}
Let us consider a skeleton of the space
$K(\I_{a \times \aa},1)$, which is realized as the structured group 
$\I_{a \times \aa} \subset O(1+1)$ of a bundle 
over the corresponding Grasmann manifold
$Gr_{\I_{a \times \aa}}(1+1,n)$ of non-oriented $1+1=2$-plans in  $n$--space. 
Denote this Grasmann manifold by
\begin{eqnarray}\label{KK1}
KK(\I_{a \times \aa}) \cong K(\I_{a \times \aa},1).
\end{eqnarray}
On the space $(\ref{KK1})$ a free involution 
\begin{eqnarray}\label{chiK}
\chi^{[2]}: KK(\I_{a \times \aa}) \to KK(\I_{a \times \aa}),
\end{eqnarray}
acts, this involution corresponds to the automorphism
$(\ref{chidd})$ and represents by a permutation of line of layers (by the antidiagonal matrix). 

Let us define a family of spaces
(\ref{KK1}) with the prescribed symmetry group, are parametrized over the base $S^1$, 
the generator of $S^1$ transforms layers by the involution (\ref{chidd}). 
This space is a subspace (a skeleton):
$K(\I_{a \times \aa},1) \rtimes_{\chi^{[2]}} \RP^2 \subset K(\D,1)$.

Define the required space
\begin{eqnarray}\label{KKin}
KK(\I_{a \times \aa} \rtimes_{\chi^{[2]}} \RP^2)
\end{eqnarray}
by a quotient of the direct product
$KK(\I_{a \times\aa}) \times S^2$ with respect to the involution $\chi^{[2]} \times (-1)$, 
where $-1: S^2 \to S^2$--is the antipodal involution. 
The following inclusion is well defined:
\begin{eqnarray}\label{KKin2}
KK(\I_{a \times \aa} \int_{\chi^{[2]}} \Z) \subset KK(\I_{a \times \aa}) \rtimes_{\chi^{[2]}} \RP^2.
\end{eqnarray}

The involution $(\ref{chiK})$ induces the $S^1$-fibrewise involution 
\begin{eqnarray}\label{chiKint}
\chi^{[2]}: KK(\I_{a \times \aa}\int_{\chi^{[2]}} \Z ) \to KK(\I_{a \times \aa}\int_{\chi^{[2]}} \Z),
\end{eqnarray}
which is extended to the involution
\begin{eqnarray}\label{chiKin}
\chi^{[2]}: KK(\I_{a \times \aa} \rtimes_{\chi^{[2]}} \RP^2 ) \to
KK(\I_{a \times \aa} \rtimes_{\chi^{[2]}} \RP^2)
\end{eqnarray}
on  $KK(\I_{a \times \aa}) \rtimes \RP^2$, 
all this extended involutions denote the same.

The universal $2$-dimensional
$\I_{a \times \aa} \int_{\chi^{[2]}} \Z/2$--bundle
over the target space in (\ref{KKin2})
$KK(\I_{a \times \aa},1) \rtimes_{\chi^{[2]}} \RP^2$,  which is denoted by 
$\tau_{a \times \aa, \rtimes}$,
is well defined.
We will need mostly the restriction of this bundle to the subspace (\ref{KKin}).
The restriction of this bundle  on the origin
subspace of $(\ref{KK1})$ is the involutive pull-back  of the universal $2$-bundle $\tau_{a \times \aa}$ over the target space of $(\ref{KK1})$.

The first direct factor of the $\I_{a \times \aa} \rtimes \Z$-framing of the immersion $(g,\Psi,\eta_N)$ by $\Upsilon$,  determines the Gaussean mapping in (\ref{KKin}). The restriction of this mapping on (\ref{pro}) is in the $\frac{3n}{4}$-skeleton of the subspace (\ref{KK1}). The restriction of this mapping on the complement to this submanifold  is 
inside the subspace (\ref{KKin2}). Local coefficients system, described in Section  \label{localcoeff} determines the class (\ref{class4}).

\subsubsection{Proof of Theorem \ref{cla}}
Let us prove that the class (\ref{class4}) satisfies the equation (\ref{class5}).
Take the decomposition (\ref{Nregb}). The first manifold in the right hand-side of the formula is mapped into the skeleton of the dimension $n-2k-2q$, where $q=7$. Therefore the fundamental class of the second term gives a contribution to the characteristic class (\ref{class4}). By construction, the formula (\ref{class5}) is satisfied. Theorem \ref{cla} is proved. \qed

In the case $dim(N)=n-2k$, we get $dim(\bar{N}^{n_{\sigma}}_{a \times \aa})=n_{\sigma}=n-16k$, see (\ref{Nsigma}), Theorem \ref{Th2};
recall, in the case $n=254$ one gets $k=15$  and $n_{\sigma}=14$.

Let us consider the  element
\begin{eqnarray}\label{etadlok}
\eta_{\ast}(N^{n_{\sigma}}_{a \times \aa}) \in D_{n_{\sigma}}^{loc}(\I_{a \times \aa};\Z).
\end{eqnarray}
Let us prove that the element (\ref{etadlok}) satisfies an additional property, which is called pure standardized Property, Definition \ref{negl} below.

\begin{lemma}\label{dd}

--1. The element $(\ref{etadlok})$ 
is decomposed  with respect to the standard base of the group
$H_{n_{\sigma}}(K(\I_a \times \II_a,1))$ contains not more then the only monomial
$t_{a,i} \otimes t_{\aa,i}$, see. $(\ref{tdtdd})$, $i=\frac{n_{\sigma}}{2}=\frac{n-16k}{2}$.
The coefficient of this monomial coincides with the Kervaire invariant, which is calculated for 
a $\Z/2^{[2]}$--framed immersion $(g,\Psi,\eta_N)$.

\end{lemma}

\subsubsection*{Proof of Lemma $\ref{dd}$}
Let us proof the statement for the case
$n_{\sigma}=14$. Let us consider the manifold $N_{a \times \aa}^{n_{\sigma}}$, 
which we re-denote in the proof by
$N^{14}$ for sort. This manifold is equipped with the mapping 
$\eta: N^{14} \to K(\I_{a \times \aa},1)$. 
Let us consider all the collection of characteristic $\Z/2$-numbers
for the mapping $\eta$, which is induced from the universal classes.
The Hurewitcz image $\eta_{\ast}([N^{14}])$ is in
the group $D_{n_{\sigma}}^{loc}(\I_{a \times \aa};\Z/2)$.

Let us estimate the fundamental class
$[N^{14}]$ and prove that this class is pure (see definition \ref{negl}) below. Because 
$N^{14}$ is oriented, among characteristic numbers: 
$\kappa_a\kappa^{13}_{\aa}$, $\kappa_a^3\kappa^{11}_{\aa}$, $\kappa_a^5\kappa^{9}_{\aa}$,  $\kappa_a^7\kappa^{7}_{\aa}$, $\kappa_a^9\kappa^{5}_{\aa}$,
$\kappa_a^3\kappa^{11}_{\aa}$, $\kappa_a\kappa^{13}_{\aa}$.
the only number $\kappa_a^7\kappa^{7}_{\aa}$ could be nontrivial.

Let us prove that the characteristic number 
$\kappa_a\kappa^{13}_{\aa}$ is trivial. Let us consider the manifold
$K^3 \subset N^{14}$, which is dual to the characteristic class  $\kappa_a\kappa_{\aa}^{10}$. The normal bundle
$\nu_N$ over the manifold $N^{14}$ is isomorphic to the Whitney sum $k\kappa_a \oplus k\kappa_{\aa}$, where the positive integer $k$ satisfies the equation:
$k \equiv 0 \pmod{8}$.
The restriction of the bundle
$\nu_N$ on the submanifold $K^3$ is trivial.  
Therefore, the normal bundle
$\nu_K$  of the manifold  $K^3$ is stably isomorphic to the bundle 
$\kappa_a \oplus 2\kappa_{\aa}$.

Because the characteristic class
$w_2(K^3)$ is trivial we get:
$\langle \kappa^3_{\aa};[K^3]\rangle=0$. But,  $\langle \kappa^3_{\aa};[K^3]\rangle=
\langle \kappa_a\kappa^{13}_{\aa};[N^{14}]\rangle$.
This proves that the characteristic number
$\kappa_a\kappa^{13}_{\aa}$ is trivial. 
Analogously, the characteristic numbers
$\kappa_a^5\kappa^{9}_{\aa}$, $\kappa_a^9\kappa^{5}_{\aa}$, $\kappa_a\kappa^{13}_{\aa}$ are trivial.

Let us prove that the characteristic number
$\kappa_a^3\kappa^{11}_{\aa}$ is trivial. Let us consider the submanifold
$K^7 \subset N^{14}$, which is dual to the characteristic class 
$\kappa_a^2\kappa_{\aa}^5$. The normal bundle $\nu_K$ of the manifold $K^6$
is stably isomorphic to the Whitney sum 
$2\kappa_a \oplus 5\kappa_{\aa}$. Because $w_5(K^7)=0$,
the characteristic class
$\kappa_{\aa}^{5}$ is trivial. In particular, 
$\langle \kappa_a\kappa_{\aa}^6;[K^7]\rangle=0$.
The following equation is satisfied: 
$\langle \kappa_a\kappa_{\aa}^6;[K^7]\rangle = \langle \kappa_a^3\kappa_{\aa}^{11};[N^{14}]\rangle$.

It is proved that the characteristic number
$\kappa_a^3\kappa^{11}_{\aa}$ is trivial. Analogously, 
the characteristic number
$\kappa_a^{11}\kappa^{3}_{\aa}$ is trivial. 

It is sufficiently to note, that the characteristic number
$\langle \kappa_a^7\kappa_{\aa}^7;[N^{14}] \rangle$ coincides with the
characteristic number (\ref{44}) in the lemma. 
Lemma $\ref{dd}$ is proved. \qed
\[  \]

\begin{definition}\label{negl}
Let $(g,\eta_N,\Psi;\Upsilon)$ be the standardized $\D$-framed immersion in the codimension  $2k$.
Let us say that this standardized immersion is negligible, if its fundamental class
$\pmod{2}$ is in the subgroup
\begin{eqnarray}\label{etad}
\eta_{\ast}(N^{n-2k}_{a \times \aa}) \in D_{n - 2k}^{loc}(\I_{a \times \aa};\Z/2)
\end{eqnarray}
and is trivial.

Let us say that this standardized immersion is pure, if its fundamental class
$\pmod{2}$ satisfies the condition of Lemma
\ref{dd}, i.e. the Hurewitcz image of the fundamental class is in the group 
$\in D_{n - 2k}^{loc}(\I_{a \times \aa};\Z/2)$ and contains no monomials 
$t_{a,\frac{n-2k}{2}+i} \otimes t_{\aa,\frac{n-2k}{2}-i}$, $i \in \{\pm 1;\pm 2;\pm 3;\pm 4;\pm 5;\pm 6;\pm 7\}$, but, probably, contains the monomial
$t_{a,\frac{n-2k}{2}} \otimes t_{\aa,\frac{n-2k}{2}}$.
\end{definition}

We may reformulate results by Lemma \ref{dd} using Definition \ref{negl}. 

\begin{proposition}\label{pr1}
Assume  an element  $[(g,\eta_N,\Psi)] \in Imm^{D}(n-2k,2k)$ is in the image of the vertical
homomorphism in the diagram $(\ref{5a})$. Then this element is represented by a standardized immersion, which is pure in the sense of Definition $\ref{negl}$.
\end{proposition}

\section{$\H_{a \times \aa}$--structure on $\Z/2^{[3]}$--framed immersion \newline
$\J_b \times \JJ_b$--structure on
$\Z/2^{[4]}$--framed immersion}\label{Sec5}

The group
$\I_b$ is defined as the cyclic subgroup of the order
$4$ in the dihedral group:
$\I_b \subset \Z/2^{[2]}$.

Let us define the subgroup
\begin{eqnarray} \label{iaa}
i_{\tilde{\J}_b \times \tilde{\JJ}_b}: \tilde{\J}_b \times \tilde{\JJ}_b \subset
\Z/2^{[4]},
\end{eqnarray}
each group $\tilde{\J}_b$, $\tilde{\JJ}_b$  is the order $8$ group, which is isomorphic to the Cartesian product of the two cyclic groups $\Z/4 \times \Z/2$.

The subgroup (\ref{iaa}) admits a subgroup with the splitted projection onto: 
\begin{eqnarray} \label{iaa}
 \tilde{\J}_b \times \tilde{\JJ}_b \supset \J_b \times \JJ_b,
\end{eqnarray}
each group $\J_b$, $\JJ_b$  is the cyclic $4$ group.

The group $\Z/2^{[4]}$ (the monodromy group for the $3$-tuple iterated self-intersection 
of a skew-framed immersion) is defined using the base  $(\e_1, \dots, \e_4, \f_1, \dots, \f_4)$ in
the Euclidean space $\R^8$. Let us denote the order $4$ generators of the subgroup
$\J_b \times
\JJ_b$ by $b$, $\bb$ correspondingly. 
 The generators of the order $2$ in the extended group $\tilde{\J}_b \times \tilde{\JJ}_b$ are denoted by $d$, $\dd$ correspondingly.

Let us describe transformations
in
$\Z/2^{[4]}$, which corresponds to each generator. 
Let us consider two subspaces described by the base vectors as following:
\begin{eqnarray}\label{ff}
Lin_{\x} = \{\e_1, \dots, \e_4\}, \quad Lin_{\y} = \{\f_1, \dots, \f_4\}.
\end{eqnarray}

The linear space $Lin_{\x} \oplus Lin_{\y}$ is equipped with the $\Z/4 \times \Z/2$ action by the generators $(b,d)$:
$$b(\e_1)=\e_2; \ b(\e_2)=-\e_1, \ b(\e_3)=\e_4; \ b(\e_4)=-\e_3;$$
$$b(\f_1)=\f_3; \ b(\f_2)=\f_4, \ b(\f_3)=\f_1; \ b(\f_4)=-\f_2;$$
$$d(\e_1)=-\e_1; \ d(\e_2)=\e_2, \ d(\e_3)=-\e_3; \ d(\e_4)=\e_4;$$
$$d(\f_1)=\f_1; \ d(\f_2)=\f_2, \ d(\f_3)=\f_3; \ d(\f_4)=\f_4.$$

The linear subspace $Lin_{\x} \oplus Lin_{\y}$ is equipped with the $\Z/4 \times \Z/2$ action by the generators $(\bb,\dd)$:
$$\bb(\e_1)=\f_3; \ \bb(\e_3)=\f_1, \ \bb(\e_2)=\f_4; \ \bb(\e_4)=\f_2;$$
$$\bb(\f_1)=\f_2; \ \bb(\f_2)=-\f_1, \ \bb(\f_3)=\f_4; \ \bb(\f_4)=-\f_3;$$
$$\dd(\e_1)=\f_1; \ \dd(\e_2)=\e_2, \ \dd(\e_3)=\e_3; \ \dd(\e_4)=\e_4;$$
$$\dd(\f_1)=-\f_1; \ \dd(\f_2)=\f_2, \ \dd(\f_3)=-\f_3; \ \dd(\f_4)=\f_4.$$

 Denote by $\chi^{[4]}: Lin_{\x} \to Lin_{\y}$ the involution by 
 $$\chi^{[4]}(\e_1)=\f_1; \ \chi^{[4]}(\e_2)=\f_2; \ \chi^{[4]}(\e_3)=\f_3; \ \chi^{[4]}(\e_4)=\f_4.$$
 The automorphism is defined the blocks involution:
 $$ \chi^{[4]}: Lin_{\x} \oplus Lin_{\y} \longrightarrow Lin_{\x} \oplus Lin_{\y}. $$
The involution $\chi^{[4]}$ satisfies the following property:
$$\chi^{[4]} \circ  b = \bb \circ \chi^{[4]}, \quad \chi^{[4]} \circ  d = \dd \circ \chi^{[4]}. $$

Let us denote the subgroup
$i_{\H_{a \times \aa}, \J_b \times \JJ_b}: \H_{a \times \aa} \subset \J_b \times \JJ_b$,
which is the product of the diagonal subgroup
$\Z/4 \subset \J_b \times \JJ_b$ with the elementary subgroup $\Z/2$
of the second factor
$\Z/2 \subset \JJ_b$ (or, equivalently, with the elementary subgroup $\Z/2 \subset \JJ_b$
of the first factor, because the two products coincide).

The following homomorphism 
\begin{eqnarray}\label{omega4}
\omega^{[4]}: \J_b \times \JJ_b \to \Z/4,
\end{eqnarray}
is defined by the formula:
$(x \times y) \mapsto xy$.


Let us consider the  inclusion
$\Z/2^{[3]} \subset \Z/2^{[4]}$, which is defined by the subspace 
$Lin(\e_1, \e_2, \f_1, \f_2) \subset \{\e_1, \dots \f_4\}$. 
The inclusion  $i_{\H_{a \times \aa}}: \H_{a \times \aa} \subset
\Z/2^{[3]}$ is well defined, such that
the following diagram is commutative:

\begin{eqnarray}\label{D1}
\label{a,aa}
\begin{array}{ccc}
\qquad \I_{a \times \aa} & \stackrel {i_{a \times \aa}} {\longrightarrow}& \qquad \Z/2^{[2]} \\
i_{a \times \aa,\H_{a \times \aa}} \downarrow &  & i^{[3]} \downarrow \\
\qquad \H_{a \times \aa} &  \stackrel {i_{\H_{a \times \aa}}}
{\longrightarrow}& \qquad \Z/2^{[3] } \\
i_{\H_{a \times \aa},\J_b \times \JJ_b} \downarrow &  &  i^{[4]} \downarrow \\
\qquad \J_b \times \JJ_b & \stackrel{i_{\J_b \times \JJ_b}}{\longrightarrow} &  \qquad \Z/2^{[4]}.\\
\end{array}
\end{eqnarray}

The following automorphisms of the order 2:
\begin{eqnarray}\label{chi3H}
\chi^{[3]}: \H_{a \times \aa} \to \H_{a \times \aa},
\end{eqnarray}
\begin{eqnarray}\label{chi3A}
\chi^{[4]}: \J_b \times \JJ_b \to \J_b \times \JJ_b,
\end{eqnarray}
and
\begin{eqnarray}\label{chi3ZA}
\chi^{[3]}: \Z/2^{[3]} \to \Z/2^{[3]},
\end{eqnarray}
\begin{eqnarray}\label{chi4ZA}
\chi^{[4]}: \Z/2^{[4]} \to \Z/2^{[4]},
\end{eqnarray}
 are marked with a loss of the strictness.

The following triple
of $\Z$-extensions of the group $\J_b \times \JJ_b$, defined below, is required.
The group
$\I_{a \times \aa} \int_{\chi^{[2]}} \Z$ was defined above by the formula  $(\ref{chi2Z})$.
Analogously, the groups 
\begin{eqnarray}\label{chi3Z}
\H_{a \times \aa} \int_{\chi^{[3]}} \Z,
\end{eqnarray}
\begin{eqnarray}\label{chi4Z}
(\J_b \times \JJ_b) \int_{\chi^{[4]}} \Z,
\end{eqnarray}
are defined as semi-direct products of the corresponding groups, equipped by automorphisms, with the group $\Z$.

The classifying space $K(\H_{a \times \aa} \int_{\chi^{[3]}} \Z,1)$
is a skew-product of the circle
$S^1$ with the space $K(\H_{a \times \aa},1)$, moreover, the mapping 
$K(\H_{a \times \aa},1) \to K(\H_{a \times \aa},1)$, 
which corresponds to a shift of the cyclic covering over
$K(\H_{a \times \aa} \int_{\chi^{[3]}} \Z,1)$, is defined by the involution,
which is induced by the automorphism
$\chi^{[3]}$. The definition of the group $(\ref{chi4Z})$ is totally analogous. 

\subsubsection{Laurent extensions}
For the second and third lines of the diagram (\ref{D1}) let us define the primary
and secondary Laurent extensions, which are self-conjugated by 
the automorphism $\chi^{[3]}$ and $\chi^{[4]}$ correspondingly. Let us denote the primary  extension on the group $\H_{a \times \aa}$ by the following bi-involutive automorphism
\begin{eqnarray}\label{mu3}
\mu^{(3)}_{b \times \bb}=\mu_{b}^{(3)} \times \mu_{\bb}^{(3)}: \H_{a \times \aa} \to \H_{a \times \aa},
\end{eqnarray}
 which inverses the cyclic generator and is given by the complex conjugation over the coordinates. 
\begin{eqnarray}\label{first}
	\H_{a \times \aa} \rtimes_{\mu_{b \times \bb}^{(3)}} \Z \times \dot{\Z}.
\end{eqnarray}
The authomorphism (\ref{mu3}) is defined by the formula:
$$ \mu_{b}^{(3)}: \e_1 \mapsto -\e_1, \e_2 \mapsto \e_2; \qquad 
\mu_{\bb}^{(3)}: \f_1 \mapsto -\f_1, \f_2 \mapsto \f_2. $$

Let us denote the secondary extension with the analogous property by
\begin{eqnarray}\label{mu4}
\mu_{b \times \bb}^{(4)} = \mu_{b}^{(4)} \times \mu_{\bb}^{(4)}: \J_b \times \JJ_b \to \J_b \times \JJ_b
\end{eqnarray}
The authomorphism (\ref{mu4}) is defined by the formula:
$$ \mu_{b}^{(4)}: \e_1 \mapsto -\e_1, \e_2 \mapsto -\e_2, \e_3 \mapsto \e_3, \e_4 \mapsto \e_4$$  
$$\mu_{\bb}^{(4)}:  \f_1 \mapsto -\f_1, \f_2 \mapsto -\f_2, \f_3 \mapsto \f_3, \f_4 \mapsto \f_4. $$
One may see that $\mu_{b \times \bb}^{(3)}$ is the transfer of $\mu_{b \times \bb}^{(4)}$ by the
subgroup $\Z/2^{[3]} \subset \Z/2^{[4]}$.

 The Laurent expansion  on the group $\J_b \times \JJ_b$ 
is determined by the automorphisms
 $\mu_b^{(4)}: \J_b \to \J_b$,  $\mu_{\bb}^{(4)}: \JJ_b \to \JJ_b$. The  Laurent extension itself is denoted by
\begin{eqnarray}\label{second}
\J_b \times \JJ_b \rtimes_{\mu^{(4)}_{b \times \bb}} (\Z \times \dot{\Z}).
\end{eqnarray}

The primary extension is $\chi^{[3]}$-equivariant, we get the total primary extension:
\begin{eqnarray}\label{firsttot}
	\H_{a \times \aa} \rtimes_{\mu^{(3)}_{b \times \bb}} \Z \times \dot{\Z} \rtimes_{\chi^{[3]}} \Z.
\end{eqnarray}

The secondary extension is $\chi^{[4]}$-equivariant, we get the total secondary extension: 
\begin{eqnarray}\label{secondtot}
\J_b \times \JJ_b \rtimes_{\mu_{b \times \bb}^{(4)}} (\Z \times \dot{\Z}) \rtimes_{\chi^{[4]}} \Z. 
\end{eqnarray}
The extension (\ref{firsttot}) is the transfer of (\ref{secondtot}) by the
subgroup $\Z/2^{[3]} \subset \Z/2^{[4]}$.

The Laurent extensions (\ref{firsttot}), and  (\ref{secondtot})  are naturally represented into $\Z/2^{[4]}$. The diagram (\ref{D1}) is a subdiagram in the following diagram:

\begin{eqnarray}\label{D2}
\label{a,aa}
\begin{array}{ccc}
\qquad \I_{a \times \aa} \rtimes_{\chi^{[2]}} \Z & \stackrel {i_{a \times \aa}} {\longrightarrow}& \qquad \Z/2^{[2]} \\
i_{a \times \aa,\H_{a \times \aa}} \downarrow &  & i^{[3]} \downarrow \\
\qquad \H_{a \times \aa} \rtimes_{\mu^{(3)};\chi^{[3]}} (\Z \times \Z) \rtimes \Z &  \stackrel {i_{\H_{a \times \aa}}}
{\longrightarrow}& \qquad \Z/2^{[3] } \\
i_{\H_{a \times \aa},\J_b \times \JJ_b}  \downarrow &  &  i^{[4]} \downarrow \\
\qquad \J_b \times \JJ_b \rtimes_{\mu^{(4)}_{b \times \bb};\chi^{[4]}}  (\Z \times \dot{\Z}) \rtimes \Z& \stackrel{i_{\J_b \times \JJ_b}}{\longrightarrow} &  \qquad \Z/2^{[4]}.\\
\end{array}
\end{eqnarray}

The primary resolution representation in the middle row of the diagram (\ref{D2}) is following. 
The primary generators of $\Z$ (of $\dot{\Z}$) is represented into $Lin_x(\e_1,\e_2)$ (into $Lin_y(\f_1,\f_2)$) by the element $A$, the reflection of the vector $\e_1$ (of the vector $\f_1$).
The secondary representation in  the third row of the diagram (\ref{D2}) extends the primary representation. By this  representation the generator of  $\Z$ (of $\dot{\Z}$) is represented into 
$Lin(\e_1,\e_4)$ (into $Lin(\f_1,\f_4)$) by the antipodal reflection of the base vectors   $\e_1$,  $\e_3$ ($\f_1$,  $\f_3$) correspondingly.
Recall, there is a defect of the integer extensions
(\ref{chi3Z}), (\ref{chi4Z}) the $\chi$-parametrization is not over $S^1$ 
but also over the projective plane $\RP^2$. 
At a final step of the construction the defect of (\ref{secondtot}) is eliminated and only the integer  Laurent extension is required. 

\subsubsection{Complex line bundle $\beta_{b \times \bb}$\label{complex}}

Let us define the following $\C$-fibred  bundle over the classifying space of the left group
of the bottom map in the diagram (\ref{D2}): 
\begin{eqnarray}\label{beta}
\beta_{b \times \bb}, 
\end{eqnarray}
 this classifying  space is denoted by (\ref{BJ}), using the automorphisms  $\mu_{b \times \bb}^{(4)}$, $\chi^{[4]}$.

Describe explicitly  finite-dimensional skeleton of the classifying space 
\begin{eqnarray}\label{BJ}
	B(\J_b \times \JJ_b \rtimes_{\mu_{b \times \bb}^{(4)}} \Z \times \Z \rtimes_{\chi^{[4]}} \Z). 
\end{eqnarray}
 Firstly, define this bundle over the
subspace of the subgroup (\ref{iaa}).

Take the product $\RP^N/\i \times \RP^N/\i$, $N$ is an arbitrary great odd number.
The automorphism $\mu_{b \times \bb}^{(4)}$
corresponds to the pairs of involutions (a $\Z/2 \times \Z/2$-equivariant mapping) by  the conjugation of the factor:
\begin{eqnarray}\label{involution}
	F\mu_{b \times \bb}^{(4)}: \RP^N/\i \times \RP^N/\i \longrightarrow \RP^N/\i \times \RP^N/\i.
\end{eqnarray}	
Take the torus of the mapping (\ref{involution}), which is skew-product over $T^2 = S^1 \times S^1$
with the layer $\RP^N/\i \times \RP^N/\i$. Take an additional involution $I$ on $\RP^N/\i \times \RP^N/\i$, which permutes the factors. This involution  extends  the pair of involutions $(\mu_{b \times \bb}^{(4)})$ into the triple of pairwise commuted involutions $(\mu_{b \times \bb}^{(4)},I)$.
This extra involution $I$ determines the involution on the mapping torus and the following fibred space,
which is fibered over $3$-dimensional cylinder $T^2 \rtimes_I S^1$:
\begin{eqnarray}\label{cyl}
	 (\RP^N/\i \times \RP^N/\i)  \rtimes_{\mu_{b \times \bb}}  (T^2 \rtimes_I S^1).
\end{eqnarray}
Below in Section \ref{sec7} this fibered space will be used to determine the corresponding local coefficients system, with inversions of cyclic factors (TypeII). 

By the  generator  $I$  the generator
$b \in \J_b$ of the subgroup (\ref{iaa}) is transformed into the generator 
$\bb \in \JJ_b$; the complex conjugation 
$\mu_{b}$ is transformed into the complex conjugation  $\mu_{\bb}$.
The bundle (\ref{beta}) over $2N$-skeleton of the space (\ref{BJ}) is defined as following.
Over the subspace $\RP^N/\i \times \RP^N/\i$ this bundle is the standard $\I_b \times \I_{\b}$-bundle with the class coincide with the sum of the  generators. This class is
invariant with respect to the involutions $\mu^{(4)}_{b \times \bb},\chi^{[4]}$ and is extended as a class over (\ref{cyl}).

\subsubsection{Local coefficients and homology groups}

Let us consider the homology group
$H_{i}(K((\J_b \times \JJ_b),1);\Z)$.
In particular, for an odd $i$ the group contains the fundamental classes of the manifolds
$$S^{i}/\i \times pt \subset S^i/\i \times S^i/\i \subset K(\J_b,1) \times K(\JJ_b,1),$$ 
$$pt \times S^{i}/\i  \subset S^i/\i \times S^i/\i \subset K(\J_b,1) \times K(\JJ_b,1),$$ 
which are denoted by
\begin{eqnarray}\label{tataa}
t_{b,i},  \in H_{i}(K(\J_b,1));\quad t_{\bb,i'} \in H_{i'}(K(\JJ_b,1)).
\end{eqnarray}

A complete description of the group $H_{i}(K((\J_b \times \JJ_b),1);\Z)$ is possible using  the Kunneth formula, as in the case (\ref{Kunnd}). A description of the standard base of the group
$H_{i}(K(\J_{b \times \bb},1);\Z)$ is described by:
\begin{eqnarray}\label{Kunnd2}
\begin{array}{c}
0 \to \bigoplus_{i_1+i_2=i} H_{i_1}(K(\J_b,1);\Z) \otimes H_{i_2}(K(\JJ_b,1);\Z) \longrightarrow \\
\longrightarrow H_i(K(\J_{b \times \bb},1);\Z)
\longrightarrow \\
\longrightarrow \bigoplus_{i_1+i_2=i-1} Tor^{\Z}(H_{i_1}(K(\J_b,1);\Z),H_{i_2}(K(\JJ_b,1);\Z) \to 0.\\
\end{array}
\end{eqnarray}

Let us define the following local integer homology groups
\begin{eqnarray}\label{loc1}
H_{i}^{loc}(K((\J_b \times \JJ_b) \rtimes_{\mu_{b \times \bb}^{(4)}} (\Z \times \Z),1);\Z).
\end{eqnarray} 
which was claimed as homology with interior system coefficients.
Local cycles in (\ref{loc1}) are represented by oriented chains of the classifying space of the group (\ref{first}) equipped with local paths to a marked point. 
An analogous more simple system of interior local coefficients for the single generator was investigated in 
$\cite{A-P}$.

Type I.
The local $\Z \times \dot{\Z} \rtimes \Z$-coefficients system
$\mu_{b \times \bb;I}^{(4)}, \chi^{(4)}$ acts along all the generators
$\mu_b$, $\mu_{\bb}, \chi^{(4)}$ by preserving orientation of local chains. In particular,
the generators of the local system acts (along paths) to global cycle as following: the generator
$t_{b,i}$ for $i \equiv 1 \pmod{4}$ is changed into the opposite along
the path $\mu_b$ and is preserved for $i \equiv 1 \pmod{4}$, the generator $t_{b,i}$
is preserved  along the path $\mu_{\bb}$; the generators
$t_{\bb,i'}$ for $i' \equiv 1 \pmod{4}$ is changed into the opposite along
the path $\mu_{\bb}$ and is preserved for 
$i' \equiv 3 \pmod{4}$;  the generator  $t_{\bb,i'}$ along the path $\mu_b$ is preserved. Because the homology group is described using generators $t_{b,i}$, $t_{\bb,i'}$ by (\ref{Kunnd2}), a global
action of the local system is well defined.

Type II.
The local  $\Z \times \dot{\Z} \rtimes \Z$-coefficients system
$\mu_{b \times \bb;IIb}^{(4)}, \chi^{(4)}$
is determined by the same monodromy of cycles as Type I. The difference is the following: the orientation of chains is changed along the  path $\mu_b$
and  along the path $\mu_{\bb}$ and is preserved along the simplest basic path of $\chi^{(4)}$.
The monodromy of the system is
described using generators as following: 
$t_{b,i}$, $i \equiv 1 \pmod{4}$ keeps the sign along $\mu_b$
and changes the sign for $i \equiv 3 \pmod{4}$ along $\mu_b$, the generator keeps the sign along $\mu_{\bb}$; 
for the generators
$t_{\bb,i'}$ the formula is analogous.

Let us determines classifying space for Type I local system. Consider a skeleton of the space
$B(\J_{b}) \rtimes_{\mu_b} S^1 \times B(\JJ_{b}) \rtimes_{\mu_{\bb}} S^1 \rtimes_{\chi^{(4)}} S^1 $, which is realized by means of the Laurent extension of the structured group 
$\J_{b \times \bb} \subset O(4+4)$, using (\ref{D1}),(\ref{D2}) of a bundle 
over the corresponding Grasmann manifold
$Gr_{\J_{b \times \bb}}(4+4,n)$ of non-oriented $4+4=8$-plans in  $n$--space. 
Denote this Grasmann manifold (analogously to (\ref{KK}) by
\begin{eqnarray}\label{KK4}
KK(\J_{b \times \bb}) \subset K(\J_{b \times \bb},1) \rtimes_{\mu_{b \times \bb}} (S^1 \times S^1).
\end{eqnarray}
On the space $(\ref{KK4})$ a free involution 
\begin{eqnarray}\label{chiK}
\chi^{[4]}: KK(\J_{b \times \bb}) \to KK(\J_{b \times \bb}),
\end{eqnarray}
acts, this involution corresponds to the automorphism
$(\ref{chi3A})$ and represents by a permutation of planes of layers (by a block-antidiagonal matrix).  The space for Type I system is constructed.

Let us determines classifying space for Type II local system. This is an analogous construction, 
but, the extended  Grasmann manifold
$Gr_{\J_{b \times \bb}}(4+4+2,n)$ is required. The $2$-dimensional extension relates with the bundle (\ref{beta}). The space for Type II system is constructed.

Recall a general facts concerning constructed local coefficients system. Take a market point $pt$ in  (\ref{KK4}) and take the image $\chi^{[4]}(pt)=pt^t$ of this point by (\ref{chiK}),
take a path $\lambda$, which joins  $pt$ and $\chi^{[4]}(pt)$. 
Using the path $\lambda$ a conjugated pair of local systems is defined:
\begin{eqnarray}\label{conjls}
\chi^{[4]}: \mu_{b \times \bb}^{loc} \mapsto \mu_{\bb \times b}^{loc}.
\end{eqnarray} 
The left system in the pair admits the corresponding  homology group (\ref{loc1}), 
the right system admits the conjugated homology group:
\begin{eqnarray}\label{loc2}
H_{i}^{loc}(K((\JJ_b \times \J_b) \rtimes_{\mu_{\bb \times b}^{(4)}} (\Z \times \dot{\Z}),1);\Z).
\end{eqnarray} 
A pair of conjugated local systems (\ref{loc2}) is well defined up to an equivalent relation.

The following lemma is analogous to Lemma
$\ref{obstr}$.

\begin{lemma}\label{oaa}

For the local system Type I, Type II the following inclusion is well defined:

$$	\bigoplus_{i_1+i_2=2s} H_{i_1}(K(\J_b,1);\Z) \otimes H_{i_2}(K(\JJ_b,1);\Z) \subset
$$
$$ 
H_{2s}^{loc}(K((\J_b \times \JJ_b) \rtimes_{\mu_{b \times \bb}^{(4)}, \chi^{(4)}} (\Z \times \dot{\Z}) \rtimes \Z,1);\Z). 
$$

\end{lemma}

\subsubsection*{Proof of Lemma \ref{oaa}}
The lemma  is analogous to Lemma \ref{obstr}. \qed

\section*{Standardized 
$\J_b \times \JJ_b$-immersions}

Let us formulate notions of standardized (and pre-standardized)
$\J_b \times \JJ_b$-immersion.

Let us consider a 
$\Z/2^{[4]}$-framed immersion
$(g,\Psi,\eta_N)$ of the codimension 
$8k$.
Assume that the image of the immersion $g$ belongs to a regular neighborhood of an
embedding
$I:S^1_{b} \times S^1_{\bb} \rtimes_{\chi^{[4]}} S^1 \subset \R^n$.

In this formula the third factor $S^1$ is not a direct, the monodromy (an involution)
of a layer $S^1_b \times S^1_{\bb}$  along the exterior circle
is given by the permutation of the factors.
Denote by
$U \subset \R^n$ a regular thin neighbourhood of the embedding $I$.

Let us consider 
a $\Z/2^{[4]}$--framed immersion  
$g: N^{n-8k} \looparrowright U \subset \R^n$, 
for which the following condition (Y) of a control of the structured group of the normal bundle
is satisfied.

\subsection*{Condition (Y)}

The immersion $g$ admits a reduction of a symmetric structured group to  the group
(\ref{secondtot}). 
Additionally, the projection
$\pi \circ g: N^{n-8k} \to S^1_b \times S^1_{\bb} \rtimes_{\chi^{[4]}} S^1$
of this immersion onto the central manifold of $U$  is  agreed with the projection
of the structured group
$(\Z \times \dot{\Z}) \rtimes \Z$ 
onto the factors of the extension.

Formally, a weaker condition
is following.

\subsection{Condition (Y1)}\label{Y1}

Assume an immersion $g$  admits a reduction of a symmetric structured group to  the group
(\ref{secondtot}).

\begin{definition}\label{stand2}
Let us say that a
$\Z/2^{[4]}$-framed immersion
$(g,\Psi,\eta_N)$ of the codimension 
$8k$ is standardized if Condition Y is satisfied.

Let us say that 
$\Z/2^{[4]}$-framed immersion 
$(g,\Psi,\eta_N)$ of the codimension
$8k$ is pre-standardized if  the definition above  week Condition (Y1) instead of Condition (Y) is formulated.
Classification problems of standardized
and of pre-standardized immersions are equivalent.

\end{definition}

Standardized 
$\Z/2^{[4]}$-framed immersions
generate a cobordism group, this group is naturally mapped into 
$Imm^{\Z/4^{[4]}}(n-8k,8k)$ when a standardization  is omitted.

The following definition is analogous to Definition
$\ref{negl}$. 
In this definition a projection of the Hurewitcz  image of a fundamental class
onto a corresponding subgroup
is defined using Lemma
\ref{oaa}.

\begin{definition}\label{pure}
Let
$(g,\eta_N,\Psi)$ be a standardized (a pre-standardized)
$\Z/2^{[4]}$-framed immersion
in the codimension $8k$.
Let us say that this standardized immersion is pure, if the Hurewitcz  image of a fundamental class
in 
$\in D_{n - 8k}^{loc}(\J_{b} \times \JJ_b \rtimes_{\mu_{b \times \bb}} (\Z \times \dot{\Z});\Z/2)$ 
contains not monomial
$t_{b,\frac{n-8k}{2}+i} \otimes t_{\bb,\frac{n-8k}{2}-i}$, $i \in \{\pm 1;\pm 2;\pm 3;\pm 4;\pm 5;\pm 6;\pm 7\}$, but, probably,
contains the only non-trivial monomial $t_{b,\frac{n-8k}{2}} \otimes t_{\bb,\frac{n-8k}{2}}$.
\end{definition}

\begin{definition}\label{negl2}

Let us say that a $\Z/2^{[4]}$-framed immersion 	$(g,\eta_N,\Psi)$ is negligible immersion, if
its $\Z/2^{[2]}$-cover $(f,\Xi,\eta)$ is a negligible standardized $D$-framed immersion in the sense of Definition \ref{negl}.


\end{definition}

\subsection{Dense Principle for formal self-intersected immersions \label{dence}}

		Let $f:M^{n-k}\looparrowright \R^n$ be an immersion, $k\ge 1$. Let us consider the associated equivariant mapping
		of Cartesian products:
		$f\times f:M\times M\to \R^n\times \R^n$, this mapping is also an immersion outside the diagonal in the origin. 
		The inverse image of the diagonal
	$(f\times f)^{-1}(\diag_{\R^n})$ outside the diagonal $\diag_{M}$ is a submanifold denoted by  $\overline{\Delta}(f)$ 
		(previously the denotation $\overline {N}$ was used).
		
		Let us consider a sufficiently small 
		$\eps_2$ neighbourhood $U_{\eps_2}(\diag_{\R^n})$ of the diagonal
		 $\diag_{\R^n}$, such that the inverse image  $(f\times f)^{-1}(\diag_{\R^n})$ has no common points with
		$\overline{\Delta}(f)$. Then take a sufficiently small 
		$\eps_1$--neighbourhood  $U_{\eps_1}(\diag_{M})$ 
		of the diagonal $\diag_M$, such that $f\times f ( U_{\eps_1}(\diag_{M})) \subset U_{\eps_2}(\diag_{\R^n})$.
		
		Let us introduce the following denotation:
		$$
		M^{(2)}=(M\times M)\setminus U_{\eps_1}(\diag_{M}),
		$$
		where the constant
		$\eps_1$ is defined explicitly by the construction. 
		
	Let us define a class of immersions, for which Dense Principle is formulated.

		\begin{definition}\label{proper}
			Let $M^{n-k}$ be a closed manifold. Let us fix positive $\eps_1$ and $\eps_2$. By $T: X \times X$ ($X=M^{n-m}, X=\R^n$) the involution $T(x,y)=(y,x)$ on the Cartesian product is denoted. Let us consider $T$-equivariant immersion $F:M^{(2)}\to \R^n\times\R^n$.
Let us call such an immersion {\it pure immersion}, if
			
			(1) $F(\partial M^{(2)}) \subset U_{\varepsilon_2}(\diag_{\R^n})$;
			
			(2) $F(\partial M^{(2)}) \cap \diag_{\R^n}=\emptyset$.
		\end{definition}

The both conditions can be refodmulated as a common condition:
		$F(\partial M^{(2)})\subset U_{\eps_2}(\diag_{\R^n})\setminus \diag_{\R^n}$.
		
		In the considered example the restriction 
		$f\times f$ on $M^{(2)}$ is a pure immersion.

		\begin{proposition}\label{prop24}
		Let	$f_0: M \looparrowright R$ be a smooth immersion of a compact closed manifold, assume that $R$  is equipped with a metric  $\dist$, $\dim(M) < \dim(R)$. Let  $g: M \to R$ be a continuous mapping, which is homotopic to the immersion $f_0$ (in our case $R=\R^n$ and this extra assumption is unrequired). Then $\forall
			\varepsilon > 0$ there exists an immersion  $f : M \looparrowright R $, which is regular homotopic to the immersion $f_0$, for which $\dist(g;f)_{C^0} <
			\varepsilon $ in the metric induced by the metric  $\dist$.
		\end{proposition}
		
		\begin{proof}
			This is a reformulation of $h$-principle by Hirsh \cite{H}.
		\end{proof}

		\begin{proposition}\label{prop24bis}
			Let $(M^{(2)},\partial M^{(2)})$ be a smooth manifold with a boundary, let 
			$T_{M^{(2)}}: (M^{(2)},\partial M^{(2)}) \to (M^{(2)},\partial M^{(2)})$ be the transposition involution. Let $R^{(2)}$  be a smooth manifold, which is equipped with the metric $dist$, and with the transposition involution		
			$T_{R^{(2)}}: R^{(2)} \to R^{(2)}$, which admits a smooth submanifold of fixed points  $\Delta_R \subset R^{(2)}$, and $\dim(R^{(2)}) = 2n$, $\dim(\Delta_R)=n$, $\dim(M^{(2)}) < \dim(R^{(2)})$.
		We need the case	 $R^{(2)} = \R^n \times \R^n$.  Assume that there exists a pure immersion (in the sence of Definition \ref{proper}), i.e. $T_M,T_R)$ is an equivariant immersion
	$F_0^{(2)}: M^{(2)} \looparrowright R^{(2)}$, 
and, additionally, the image of the boundary has no intersection with the fixed point manifold:
			\begin{eqnarray}\label{uslF0}
				Im(F_0^{(2)}(\partial M^{(2)})) \subset R^{(2)} \setminus \Delta_R.
			\end{eqnarray}
			Let $G^{(2)}: M^{(2)} \to R^{(2)}$ be a smooth $(T_{M^{(2)}},T_{R^{(2)}})$--equivariant mapping, for which the following condition is satisfied: 
			\begin{eqnarray}\label{uslG}
				Im(G^{(2)}(\partial M^{(2)})) \subset R^{(2)} \setminus \Delta_R
			\end{eqnarray}
and which is equivariant homotopic to an immersion
			$F_0^{(2)}$ is the class of equivariant immersions described above. Then 
		 $\forall
			\varepsilon > 0$ there exists $(T_{M^{(2)}},T_{R^{(2)}})$--equivariant pure immersion (in the sence of definition described in Theorem \ref{proper})
			$F_1^{(2)} : M^{(2)} \looparrowright R^{(2)} $, which is regular homotopic to an immersion
			$F_0^{(2)}$, for which the following conditions 
 $\dist(F_1^{(2)};G^{(2)})_{C^0} < \varepsilon $ in the metric, induced by the metric  $\dist$, $R^{(2)}$.
		\end{proposition}
		
		\subsubsection*{Proof of Proposition $\ref{prop24bis}$}
Let us consider a triangulation of the manifold
$(M^{(2)},\partial M^{(2)})$ of the calibre, much more smaller then $\varepsilon$. The induction over skeletons, analogously to Hirsh induction proves the proposition. \qed 
	\[   \]

Let us apply the Proposition \ref{prop24bis} in the following situation.	 Let $(f,\kappa,\Psi)$ be a skew-framed immersion
	$f: M^{n-k} \looparrowright \R^n$. Let us consider an open manifold $M^{n-k} \times M^{n-k} \setminus \Delta_M$, which is equipped with the involution $T_M$.
	Denote by  $M^{(2)}$ a spherical blow-up of the manifold  $M^{n-k} \times M^{n-k} \setminus \Delta_M$ along the diagonal, equipped with the free involution 
 $T^{(2)}_M$. (A spherical blow-up is homeomorphic to $M\times M$ without a small regular deleted $T^{(2)}_M$-equivariant neighbourhood of the diagonal. The boundary of the manifold $M^{(2)}$ is a fibred space of the spherical $S^{n-k-1}$ -bundle over $M^{n-k}$. Let us denote by  $\hat{M}^{(2)}$ 
a quotient $M^{(2)}/T^{(2)}_M$. The boundary $\partial \hat{M}^{(2)}$ coincides with the projectivization $TP(M^{n-k})$ of the tangent bundle  $T(M^{n-k})$. Let us denote $\hat{M}^{(2)} \setminus \partial \hat{M}^{(2)}$ by $\hat{M}^{(2)}_{\circ}$.

Let $(\R^n)^{(2)}$ be a manifold $\R^n \times \R^n$, equipped with the involution $T^{(2)}_{\R^n}$. The following mapping of configuration spaces
	\begin{eqnarray}\label{Gauss}
		f^{(2)}: M^{(2)} \to \R^n \times \R^n,
	\end{eqnarray}
which is a $(T^{(2)}_M,T^{(2)}_R)$--equivariant immersion is well defined. The equivariant immersion
	$(\ref{Gauss})$ satisfies an analogous condition, as for $F_0$ in Proposition \ref{prop24bis}.
	
	Let us use a skew-framed framing
	 $(\kappa,\Psi)$ for the immersion  $f$ to compute the normal bundle  $\nu_{f^{(2)}}$ of the immersion $f^{(2)}$.
Evidently, the immersion
	 $\Psi$ induces the isomorphism:
	\begin{eqnarray}\label{kappa+kappa}
		\bar{\Psi}^{(2)}: \nu_{f^{(2)}} = k(\kappa_1 \oplus \kappa_2),
	\end{eqnarray}
	where
	$\kappa_i$ is a linear bundle, associated with the immersion of the $i$-th factor $i=1,2$.

The involution
	 $T^{(2)}_M$ is covered by an involutive authomorphism of the bundle $\nu_{f^{(2)}}$ (which is not the identity on the base), which permutes the corresponding lines in the right hand side of the formula \ref{kappa+kappa}).  
	Therefore a corresponding vector bundle over  $M^{(2)}$ is well defined, let us denote this bundle by
	$\nu_{f^{(2)}}$. The isomorphism (\ref{kappa+kappa}) induces the isomorphism
	\begin{eqnarray}\label{D}
		\Psi^{(2)}: \nu_{\hat{f}}^{(2)} = k(\eta),
	\end{eqnarray}
	where $\eta$ is the 2-dimensional vector bundle with the structuring group $\D$.
	
	Let us assume that the mapping  (\ref{Gauss}) is transversal along the diagonal
	 $\Delta_{\R^n} \subset \R^n \times \R^n$ and consider the inverse image $(f^{(2)})^{-1}(\Delta_{\R^n})$ of this diagonal, which is denoted by   
	$\bar N \subset \bar{M}^{(2)}$. It is easy to check that $\bar N$ is a closed manifold of the dimension $n-2k$. 
	The manifold $\bar N$ is equipped with a free involution   $T^{(2)}_M \vert_{\bar N}: \bar N \to \bar N$, the factormanifold is denoted by $N^{n-2k}$. The manifold $N^{n-2k}$ is closed and is the formal analogue of the manifold, defined by the formula
	(\ref{Nregb}).

This new definition is a more general, because this makes sense without assumption that the equivariant mapping
$\bar{f}^{(2)}$ is holonomic. This definition is possible for $(T^{(2)}_M,T^{(2)}_R)$--equivariant mappings, which satisfies the condition (\ref{uslF0}), from a regular equivariant homotopy class of the mapping  $f^{(2)}$.
	The immersion  $\bar N^{n-2k} \looparrowright M^{n-k}$ and its restriction $f \vert_{\bar N^{n-2k}}: \bar N^{n-2k} \looparrowright \R^n$ are well defined.
	The normal bundle of the mapping  $f \vert_{\bar N^{n-2k}}$ is isomorphic to the restriction of the bundle $\nu_{f^{(2)}}$ on the submanifold 
	$\bar N \subset M^{(2)}$.

Let us consider the mapping
$g: N^{n-2k} \looparrowright \R^n$, which is defined as a factormapping of the restriction of $f^{(2)}$ on $\bar{N}$. 
From the transversality condition for
$f^{(2)}$ along the diagonal we get that  $g$ is a local embedding, i.e. an immersion. 
The normal bundle
	 $\nu_g$ of the immersion  $g$ is naturally isomorphic to the restriction of the bundle  $\nu_{f}^{(2)}$ on the submanifold
	$N^{n-2k} \subset M^{(2)}$. The dihedral framing, which is defined by the formula  (\ref{D}), coinsids with a dihedral framing  $\Psi$.

	Assume that an arbitrary
	 $(T^{(2)}_{M},T^{(2)}_{\R^n})$--equivariant immersion 
	$G^{(2)}: M^{(2)} \looparrowright \R^n \times \R^n$, 
which satisfies the condition $(\ref{uslG})$, and is equivariant homotopic to the immersion  
	$(\ref{Gauss})$ with a prescribed condition is well defined. In particular, the isomorphism  $(\ref{D})$ is preserved by a regular homotopy. Assume that a regular homotopy keeps the condition  $(\ref{uslG})$ and is transversal along 
	$\Delta_{\R^n} \subset \R^n \times \R^n$.  Let us consider the manifold $N^{n-2k}(G)$ in the target. This manifold is closed, and an immersion $g(G): N^{n-2k}(G) \looparrowright \Delta_{\R^n} = \R^n$ is well defined as the restriction of $G$.
 The normal bundle of this immersion, which is denoted by $\nu_{g(G)}$, is isomorphic to the restriction of the normal bundle of the immersion 
	$G/T_M$. The normal bundle $\nu_{g(G)}$ is equipped with a dihedral framing, which is defined by the formula$(\ref{D})$.
Therefore the following
	 $\D$--framed immersion $(g(G),\eta(G),\Psi(G))$, which is regular homotopic to the  $\D$--framed immersion
	$(g,\eta,\Psi)$ of the self-intersection manifold of the origin skew-framed immersion $(f,\kappa,\Xi)$ is well defined. The constructed $\D$-framed immersion
	 $(g,\eta,\Psi)$ satisfies additional properties, which are defined explicitly starting from the immersion
	$(\ref{Gauss})$.
	
\subsection{Dense Principle for formal iterated  self-intersected immersions}

\subsection*{Formal regular homotopy}
Let $g: N \looparrowright \R^n$ be an arbitrary immersion. 
This immersion induces a $(T_N,T_{\R^n})$- equivariant mapping $g^{(2)}: N^{(2)}=N \times N \looparrowright \R^n \times \R^n$, where $T_N$ is a permutation of points in the origin, $T_{\R^n}$ is a permutation in the target. We will call that
$F_t^{(2)}: N \times N \looparrowright \R^n \times \R^n$ is a formal generic regular deformation of $g^{(2)}$ 
$F_0^{(2)}=g^{(2)}$ and the following conditions are satisfied:

--1. a support $supp(F^{(2)}) \subset N \times N$ of $F^{(2)}$ is outside the diagonal $\Delta_N \subset N \times N$;

--2. $F^{(2)}$ is a regular $(T_N,T_{\R^n})$-equivariant homotopy on $supp(F)$;

--3. 	$F^{(2)}$ is regular along $\Delta_{\R^n}$.

For an arbitrary formal regular homotopy of an immersion $g$ to a formal immersion $g_1^{(2)}$ the self-intersection manifold $L$, $h: L \looparrowright \R^n$, of $g$ is regular cobordant to the manifold $L_1$ of formal self-intersection of   $g^{(2)}_1$, which is defined by
$(g^{(2)}_1)^{-1}(\Delta_{\R^n})/T_N$. For generic $F^{(2)}$ the manifold $L_1$ is equipped with a parametrized immersion $h_1: L_1 \looparrowright \R^n$ and the immersions $h$ and $h_1$ of manifolds $L$ and $L_1$
are regular cobordant.   In the case $g$ is a $\Z/2^{[s]}$-framed immersion,
the immersion $h$ is $\Z/2^{[s+1]}$-framed and $L_1$ and the cobordism between $L$ and $L_1$ are also $\Z/2^{[s+1]}$-framed. This statement is a straightforward generalisation of the construction
for skew-framed immersion, described in \ref{dence} to the case of $\Z/2^{[s]}$-framed immersions.

We need a generalisation of the construction described in Subsection \ref{dence} and we will consider a formal self-intersection immersed manifold of a formal self-intersection immersed manifold of an original immersion $g$ in a common construction.

The immersion $g$ induces a $(T^{(2)}_N,T^{(2)}_{\R^n})$- equivariant mapping $g^{(2+2)}: (N \times N) \times (N \times N) \looparrowright (\R^n \times \R^n) \times (\R^n \times \R^n)$, where $T^{(2)}_N$ is the involution of involutions by permutation of pairs (in a common bracket) of points in the origin, $T^{(2)}_{\R^n}$ is a permutation of pairs of points in the target.
The restriction $g^{(2+2)}$ on a formal self-intersection manifold of $g^{(2)}$ is an immersion and we may apply the iteration of the definition to consider a formal self-intersection points of this restriction.

 We will call that

$$F_t^{(2+2)}: (N \times N) \times (N \times N) \looparrowright (\R^n \times \R^n) \times (\R^n \times \R^n)$$ 

is a formal generic regular deformation of $g^{(2+2)}$, 
$F_0^{(2+2)}=g^{(2+2)}$ if the following conditions are satisfied:

--1. a support $supp(F^{(2+2)}) \subset (N \times N) \times (N \times N)$ of $F^{(2+2)}$ is outside the thick diagonal $\Delta_N^{(2)} \subset (N \times N) \times (N \times N)$;

--2. $F^{(2+2)}$ is a regular $(T^{(2)}_N,T^{(2)}_{\R^n})$-equivariant homotopy on $supp(F)$;

--3. 	$F^{(2+2)}$ is regular along $\Delta_{\R^n}^{(2)}  \subset (\R^n \times \R^n) \times (\R^n \times \R^n)$.

An arbitrary generic immersion $g: N \looparrowright \R^n$ with a self-intersection manifold $h: L \looparrowright \R^n$ admits an iterated self-intersection manifold
$f: M \looparrowright \R^n$
which  is a self-intersection manifold of $h$.
For an arbitrary formal regular homotopy $F^{(2+2)}$ of the extension $g^{(2+2)}$ $g$ to a formal immersion $g_1^{(2+2)}$ the iterated self-intersection manifold $M$, $f: M \looparrowright \R^n$, of $g$ is regular cobordant to the manifold $M_1$ of formal iterated self-intersection of   $g^{(2+2)}_1$, which is defined by
$(g^{(2+2)}_1)^{-1}(\Delta^{(2)}_{\R^n})/T^{(2)}_N$. For generic $F^{(2+2)}$ the manifold $M_1$ is equipped with a parametrized immersion $f_1: M_1 \looparrowright \R^n$ and the immersions $f$ and $f_1$ of manifolds $M$ and $M_1$
are regular cobordant.   In the case $g$ is a $\Z/2^{[s]}$-framed immersion,
the immersion $f$ is $\Z/2^{[s+2]}$-framed and $M_1$ and the cobordism between $M$ and $M_1$ are also $\Z/2^{[s+2]}$-framed.

The definition of formal homotopies (with no regular condition) should be applied for generic singular mappings.

\begin{theorem}\label{T11}

Let
$(g,\Psi,\eta_N)$ be a standardized pure $\D$-framed immersion
in the codimension $2k$ (see Definitions \ref{stand}, \ref{negl}),
$k \ge 7$. Let
$(h,\Psi_L,\eta_{L})$ be the $\Z/2^{[3]}$-framed immersion of self-intersection points of $g$;
$(f,\Psi_M,\zeta_{M})$ be the $\Z/2^{[4]}$-framed immersion
of iterated self-intersection points of $g$.

Then in each regular homotopy class of $\D$-framed standardized immersions there exist $(g,\Psi,\eta)$ and a formal regular deformation $g^{(2+2)} \mapsto \bar{g}^{(2+2)}$,  the iterated self-intersection manifold  of $\bar{g}^{(2+2)}$
is  a disjoin union of a pre-standardized pure $\Z/2^{[4]}$-immersion
$(f,\Xi_M,\eta_M)$ (see  subsection \ref{Y1}, Definition \ref{pure}) and a negligible $\Z/2^{[4]}$-immersion.

Moreover, the cobordism class of the of $\Z/2^{[4]}$-standardized immersion 
$(f,\Xi_M,\eta_M)$ is well defined for
$(g,\Psi,\eta_N)$ in its regular cobordism class.

\end{theorem}

\subsubsection*{A sketch of a proof of Theorem \ref{T11}}
In subsubsection \ref{d^2} a spesial (non-generic) $PL$-mapping $d^{(2)}$, which is a ramified $\Z/4 \times \Z/4$ covering, is constructed. A formal self-intersections of  $d^{(2)}$ is a polyhedron, which is explicitly described in Section \ref{iterated}. This polyhedron admits
a subpolyhedron (\ref{abels2Q}), which becomes a closed component after a vertical formal deformation of the mapping $d^{(2)}$, described in Lemma \ref{l13}. This closed component is a support of a characteristic class, the cohomology class, which is detected Arf-invariant, as is proved in Lemma \ref{l24}. The proof itself of Theorem \ref{T11} is presented at the last part of the section, estimations of the regular codimension for $n=126$ are restricted,  in the case $n \ge 254, \dots$ calculations are not restricted. 

\subsection{Proof of Theorem \ref{T11}}

\subsubsection{Construction of the mapping  $d^{(2)}$ \label{d^2}}
Consider the standard covering $p: S^1 \to S^1$ of the degree $4$. 
It is convenient to write-down:
$p: S^1 \to S^1/\{\i\}$,
by the quotient:
$\Z/4 \times S^1 \to S^1/\{\i\}$. 

Consider the join
of $\frac{n'-k+1}{2}=r$-copies of the circle
$S^1/\{\i\}$ $S^1/\{\i\} \ast \dots \ast S^1/\{\i\} = S^{n'-k}$, 
which is PL-homeomorphic to the standard
$n'-k$-dimensional sphere. Let us define the join of $r$ copies of the mapping
$p$:
$$ \tilde{d}: S^1 \ast \dots \ast S^1 \to S^1/\{\i\} \ast \dots \ast S^1/\{\i\}.$$

On the pre-image acts the group
$\Z/4$ by the diagonal action, this action is commuted with
$\tilde{d}$. The mapping  $\bar{d}$ is defined by the composition of the quotient 
$\tilde{d}/\{\i\}: S^{n'-k}/\{\i\} \to S^{n'-k}$ with the standard inclusion 
$$  S^{n'-k} \subset \R^{n'}. $$
A formal (non-holonomic, vertical) small deformation of the formal (holonomic) extension of the mapping $\bar{d}$ is the required mapping 
$d^{(2)}$.  

\subsubsection{Construction of the mapping  $d^{(2+2)}$}
The mapping 
$\bar{d}: \RP^{n'-k} \to S^{n'-k}$
has to be generalized, using two-stages tower
(\ref{zz}) of ramified coverings. 

Let us recall, that a positive integer $m_{\sigma}=14$.
Denote by
$ZZ_{\J_a \times \JJ_a}$ 
the Cartesian product o standard lens space   $\pmod{4}$, namely, 
\begin{eqnarray}\label{ZZIaIa}
ZZ_{\J_a \times \JJ_a}= S^{n-\frac{n-m_{\sigma}}{8}+1}/\i \times S^{n-\frac{n-m_{\sigma}}{8}+1}/\i.
\end{eqnarray}
Evidently,  $\dim(ZZ_{\J_a \times \JJ_a}) = \frac{7}{4}(n+m_{\sigma})+2 > n$.

On the space $ZZ_{\J_a \times \JJ_a}$ a free involution $\chi_{\J_a \times \JJ_a}: ZZ_{\J_a \times \JJ_a} \to ZZ_{\J_a \times \JJ_a}$ acts by the formula:
$\chi_{\J_a \times \JJ_a}(x \times y)=(y \times x)$.

Let us define a subpolyhedron (a manifold with singularities)
$X_{\J_a \times \JJ_a} \subset ZZ_{\J_a \times \JJ_a}$. 
Let us consider the following family 
$\{X_{j}, \quad j=0,1, \dots, j_{max}\}$, $j_{max} \equiv 0 \pmod{2}$, of submanifolds $ZZ_{\J_a \times \JJ_a}$:
$$X_0 = S^{n-\frac{n-m_{\sigma}}{8}+1}/\i  \times S^1/\i,$$
$$X_1= S^{n-\frac{n-m_{\sigma}}{8}-1}/\i \times S^3/\i, \quad \dots$$
$$ X_{j} = S^{n-\frac{n-m_{\sigma}}{8}+1-2j}/\i \times S^{2j+1}/\i, \quad$$
$$X_{j_{max}} = S^1/\i \times S^{n-\frac{n-m_{\sigma}}{8}+1}/\i,$$
where
\begin{eqnarray}\label{jmax}
j_{max}=\frac{7n+m_{\sigma}}{16}=2^{n-1}, \quad m_{\sigma}=14.
\end{eqnarray}

The dimension of each manifold in this family equals to
$n-\frac{n-m_{\sigma}}{8}+2$ and the codimension in $ZZ_{\J_a \times \JJ_a}$ equals to 
$n-\frac{n-m_{\sigma}}{8}$. Let us define 
an embedding 
$$X_j \subset ZZ_{\J_a \times \JJ_a}$$
by a Cartesian product of the two standard inclusions.

Let us denote by 
$\chi_{\J_a \times \JJ_a}: ZZ_{\J_a \times \JJ_a} \to ZZ_{\J_a \times \JJ_a}$
the involution, which permutes coordinates.
Evidently, we get:
$\chi_{\J_a \times \JJ_a}(X_j) = X_{j_{max}-j}$.

A polyhedron  
$X_{\J_a \times \JJ_a} = \bigcup_{j=0} ^{j_{max}} X_j
\subset ZZ_{\J_a \times \JJ_a}$ 
is well defined. This polyhedron is invariant with respect to the 
involution
$\chi_{\J_a \times \JJ_a}$. 
The polyhedron
$X_{\J_a \times \JJ_a}$ 
can be considered as a stratified manifolds with strata of the codimension 2.

The restriction of the involution 
$\chi_{\J_a \times \JJ_a}$ 
on the polyhedron
$X_{\J_a \times \JJ_a}$ denote by 
$\chi_{\J_a \times \JJ_a}$.

Write-down the sequence of the subgroups of the index $2$ of the diagram
$(\ref{a,aa})$:
\begin{eqnarray}\label{zep}
\I_a \times \II_a \longrightarrow \H_{a \times \aa} \longrightarrow \J_b
\times \JJ_b.
\end{eqnarray}
Define the following tower of $2$-sheeted coverings, which is associated with the sequence
$(\ref{zep})$:
\begin{eqnarray} \label{zz}
ZZ_{a \times \aa}   \longrightarrow ZZ_{\H_a \times \aa} \longrightarrow ZZ_{\J_b \times \JJ_b}.
\end{eqnarray}
The bottom space of the tower
$(\ref{z})$ coincides to a skeleton of the Eilenberg-MacLane space: 
$ZZ_{\J_a \times \JJ_a}
\subset K(\J_a,1) \times K(\JJ_a,1)$. 
This tower (\ref{z}) determines the tower $(\ref{zz})$ by means of the inclusion 
$ZZ_{\J_b \times \JJ_b} \subset K(\J_b,1)
\times K(\JJ_b,1)$ of the bottom.

Let us define the following tower of double coverings:
\begin{eqnarray}\label{z}
X_{a \times \aa}   \longrightarrow X_{\H_{b \times \bb}}  \longrightarrow X_{\J_a \times \JJ_a}.
\end{eqnarray}

The bottom space of the tower  
$(\ref{z})$ 
is a subspace of the bottom space of the tower  
$(\ref{zz})$ by means of an inclusion
$X_{\J_b \times \JJ_b} \subset ZZ_{\J_b \times
\JJ_b}$. 
The tower 
$(\ref{z})$ 
determines as the restriction of the tower 
$(\ref{zz})$ 
on this subspace.

Let us describe a polyhedron   
$X_{a \times \aa} \subset ZZ_{a \times \aa}$
explicitly. Let us define a family 
$\{X'_0,X'_1,  \dots,  X'_{j_{max}} \}$
of standard submanifolds in the manifold 
$ZZ_{a \times \aa} =
\RP^{n-\frac{n-m_{\sigma}}{8}+1} \times \RP^{n-\frac{n-m_{\sigma}}{8}+1}$
by the following formulas:
\begin{eqnarray}\label{Xj}
X'_0 = \RP^{n-\frac{n-m_{\sigma}}{8}+1} \times \RP^{1}  \dots
\end{eqnarray}
$$
X'_{j} = \RP^{n-\frac{n-m_{\sigma}}{8}+1-2j} \times \RP^{2j+1}  \dots
$$
$$
X'_{j_{max}} = \RP^{1} \times \RP^{n-\frac{n-m_{\sigma}}{8}+1}.
$$
In this formulas the integer index 
$j_{max}$ 
is defined by the formula
$(\ref{jmax})$.
The polyhedron
$X_{a \times \aa} \subset ZZ_{a \times \aa}$ 
is defined 
as the union of standard submanifolds in this family.
The polyhedron 
$X_{\H_{a \times \aa}} \subset ZZ_{\H_{a \times \aa}}$
a  quotient of the double covering, which corresponds to the tower of the groups. 

The spaces
$X_{\H_{a \times \aa}}$, $X_{a \times \aa}$ 
admit free involutions, which are pull-backs of the involution
$\chi_{\J_a \times \JJ_a}$ 
by the projection on the bottom space of the tower.

The cylinder of the involution
$\chi_{a \times \aa}$ is well defined,
(correspondingly, of the involution $\chi_{\H_{a \times \aa}}$), which is denoted by
$X_{a \times \aa} \int_{\chi} S^1$ 
(correspondingly, by
$X_{\H_{a \times \aa}} \int_{\chi} S^1$).
The each space is embedded into the corresponding fibred space over  
$\RP^2$:
$$X_{a \times \aa} \int_{\chi} S^1 \subset X_{a \times \aa} \int_{\chi} \RP^2,$$
$$X_{\H_{a \times \aa}} \int_{\chi} S^1 \subset X_{\H_{a \times \aa}} \int_{\chi} \RP^2.$$

Then let us define a polyhedron
$J_{b \times \bb}$, 
which is a base of a ramified covering 
$X_{a \times \aa} \to J_{b \times \bb}$.

Then let us extend the ramified covering over the bottom space of the tower to the ramified covering:
$X_{a \times \aa} \int_{\chi} \RP^2 \to J_{b \times \bb} \int_{\chi} \RP^2$, 
and the ramified covering over the middle space of the tower of the ramified covering
$X_{\H_{a \times \aa}} \int_{\chi} \RP^2 \to J_{b \times \bb} \int_{\chi} \RP^2$.

Let us define a polyhedron (a manifold with singularities) 
$J_{b \times \bb}$. 
For an arbitrary 
$j= 0,1, \dots, j_{max}$, where $j_{max}$
is defined
by the formula
$(\ref{jmax})$, let us define the polyhedron 
$J_{j} =
S^{n-\frac{n-m_{\sigma}}{8}-2j+1} \times S^{2j+1}$ 
(the Cartesian product). 
Spheres (components of this Cartesian product)
$S^{n-\frac{n-m_{\sigma}}{8}-2j+1}$, $S^{2j+1}$ 
are re-denoted by 
$J_{j,1}$, $J_{j,2}$ 
correspondingly.
Using this denotations, we get: 
$$J_{j} = J_{j,1} \times J_{j,2}.$$

The standard inclusion
$i_{j}: J_{j,1} \times J_{j,2}
\subset S^{\frac{n-m_{\sigma}}{8}+1} \times S^{\frac{n-m_{\sigma}}{8}+1}$
is well defined, each factor is included into the target sphere as the standard subsphere.

The union  
$\bigcup_{j=0}^{j_{max}}
Im(i_{j})$ 
of images of this embeddings is denoted by
\begin{eqnarray}\label{JX}
J_{b \times \bb}  \subset
S^{\frac{n-m_{\sigma}}{8}+1} \times S^{\frac{n-m_{\sigma}}{8}+1}.
\end{eqnarray}
The polyhedron
$J_{b \times \bb}$ is constructed.

Let us define a ramified covering 
\begin{eqnarray}\label{varphiX}
\varphi_{a \times \aa}: X_{a \times \aa} \to J_{b \times \bb}.
\end{eqnarray}
The covering
(\ref{varphiX}) 
is defined as the union of the Cartesian products of the ramified coverings,
which was constructed in subsubsection 
\ref{d^2}.

The covering 
(\ref{varphiX})
is factorized into the following ramified covering:
\begin{eqnarray}\label{varphiXH}
\varphi_{\H_{a \times \aa}}: X_{\H_{a \times \aa}} \to J_{b \times \bb}.
\end{eqnarray}
Because
$X_{a \times \aa} \to X_{\H_{a \times \aa}} \to J_{b \times \bb}$
is a double covering, the number of sheets of the covering
(\ref{varphiXH}) is greater by the factor $2^{r}$,
where $r$ is the denominator of the ramification.

The polyhedron
$J_{b \times \bb}$ 
is equipped by the involution
$\chi$, which is defined analogously to the involutions
$\chi_{a \times \aa}$, $\chi_{\H_{a \times \aa}}$.

The cylinder of the involution is well defined,
let us denote this cylinder by
$J_{b \times \bb} \int_{\chi} S^1$. 
The inclusion
$J_{b \times \bb} \int_{\chi} S^1 \subset  J_{b \times \bb} \int_{\chi} \RP^2$.
is well defined.

The ramified covering
$(\ref{varphiX})$ commutes with the involutions
$\chi_{a \times \aa}$, $\chi_{\H_{a \times \aa}}$ 
in the origin and the target.

Therefore the ramified covering
\begin{eqnarray}\label{varphiXS}
c_X: X_{a \times \aa} \int_{\chi} \RP^2 \to J_{b \times \bb} \int_{\chi} \RP^2,
\end{eqnarray}
which is factorized into the ramified covering
\begin{eqnarray}\label{varphiYS}
c_Y: X_{\H_{a \times \aa}} \int_{\chi} \RP^2 \to J_{b \times \bb} \int_{\chi} \RP^2.
\end{eqnarray}
is well defined.

\begin{lemma}\label{embJ}
There exist an inclusion
\begin{eqnarray}\label{iJ}
i:  J_{b \times \bb} \int_{\chi} \RP^2 \times D^{5} \rtimes D^3 \subset \R^n,
\end{eqnarray}
where  
$D^{5}$
is the standard  disks (of a small radius), the disk $D^{5}$ is a layer of the trivial $9$-bundle over $\RP^2$;
$D^{3}$
is the standard  disks (of a small radius), the disk $D^{3}$ is a layer of the non-trivial twisted $3$-bundle $3\kappa_{\chi}$ over $\RP^2$, where $\kappa_{\chi}$ is the line bundle with the generic class in $H^1(\RP^2;\Z/2)$.


\end{lemma}

\subsubsection*{Proof of Lemma 
$\ref{embJ}$}

Recall, $n=2^l-2, l \ge 8$.
Put
$m_{\sigma}=14$,  $k = \frac{n-m_{\sigma}}{16}$. In the case $l=8$ we get $k=15$ and this case is the most complicated.

Let us prove that the polyhedron
$J_{b \times \bb}$ 
is embeddable into the sphere
$S^{n-11}$.

The sphere $S^{n-11}$ 
is the unite sphere (the target of the standard projection) of  an
$n-10$-dimensional 
layer. 
Take a collection of embedding: 
$S^1 \times S^{n-13} \subset S^{n-11}$, $S^3 \times S^{n-15} \subset S^{n-11}$, $\dots$, $S^{1} \times S^{n-13} \subset S^{n-11}$.
Take the Euclidean coordinates $x_1, \dots, x_{n-10}$ in $\R^{n-10}$, where
the sphere $S^{n-11} \subset \R^{n-10}$ is the unite sphere with the centre at the origin.
Each embedding
$S^{2k_1+1} \times S^{2k_2+1} \subset S^{n-11}$, $k_1+k_2=\frac{n-14}{2}$  is the 
hypersurface in $S^{n-11}$ given by the equation $x_1^2 +  \dots + x_{2k_1+2}^2 - x_{2k_1+3}^2 - \dots - x_{n-10}^2=0$. 
The  polyhedron $J_{b \times \bb}$ is $PL$-homeomorphic to the union of the tori.

The normal bundle of the embedding  
$emb: \RP^2 \subset \R^n$
is the Whitney sum:
$$ \nu(emb) =   (\frac{n-12}{2}+1)\varepsilon \oplus (\frac{n-12}{2}+1) \kappa_{\chi}  \oplus 3\kappa_{\chi} \oplus 5\varepsilon.$$
because $n-12 = 2 \pmod{16}$. 
The polyhedron $J_{b \times \bb} \rtimes \RP^2$
is embedded into  $(\frac{n-12}{2}+1)\varepsilon \oplus (\frac{n-12}{2}+1) $.

Lemma \ref{embJ}
is proved.
\qed

There exist a formal deformation of the  ramified covering
(\ref{varphiXS}), which satisfies  special properties, which are formulated in Definition 
\ref{abstru2}.

Let $\bar{d}_{X}=i \circ c_{\bar{X}}: \bar{X}_{a \times \aa} \rtimes_{\chi} \RP^2 \to \R^n$,
$\bar{d}^{(2+2)}_X$ be a formal $\Z/2 \times \Z/2$-equivariant holonomic extension  of the mapping $\bar{d}_X$ on the space of $(2+2)$-configurations:
the space of non-ordered pairs of non-ordered two-points. A general facts on such extensions are not required, because we will considered the only following case.

Define a formal mapping
$\bar{d}^{(2+2))}_{X}$ as a small formal deformation of the holonomic extension $\bar{d}^{(2+2)}_X$ of the mapping  $\bar{d}_X$. The deformation is a codimension $4$ vertical with respect to the line bundles described in Lemma \ref{embJ} and is $\chi^{(4)}$-equivariant.

On the first stage of the construction we consider the mapping
$\bar{d}_Y$
$\bar{d}_Y=i \circ c_Y: Y_{\H_{a \times \aa}} \rtimes_{\chi} \RP^2 \to \R^n$
and its formal 
$\Z/2$-extension  $(\bar{d}_Y)^{(2)}$. On the second stage  we consider the double formal extension $((\bar{d}_X)^{(2)})^{(2)}$, this extension determines
an extra self-intersections over self-intersection points of the mapping  $(\bar{d}_Y)^{(2)}$. This resulting polyhedron is called a polyhedron of
iterated self-intersections. To define this polyhedron  two steps are required and one may changes an order of the steps, if required. 
By this, on the first step  self-intersection points
of the upper polyhedron in the tower is organized. On the second step self-intersection points of of the polyhedron   by the mapping  
$\bar{d}_Y$ is organized. Recall the towers of coverings are parametrized over $\RP^2$.

Then we define a special deformation (formal) $((\bar{d}_X)^{(2)})^{(2)} \mapsto d^{(2+2)}_{X}$. Properties of iterated self-intersections are described in Definition \ref{abstru2} and then proved in Lemma \ref{l13}. 

\begin{definition}\label{abstru2}

Let us say that a formal mapping	
$d^{(2+2)}_{X}$ (by a formal vertical equivariant deformation of the holonomic extension of (\ref{iJ})) admits a bi-cyclic structure, if the following condition is satisfied: the polyhedron of iterated self-intersections is divided  into two subpolyhedra: a closed subpolyhedron
$\NN_{b \times \bb}$ and a subpolyhedron  $\NN_{\circ}$ with a boundary. Additionally, the following conditions, concerning a reduction of the structuring mapping are satisfied:

-- 1.  $\NN_{\circ}$ is neglected (the Hurewitcz image of the fundamental class (which exists)
has the trivial characteristic numbers).

--2. On a closed polyhedron (a component of self-intersection polyhedron is closed and does not contain boundaries singular points)  $\NN_{b \times \bb}$, 
outside the critical subpolyhedron $Q \subset \NN_{b \times \bb}$ of dimension less then $32$
the structuring mapping admits the following reduction:  
\begin{eqnarray}\label{abels2Q}
\eta_{b \times \bb}: \NN_{b \times \bb} \setminus Q \to K(\J_{b \times \bb},1) \int_{\mu_{b \times \bb}^{(4)},\chi^{[4]}} S^1 \times S^1 \times \RP^2,
\end{eqnarray}
where the structuring mapping corresponds to the group
(\ref{secondtot}). 

Additionally,  the closed subpolyhedron
$\NN_{b \times \bb}$  contains a fundamental class  $[\NN_{b \times \bb}]$.
The image of the fundamental by the structuring mapping (\ref{abels2Q}) 
is estimated (is non-trivial) as following. Take an arbitrary 
submanifold   $\RP^{\frac{n-2k}{2} + t} \times \RP^{\frac{n-2k}{2} - t} \subset \bar{X}_{a \times \aa}$.  Take a small generic (formal) alteration of the mapping $d_X$, restricted on this submanifold; consider the polyhedron $\N(t)$ of iterated self-intersections of this mapping, which are near $\NN_{b \times \bb}$. The image of the fundamental class $[\NN(t)]$ by the mapping (\ref{abels2Q}) (after the 4-ordered transfer) is the following 
element in $D^{loc}_{n-8k}(\I_{a \times \aa};\Z/2)$, which is
calculate as $\kappa_a^{\frac{n-8k}{2}+2t}\otimes\kappa_{\aa}^{\frac{n-8k}{2}-2t}$.

\end{definition}

\begin{lemma}\label{l13}{A big Lemma}
	
There exists a formal mapping, which satisfies  Definition \ref{abstru2}.
\end{lemma}

\subsubsection*{Proof of Lemma \ref{l13}}
This is an analog of Lemma \ref{7}. A difference is:  the deformation $c_0^{2+2} \mapsto d^{2+2}_X$ is not a generic holonomic, but is generic equivariant and vertical.
A required deformation (formal and vertical) with bi-cyclic structure (see Definition \ref{abstru2}).
This deformation is the result of a gluing a collection of elementary deformations over 
elementary component in the decomposition (\ref{jmax}). Over the each elementary component
the deformation is detected by the deformation, using the resolution mapping Theorem \ref{secoequil},1. For a simplicity of denotations, let us consider the case  $n\ge254$, $k\ge 15$. Homology conditions (\ref{l24}) for fundamental classes of elementary strata 
are analogous for the result of a deformation.

Take the normal bundle structure. Let us consider the factor $3\varepsilon$ with the standard  torus $\gamma: T^2  \subset \R^3$ inside, the factor $2\kappa_{\chi}$, which is called the elimination space of the secondary deformation, the factor $2\varepsilon$, which is called the elimination space for the primary deformation. Using Lemma \ref{rezbb}, a vertical deformation along elimination spaces  to kill the polyhedron $KK_{\I\I_d}$, $KK_{\I\I_b,\I_d}$ is well defined.
After this deformation iterated self-intersection is localized inside a closed polyhedron $KK_{\I\I_b;\I_b}$, all the last component give no contribution to the characteristic number by Lemma \ref{l22}.

To pass into $RKK_{\II_b;\I_b}$ (\ref{RIbsub}) we get a $"$primary$"$ resolution mapping over $KK_{\I_b}$, using the middle row of the diagram  (\ref{a,aa}).
This deformation constructed in Lemma \ref{secoequil}.

The self-intersection points of $KK_{\I\I_b,\I_b}$ are parametrized over $6$-dimension: $(S^1)^4\times \RP^2$, this gives $6$ less parameters of the vertical lift. We have 2 extra less parameters of the vertical linft, because of restrictions in (\ref{subchi}). The calculation of the defect  in subsketion \ref{subchi} proves that the deffect is not intersected with the quadruple self-intersection polyhedron of a vertical generic deformation, because $40 < 3\cdot2^{l-4}$, $l \ge 8$.
 Lemma 
\ref{l13} is proved. \qed

\subsection*{Proof of Theorem  \ref{T11}}

Consider a standardized $\D$-framed immersion
$(g,\Psi,\eta_N)$, where $g$ is closed to the composition of the structuring mapping
$\eta_N: N^{n-2k} \to X_{\H_{a \times \aa}} \int_{\chi} \RP^2$ with the mapping  (\ref{iJ}). 
Then a formal deformation of the immersion $g \mapsto g_0$, which is induced by the formal deformation
$\bar{d} \mapsto d$ is considered. This formal deformation is fixed near diagonals.
The bi-cyclic component of the iterated self-intersection is defined as the component
in a neighbourhood of a regular domain of the component
$\NN_{b \times \bb}$. By the Hirsh principle, which can be explicitly applied for formal immersions, the formal immersion $g_0$ has a component $N^{n-8k}$, which is immersed into a regular neighbourhood of $\NN_{b \times \bb}$.

The codimension of the mapping $\eta_N$ equals to $8$, including the extended torus $\gamma$.
Then the codimension of the iterated self-intersection points of $g_0$ equals to $32$. Because the
singular polyhedron of the mapping  (\ref{abels2Q}) in $\NN_{b \times \bb}$ has the dimension strictly less then $32$, the induced structure on $N^{n-8k}$, described in Definition \ref{stand2},
is regular. 

The image of the structuring mapping on $N^{n-8k}$, restricted to the canonical covering, belongs to the subspace (\ref{KKin2}), because the defect on the iterated self-intersection is empty, see 
Definition \ref{stand}.

The structuring group of the characteristic mapping has the required reduction,
as it is followed from Statement 2 in Definition 
\ref{abstru2}.
The characteristic number in Statement 2 Definition  \ref{abstru2}
is restricted using the collection of characteristic number for the polyhedra  $\NN(t)$ (see Lemma \ref{l24}). The immersion $g_0$ is pure in the sense of the Definition \ref{pure}. The only characteristic number in the described collection could be non-trivial. Theorem  \ref{T11} is proved.
\qed

\section{$\Q \times \Z/4$--structure (quaternionic-cyclic structure)
on self-intersection manifold of 
a standardized $\Z/2^{[4]}$--framed immersion}\label{sec7}

Let us take  the standard order $8$ quaternionic subgroup
$\{ \pm1,\pm\i, \pm\j, \pm\bf{k}\}=\Q \subset \Z/2^{[3]}$, which contains the order $4$ cyclic subgroup $\J_b \subset \Q$.

Let us define the following subgroups:

\begin{eqnarray} \label{QZ4}
i_{\J_a \times \JJ_a, \Q \times \Z/4}: \J_a \times \JJ_a \subset  \Q \times \Z/4,
\end{eqnarray}

\begin{eqnarray} \label{iQaa}
i_{\Q \times \Z/4}: \Q \times \Z/4 \subset
\Z/2^{[5]},
\end{eqnarray}

\begin{eqnarray} \label{iaaa}
i_{\J_a \times \JJ_a \times \Z/2}: \J_a \times \JJ_a \times \Z/2
\subset
\Z/2^{[5]}.
\end{eqnarray}

Define the subgroup $(\ref{QZ4})$. 
Define the epimorphism
$\J_b \times \JJ_b \to \Z/4$ by the formula
$f_1: (x \times y) \mapsto xy$. The kernel of this epimorphism coincides with
the antidiagonal subgroup
$\I_b = antidiag(\J_b \times \JJ_b) \subset \J_b \times \JJ_b$,
and this epimorphism admits a section, the kernel is a direct factor
(the subgroup $\J_b \subset \J_b \times \JJ_b$).
This kernel is mapped onto the group
$\Z/4$ by the formula:
$f_2:(x \times x^{-1}) \mapsto x$.
The inclusion (\ref{QZ4}) is given by the formula $(f_1 \times f_2)$.

Let us define subgroups
$(\ref{iQaa})$, $(\ref{iaaa})$.
Consider the basis (\ref{ff})
in the space $\R^8$.
Let us define an analogous basis of the space 
$\R^{16}$. This basis contains
$16$ vectors, the basis vectors are divided into two subset $16=8+8$.

\begin{eqnarray}\label{Qh}
		\begin{array}{c}
		Lin_{\x} = \{\e_{1,+}, \dots, \e_{4,+}, \e_{1,-}, \dots \e_{4_-}\}, \\
		 Lin_{\y} = \{\f_{1,+}, \dots, \f_{4,+}, \f_{1,-}, \dots \f_{4,-}\}.
	\end{array}
\end{eqnarray}

Let us define the subgroup
$(\ref{iQaa})$. 
The representation
$i_{\Q \times
\Z/4}$ 
is an extension of the two diagonal representations of $\J_b \times \JJ_b$ in the subspaces, labelled by $+$ and by $-$, by the element $\j \in \Q$ (or, by $\bf{k} \in \Q$).
Recall that the representation of $\J_a$ in the bottom row of the diagram (\ref{a,aa}) admits the
invariant planes $Lin_{\x}(\e_1,\e_2), Lin_{\x}(\e_3,\e_4),Lin_{\y}(\f_1,\f_3),Lin_{\y}(\f_2,\f_4)$.
The representation of $\JJ_a$ in the same bottom row  admits the
invariant planes $Lin_{\x}(\e_1,\e_3), Lin_{\x}(\e_2,\e_4),Lin_{\y}(\f_1,\f_2),Lin_{\y}(\f_3,\f_4)$.

 The planes are transformed into invariant $4$-dimensional spaces of $(\ref{iQaa})$. The element $\j \in \Q \times \Z/4$
is represented by the following matrix:
$$ \j: \e_{s,+} \mapsto \e_{s,-}, \quad \e_{s,-} \mapsto -\e_{s,+}; \quad s=1,2,3,4; $$
$$ \j: \f_{s,+} \mapsto \f_{s,-}, \quad \f_{s,-} \mapsto -\f_{s,+}; \quad s=1,2,3,4. $$
For the generator $\i \in \J_a$ we get the standard $\Q$-representation is the subspaces
$Lin(\e_{1,+},\e_{2,+},\e_{1,-},\e_{2,-})$, $Lin(\e_{3,+},\e_{4,+},\e_{3,-},\e_{4,-})$, and
the $\Z/2 \times \Z/4$-representation in the subspaces
$Lin(\f_{1,+},\f_{3,+},\f_{1,-},\f_{3,-})$, $Lin(\f_{2,+},\f_{4,+},\f_{2,-},\f_{4,-})$.
For the generator $\i \in \JJ_a$ we get the standard $\Q$-representation is the subspaces
$Lin(\f_{1,+},\f_{2,+},\f_{1,-},\f_{2,-})$, $Lin(\f_{3,+},\f_{4,+},\f_{3,-},\f_{4,-})$, and
the $\Z/2 \times \Z/4$-representation in the subspaces
$Lin(\e_{1,+},\e_{3,+},\e_{1,-},\e_{3,-})$, $Lin(\e_{2,+},\e_{4,+},\e_{2,-},\e_{4,-})$.

The representation $(\ref{iQaa})$ is extended to the representation
\begin{eqnarray} \label{QtilZ5}
	i_{\J_a \times \JJ_a, \tilde{\J}_a \times \tilde{\JJ}_a, \Q \times \Z/4}: \J_a \times \JJ_a \subset \tilde{\J}_a \times \tilde{\JJ}_a  \to  \Q \times \Z/4.
\end{eqnarray}

Let us define the representation
$(\ref{iaaa})$ as following.
The factor
$\J_b \times \JJ_b \subset \J_b \times \JJ_b \times
\Z/2$ is represented in each $8$-dimensional 
subspace with $+$ and $-$ indexes
by the bottom line representation in the diagram (\ref{a,aa}).
The factor $\Z/2
\subset \J_b \times \JJ_b \times \Z/2$ 
is represented by the diagonal in the each planes $(\e_{s,+},\e_{s,-})$, $(\f_{s,+},\f_{s,-})$, $s=1, \dots, 4$.

The representation $(\ref{iaaa})$ is extended to the representation
\begin{eqnarray} \label{QtilZ5}
	i_{\J_a \times \JJ_a \times \Z/2, \tilde{\J}_a \times \tilde{\JJ}_a \times \Z/2}: \J_a \times \JJ_a \times \Z/2 \subset \tilde{\J}_a \times \tilde{\JJ}_a \times \Z/2 \subset \Z/2^{[5]}.
\end{eqnarray}

On the $16$-dimensional space by direct sum of $8$-dimensional subspaces (\ref{Qh})
the involutive automorphism
$\chi^{[5]}$ is well-defined by the permutation of the corresponding base vectors in the subspaces.

Consider the projection
\begin{eqnarray}\label{prQ}
p_{\Q}: \Q \times \Z/4 \to \Q
\end{eqnarray} 
on the first factor. The kernel of the homomorphism
$p_{\Q}$ coincides with the antidiagonal subgroup
$\II_b \subset \J_b \times \JJ_b \subset \Q \times \Z/4$. 
Evidently, the following equality is satisfied on the first factor. The kernel of the homomorphism 
$p_{\Q}$ coincides with the antidiagonal subgroup 
$\II_b \subset \J_b \times \JJ_b \subset \Q \times \Z/4$. 

Evidently, the following equality is satisfied:
\begin{eqnarray}\label{chi5com}
p_{\Q} \circ \chi^{[5]}    = p_{\Q}.
\end{eqnarray}

Analogously, on the group 
$\J_b \times \JJ_b \times \Z/2$ 
define the automorphism
$\chi^{[5]}$ of the order 2 (this new automorphism denote the same).
Define the projection
\begin{eqnarray}\label{prZ4Z2}
p_{\Z/4 \times \Z/2}: \J_a \times \JJ_a \times \Z/2 \to \Z/4 \times \Z/2,
\end{eqnarray} 
the kernel of this projection coincides with the diagonal subgroup

Obviously, the following formula is satisfied the kernel of this projection coincides with the diagonal subgroup
$\II_b \subset \J_b \times \JJ_b \times \Z/2$.
Obviously, the following formula is satisfied:
\begin{eqnarray}\label{chi5com2}
\chi^{[5]} \circ p_{\Z/4 \times \Z/2}  = p_{\Z/4 \times \Z/2}.
\end{eqnarray}
This allows to define analogously with
$(\ref{chi2Z})$, $(\ref{chi3Z})$, $(\ref{chi4Z})$
the groups

\begin{eqnarray}\label{chi5Z}
(\Q \times \Z/4) \int_{\chi^{[5]}} \Z,
\end{eqnarray}
\begin{eqnarray}\label{chi55Z}
(\J_a \times \JJ_a \times \Z/2) \int_{\chi^{[5]}} \Z,
\end{eqnarray}
as semi-direct products of the corresponding groups with automorphisms with the  group $\Z$.

Let us define the epimorphism:
\begin{eqnarray}\label{pchi5Z}
\omega^{[5]}: (\Q \times \Z/4) \int_{\chi^{[5]}} \Z \to \Q,
\end{eqnarray}
the restriction of this epimorphism on the subgroup
$(\ref{QZ4})$ coincides with the epimorphism 
$(\ref{prQ})$. For this definition use the formula 
$(\ref{chi5com})$ and define
$z \in Ker(p_{\Q})$, where
$z \in \Z$ is the generator.

Evidently, the epimorphism
\begin{eqnarray}\label{pchi55Z}
\omega^{[5]}: (\J_a \times \JJ_a \times \Z/2) \int_{\chi^{[5]}} \Z \to \Z/4 \times \Z/2,
\end{eqnarray}
analogously is well defined, denote this automorphism as the automorphism
$(\ref{pchi5Z})$.

Moreover, the following homomorphisms
\begin{eqnarray}\label{chi5ZQ}
\Phi^{[5]}: (\Q \times \Z/4) \int_{\chi^{[5]}} \Z \to \Z/2^{[5]},
\end{eqnarray}
\begin{eqnarray}\label{chi5ZZZ}
\Phi^{[5]}: (\J_b \times \JJ_b \times \Z/2) \int_{\chi^{[5]}} \Z \to \Z/2^{[5]},
\end{eqnarray}
which extend the diagram $(\ref{D2})$, are well defined, they are included into the following commutative diagrams
$(\ref{a,Qaa})$, $(\ref{a,aaa})$. 

\begin{eqnarray}\label{a,Qaa}
\begin{array}{ccc}
\qquad (\J_a \times \JJ_a) \int_{\chi^{[4]}} \Z & \stackrel{\Phi^{[4]} \times \Phi^{[4]})}
{\longrightarrow} & \qquad \Z/2^{[4]} \times \Z/2^{[4]} \\
i_{\J_a \times \JJ_a, \Q \times \Z/4} \downarrow \qquad & &  i_{[5]} \downarrow \\
\qquad (\Q \times \Z/4) \int_{\chi^{[5]}} \Z &  \stackrel {\Phi^{[5]}}{\longrightarrow}& \qquad \Z/2^{[5]}, \\
\end{array}
\end{eqnarray}

In this diagram the left vertical homomorphism
$$i_{\J_b \times \JJ_b, \Q \times \Z/4}:
(\J_b \times \JJ_b) \int_{\chi^{[4]}} \Z \to (\Q \times \Z/4) \int_{\chi^{[5]}} \Z$$
is induced by the homomorphism 
$(\ref{QZ4})$, the right vertical homomorphism
$$i_{[5]}: \Z/2^{[4]} \times \Z/2^{[4]} \subset \Z/2^{[5]}.$$
is the inclusion of the subgroup of the index 2.

The following diagram
\begin{eqnarray}\label{a,aaa}
\begin{array}{ccc}
\qquad (\J_a \times \JJ_a) \int_{\chi^{[4]}} \Z & \stackrel{\Phi^{[4]} \times \Phi^{[4]}}
{\longrightarrow} & \qquad \Z/2^{[4]} \times \Z/2^{[4]} \\
i_{\J_a \times \JJ_a, \J_a \times \JJ_a \times \Z/2} \downarrow \qquad & &  i_{[5]} \downarrow \\
\qquad (\J_a \times \JJ_a \times \Z/2) \int_{\chi^{[5]}} \Z &  \stackrel {\Phi^{[5]}}{\longrightarrow}& \qquad \Z/2^{[5]}, \\
\end{array}
\end{eqnarray}
the left vertical homomorphism
$$i_{\J_a \times \JJ_a, \J_a \times \JJ_a \times \Z/2}: (\J_a \times \JJ_a) \int_{\chi^{[4]}} \Z \to
(\J_a \times \JJ_a \times \Z/2) \int_{\chi^{[5]}} \Z$$
is an inclusion.

\begin{lemma}\label{l14}
The homomorphism
(\ref{chi5ZQ}) 
is extended 
from the index 2 subgroup 
$\J_b \times \JJ_b \rtimes_{\mu_{b \times \bb}^{(4)}} \Z \times \Z \rtimes_{\chi^{[4]}} \Z$ to the group
\begin{eqnarray}\label{hom1}
(\Q \times \Z/4) \rtimes_{\mu_{b \times \bb}^{(5)}} \Z \times \Z \int_{\chi^{[5]}} \Z.
\end{eqnarray}
\end{lemma}

\subsection*{Proof of Lemma \ref{l14}}

Let us construct the extension
$$(\Q \times \Q) \rtimes_{\mu_{b \times \bb}^{(5)}} \Z \times \Z \int_{\chi^{[5]}} \Z,$$
then let us pass to the required subgroup of the index $2$.

The automorphism
$\mu_{b \times \bb}^{(5)}$
is induced
from the automorphism
$\mu_{\Q}: \Q \to \Q$ of the factors:
$\mu_{\Q}(\i)=-\i$, $\mu_{\Q}(\j)=-\j$, $\mu_{\Q}(\bf{k})=\bf{k}$. \qed

Additionally, one has to consider, as this automorphism is defined on
$3$-dimensional quaternionic lens
$S^3/\Q$ (the quaternionic action on $S^3$ is on the right).

The automorphism is given by the transformation of $S^3$, which is the right multiplication of the quaternion unite. 
The automorphism
$\mu_{\Q}$ is commuted with the transformation of 
the quaternion units (the only non-obvious calculation concerns the unite 
$\bf{k}$). Therefore on the quotient
$S^3/\Q$ the automorphism is well defined. 

The trivialization of the tangent bundle on the quaternionic lens, which is defined by  the left action of
the units  $(\i,\j,\bf{k})$, 
is changed by corresponding automorphisms. Therefore the image of the quaternion framing by the differential
\begin{eqnarray}\label{muQ}
d(\mu_{\Q}): T(S^3/\Q) \to T(S^3/\Q)
\end{eqnarray}
is fibred isotopic to the identity. This isotopy consists of a family of rotations
trough the angle $180^0$ 
in oriented planes, which are orthogonal to the vectors
$\bf{k}$. 

The homomorphism
(\ref{chi5ZZZ}) is extend to the homomorphism of the Laurent extension from the subgroup
$\J_b \times \JJ_b \rtimes_{\mu_{b \times \bb}^{(4)}} \Z \times \Z \rtimes_{\chi^{[4]}} \Z$ 
of the index  $2$ to the all group: 
\begin{eqnarray}\label{hom2}
(\J_b \times \JJ_b \times \Z/2) \rtimes_{\mu_{b \times \bb}^{(5)}} \Z \times \Z \int_{\chi^{[5]}} \Z \mapsto \Z/2^{[5]}.
\end{eqnarray}

\begin{definition}\label{standQ}

Let us say that a
$\Z/2^{[4]}$-framed immersion 
$(g,\Psi,\eta_N)$ of the codimension
$8k$, which is standardized in the sense of Definition 
\ref{stand2} and is pure in the sense of Definition \ref{pure}, is an immersion with 
$\Q \times \Z/4$-structure 
(with a quaternion-cyclic structure), if $g$ is self-intersects along
a $\Z/2^{[5]}$-framed immersion 
$(h,\Xi,\zeta_L)$
with the self-intersection manifold  
$L$ of the dimension
$n_{\sigma}=n-16k=14$, (\ref{nsigma}) (for  $n=254$ we have $k=15$,
$n-16k=14$ ) and the following condition is satisfied.
The manifold $L$ is divided into two components:
\begin{eqnarray}\label{twocomp} 
	L = L_{\Q \times \Z/4} \cup L_{\Z/4 \times \Z/4 \times \Z/2},
\end{eqnarray}	
	 fundamental classes of a component is represented into the classifying space for the corresponding internal symmetry group in (\ref{hom1}), (\ref{hom2}). 
\end{definition}

\begin{lemma}\label{Q}

Let
$(g,\eta_N,\Psi)$ a a pre-standardized  $\Z/2^{[4]}$-framed immersion
in the codimension $8k$, see Definition 	\ref{stand2}.
Then in a regular homotopy class of this immersion there exists a standardized immersion with quaternionic-cyclic structure, see Definition  
\ref{standQ}.		

\end{lemma}

\subsubsection*{Proof of Lemma \ref{Q}}

Let us define the space
\begin{eqnarray}\label{ZZQ}
ZZ_{\J_a \times \JJ_a}= S^{\frac{n+6}{2}+1}/\i \times
S^{\frac{n+6}{2}+1}/\i
\end{eqnarray}
Evidently,  $\dim(ZZ_{\J_a \times \JJ_a}) = n + 8 > n$.

Let us define a family
$\{Z_0,  \dots,  Z_{j_{max}} \}$ of standard submanifolds in the manifold
$ZZ_{\J_a \times \JJ_a}$, where
\begin{eqnarray}\label{jmaxQ}
j_{max} = \frac{n+6}{8},
\end{eqnarray}
by the following formula:
\begin{eqnarray}\label{Zj}
Z_0 = S^{\frac{n+12}{2}} \times
S^{3}/\i \subset
S^{\frac{n+12}{2}}/\i \times
S^{\frac{n+12}{2}}/\i, 
\end{eqnarray}
$$
Z_1 = S^{\frac{n+4}{2}} \times
S^{7}/\i \subset
S^{\frac{n+12}{2}}/\i \times
S^{\frac{n+12}{2}}/\i,  \dots
$$
$$
Z_j = S^{\frac{n+12-8j}{2}}/\i
\times S^{4j+3}/\i \subset
S^{\frac{n+12}{2}}/\i \times
S^{\frac{n+12}{2}}/\i,  \dots
$$
$$
Z_{j_{max}} = S^{3}/\i \times
S^{\frac{n+12}{2}}/\i \subset
S^{\frac{n+12}{2}}/\i \times
S^{\frac{n+6}{2}}/\i.
$$

In the formula
$j_{max}$ is defined by the formula $(\ref{jmaxQ})$.
A subpolyhedron
\begin{eqnarray}\label{Z}
Z_{a \times \aa} \subset ZZ_{\J_a \times \JJ_a}
\end{eqnarray}
is defined as the union of submanifolds of the family
$(\ref{Zj})$. Evidently,
$\dim(Z_{\J_a \times \JJ_a})=
\frac{n+18}{2}$.

The standard involution, which is permuted the factors, is well defined:
\begin{eqnarray}\label{ZZ3}
\chi^{[4]}: ZZ_{\J_a \times \JJ_a} \to ZZ_{\J_a \times \JJ_a}, 
\end{eqnarray}
this involution is invariant on the subspace
$Z_{a \times \aa}$ (\ref{Z}).

Let us consider the standard cell-decomposition of the space $K(\J_a
\times \JJ_a,1)= K(\J_a,1) \times K(\JJ_a ,1)$, 
this cell-decomposition is defined as the direct product of the cell-decomposition
of the factors. The standard inclusion of the skeleton is well defined:
\begin{eqnarray}\label{etaIaIaZ}
Z_{a \times \aa} \subset K(\J_a
\times \JJ_a,1).
\end{eqnarray}

The skeleton 
$(\ref{etaIaIaZ})$ is invariant with respect to the involution $(\ref{ZZ3})$. 
Therefore the inclusion
\begin{eqnarray}\label{ostov}
Z_{a \times \aa} \int_{\chi^{[4]}} S^1 \subset Z_{a \times \aa} \int_{\chi^{[4]}} S^1 \subset K((\J_a \times \JJ_a) \int_{\chi^{[4]}} \Z,1) 
\end{eqnarray}
is well defined.

Let us define a polyhedron $J_Z$, a subpolyhedron of the polyhedron $J_X$,
see the formula
$(\ref{JX})$. Denote by  $JJ_{\J_a}$ the join of $j_{max}$ 
copies of the lens space
$S^{3}/\i$. Denote by  $JJ_{\J_a}$ the join of $j_{max}$ 
copies of the lens space
$S^{3}/\i$. 

Denote by  $JJ_{\Q}$ the join of $j_{max}$ copies of the quaternionic lens space 
$S^{3}/\Q$, the second exemplar of this join.
Recall that the positive integer 
$j_{max}$ is defined by the formula  $(\ref{jmaxQ})$.

Let us define the subpolyhedron
$J_{\bar{Z},j} \subset
JJ_{\J_a} \times JJ_{\JJ_a}$ by the formula 
$$J_{\bar{Z},j} = JJ_{\J_a,j} \times JJ_{\JJ_a,j}, 1 \le i \le j_{max},$$
where $JJ_{\J_a,j} \subset JJ_{\J_a}$ is the subjoin with the coordinates
$1 \le
i \le j_{max}-j$, $JJ_{\JJ_a,j} \subset JJ_{\JJ_a}$ is the subjoin with the coordinates  $1 \le i \le j$. Let us define the subpolyhedron
$J_{Z,j} \subset
JJ_{\Q} \times JJ_{\QQ}$ by the formula: 
$$JJ_{Z,j} = JJ_{\Q,j} \times JJ_{\QQ,j}, 1 \le i \le j_{max},$$
where $JJ_{\Q,j} \subset JJ_{\Q}$ is the subjoin with coordinates
$1 \le
i \le j_{max}-j$, $JJ_{\QQ,j} \subset JJ_{\QQ}$ is the subjoin with the coordinates $1 \le i \le j$.

Define $J_Z$ by the formula:
\begin{eqnarray}\label{JZ}
J_Z = \cup_{j=1}^{j_{max}} J_{Z,j} \subset JJ_{\Q} \times JJ_{\QQ}.
\end{eqnarray}
Define $J_{\bar{Z}}$ by the formula: 
\begin{eqnarray}\label{barJZ}
J_{\bar{Z}} = \cup_{j=1}^{j_{max}} J_{\bar{Z},j} \subset JJ_{\J_a} \times JJ_{\JJ_a}.
\end{eqnarray}

On the polyhedron
$Z_{a \times \aa}$, $J_Z$ a free involution, which is denoted by
$T_{\Q}$, and which is corresponded to the quadratic extension $(\ref{QZ4})$
is well defined. The polyhedron
$(\ref{JZ})$ is invariant with respect to the involution $T_{\Q}$, 
the inclusion of the factor-polyhedron is defined by 
\begin{eqnarray}\label{QJZ}
(J_Z)/T_{\Q} = \cup_{j=1}^{j_{max}} J_{Z,j}/T_{\Q} \subset JJ_{\J_a}/T_{\Q} \times JJ_{\JJ_a}/T_{\Q}.
\end{eqnarray}

The standard $4$-sheeted covering (with ramifications)
\begin{eqnarray}\label{nakr}
c_Z: Z_{a \times \aa} \to  J_Z, 
\end{eqnarray}
is well defined. This covering is factorized to the ramified covering
$\hat c_Z$. This covering is defined by the composition of the standard $2$-sheeted covering  
\begin{eqnarray}\label{nakrTQ}
Z_{a \times \aa}/T_{\Q} \to (J_Z)/T_{\Q} 
\end{eqnarray}
and the standard $2$-sheeted covering:
$Z_{a \times \aa} \to Z_{a \times \aa}/T_{\Q}$.
The ramified covering
$(\ref{nakr})$ is equivariant with respect to the involution $(\ref{ZZ3})$,
which acts in the target and the source.

Let us consider the embedding $S^{3}/\Q  \subset
\R^{4}$, which is constructed in
[Hi]. 
The Cartesian product of joins of corresponding copies of the embeddings
gives the embedding
\begin{eqnarray}\label{Jemb}
J_{Z,j} \subset \R^{5j_{max}+3}.
\end{eqnarray}

The following lemma is analogous to Lemma
$\ref{embJ}$.

\begin{lemma}\label{embJZ}
There exists an embedding
\begin{eqnarray}\label{iJZ}
\rm{i}_{J_Z}: J_Z \int_{\chi} S^1  \subset
\D^{n-5} \times S^1  \subset \R^n.
\end{eqnarray}
\end{lemma}
\[  \]

The following extension of the polyhedron
(\ref{ostov}), which corresponds to the Laurent extension of the structuring groups, as in Lemma  \ref{l14}, is well defined.
Let us define a ramified covering:

\begin{eqnarray}\label{nakr1}
c_Z: Z_{a \times \aa} \rtimes_{\mu_{b \times \bb}} S^1 \times S^1 \to  
J_Z \rtimes_{\mu_{b \times \bb}} S^1 \times S^1, 
\end{eqnarray}
which is an extension of the covering
(\ref{nakr}), as a parametrized covering over $2$-dimensional torus.
The factors of the torus acts on factors of the coverings  
(\ref{nakr}), correspondingly with the extension of the formula
(\ref{muQ}) on the quaternion lens space. 

The inclusion of the base of the covering
(\ref{embJZ}) is extended to the following immersion (which is denoted by the same):
\begin{eqnarray}\label{s12}
\rm{i}_{J_Z}: J_Z \int_{\mu_{b \times \bb},\chi} S^1  \looparrowright
\D^{n-5} \times S^1 \times S^1 \rtimes_{\chi} S^1 \subset \R^n.
\end{eqnarray}

An arbitrary immersion
$g$ with a $\Q \times \Z/4$-structure is defined as a small approximation of a mapping,
which is the composition of the mapping to the total space of the $2$-covering
(\ref{nakr1}), the ramified covering itself, and the immersion  (\ref{s12}). 

The decomposition to maximal strata of self-intersection points of the ramified covering
(\ref{nakr1}) determines a control of the image of the fundamental class of self-intersection points of the immersion 
$g$ by the structuring mapping. 
Cycles of the types
1,2  correspond to structuring groups of maximal strata, cycles of type 3
can be exist, because   
(\ref{s12}) is an immersion and could have additional closed component
of a self-intersection. Lemma
\ref{Q} is proved. \qed

In the following Lemma the main result is proved.
 

\begin{theorem}\label{Arf}
Let $(g,\Psi,\eta_N)$ $\Z/2^{[4]}$ be a standardized (see Definition  \ref{stand2})
immersion, with $\Q \times \Z/4$-structure (see Definition  \ref{standQ}).
Let us assume that   $(g,\Psi,\eta_N)$ is pure (see Definition \ref{pure}). 
Then
$(g,\Psi,\eta_N)$ is negligible. 
\end{theorem}

\subsubsection*{Proof of Theorem \ref{Arf}}
Let us consider the immersion
$$ g: N^{n-8k} \looparrowright \R^n, $$
with the controle over $\RP^2 \times (S^1 \times S^1)$, which corresponds to the structuring group of the immersion, see see (\ref{prZ4Z2}).
The self-intersection manifold $L^{n-16k}$ is divided into two components, denoted shortely by $L_1$, $L_2$, as in the formula (\ref{twocomp}), with the stucturing groups, described by the dyagrames (\ref{a,Qaa}), (\ref{a,aaa}) correspondingly.

The Arf-invariant is described by the characteristic number, the image of the fundamental class in $D_{n_{\sigma}}^{loc}(\J_b \times \JJ_b \rtimes_{\mu_{b \times \bb}}(\Z \times \dot{\Z});\Z)$ by  the characteristic mapping of  the cannonical covering $[\bar{L}_1]$ over the component $L_1$. The  cannonical covering $[\bar{L}_2]$ over the component $L_2$ has trivial charackteristic numbers, because this mapping is compressed in two-sheeted covering over the mapping (\ref{hom2}). 

The local coeffisient system on $D_{n_{\sigma}}^{loc}(\J_b \times \JJ_b \rtimes_{\mu_{b \times \bb}}(\Z \times \dot{\Z});\Z/2)$ modulo 2 is trivial. The local charactheristic classes are pull-backs of classes over the group (\ref{chi55Z}), which are detected by 
the  4-shetted covering as the classes of the group (\ref{DZ2lok}). All the modulo 2 charactheristic classes of $L_2$ are trivial

The immestion $g$ is pure. This implies
 the only class for $L_1$ near the central Arf-monomial of the basus, described in Lemma \ref{obstr}, is defined by (\ref{arf}), which is Arf-invariant. 

We use arguments related with the Herbert immersion theorem \cite{He}.
We use a non-standard modification of the Herbert formula with local
$\mu_{b \times \bb},\chi^{(4)}$ coefficients valued in the group described in Lemma \ref{oaa}, and for immersions with additional marked component, which gives an extra contribution on the both sides of the Herbert formula. 
The Euler class of the bundle
(\ref{beta}) is a submanifold in the codimension 2, denote this submanifold by $N_1^{n-8k-2} \subset N^{n-8k}$. 

Let us restrict the immersion $g$ on the submanifold
$N_1^{n-8k-2} \subset N$ and let us consider the manifold $L_1^{n_{\sigma}-4} \looparrowright L^{n_{\sigma}}$ 
of self-intersection points of the immersed manifold
$g(N_1^{n-8k-2})$. 
The fundamental class of this submanifold is calculated using the Euler classes of the structured $4$-bundle over $L$ and
the image $[\bar{L}^{n_{\sigma}}]$, which is  an element in 
$D_{n_{\sigma}}^{loc}(\J_b \times \JJ_b \rtimes_{\mu_{b \times \bb}}(\Z \times \dot{\Z});\Z)$. This element modulo 2 is non-trivial and is given by the corresponding non-trivial monomial of the basis described in Lemma \ref{oaa}. We get that Arf-invariant 
is described as the Hirewich image of the fundamental class $[L_1^{n_{\sigma}-4}]$, which is an odd nontrivial class modulo $4$.

The calculation of  $[\bar{L}_1^{n_{\sigma}-4}]$ is  possible by the Herbert theorem, which is applied to $[\bar{L}]$. The calculation is also possible by a geometrical argument
as the transfer of the Euller class as the  $\Q$-class of the structuring group (\ref{hom1}).
The characteristic classes takes values modulo 4. By arguments as in \cite{A-P}, this characteristic class in   $D_{n_{\sigma-4}}^{loc}(\J_b \times \JJ_b \rtimes_{\mu_{b \times \bb}}(\Z \times \dot{\Z});\Z)$
is opposite to itself. Because $g$ is pure, this proves that the characteristic number is even. 
Theorem  \ref{Arf} is proved.  \qed

\section{Classifying space for self-intersections of skew-framed immersions \label{sec8}}

\subsection{Preliminary constructions and definitions}
In this section we assume that $n$, $k$ are even positive integers, $n>k$;
 $r\ge 3$, $\frac{n-k+2}{2}=r$,   in an arbitrary positive integer? we assume $r \equiv-1\pmod{3}$.

\subsubsection{Axillary mappings  $c: \RP^{n-k+1} \to
\R^n$,
$c_0: \RP^{n-k+1} \to
\R^n$, which are used in Lemma  \ref{7}, in Theorem $\ref{T11}$}

Let us denote by $J_0 \subset \R^{n}$ the standard sphere of
dimension $(n-k) $ in Euclidean space. This sphere is $ PL $ -
homeomorphic to the join of $\frac{n-k+2}{2}=r$ copies of
the circles $ S^1 $. Let us denote by $i_{J_0}: J_0 \subset \R^{n}$
the standard embedding.

Let us denote by
\begin{eqnarray}\label{iJ}
i_{J}: J \subset \R^{n} 
\end{eqnarray} 
the standard $(n-k+1)$-dimensional  sphere in Euclidean space.
This sphere is $ PL $ -
homeomorphic to the join of $r$ copies of
the circles $ S^1 $.

Define the following mapping
$p_0: S^{n-k+1} \to J_0$ as  the join  of $r$ copies of the standard 4-sheeted coverings  $S^1 \to S^1/\i$. The standard antipodal action  $-1 \in \I_d
\times S^{n-k+1} \to S^{n-k+1}$ (for the group  $\I_d$ and analogous denotations  $\I_\b$, $\I_{\bf{a},\aa}$ for subgroups in the dihedral group $\D$ below
see Section  \ref{sec1}) commutes with the mapping  $p_0$.
Therefore the ramified covering 
$p_0: \RP^{n-k+1} \to J$ is well defined. The mapping
\begin{eqnarray}\label{c0}
c_0: \RP^{n-k+1} \to \R^{n}
\end{eqnarray}
 is defined as the composition 
$i_J \circ p_0$.

Replace the mapping $p_0$ defined above by the mapping $p$, the join of $r$ copies of $2$-sheeted coverings $S^1 \to \RP^1$. The ramified covering $p: \RP^{n-k+1} \to J$ is analogously well defined. The mapping
\begin{eqnarray}\label{c}
 c: \RP^{n-k+1} \to \R^n
\end{eqnarray} 
is the composition $i_J \circ p$. 
	
	\subsubsection{Deleted square}

Let us consider the following space
\begin{eqnarray}\label{99}
(\RP^{n-k+1} \times
\RP^{n-k+1} \setminus \Delta)/T' = \Gamma_{\circ},
\end{eqnarray}
which is called the "deleted product", or, the "deleted square" of the space
$\RP^{n-k+1}$.  
The subspace $\Delta \subset \RP^{n-k+1} \times
\RP^{n-k+1}$ is  the diagonal subspace of the Cartesian product. 
 This space (\ref{99}) is defined by the factor on the Cartesian product by the involution $T': \RP^{n-k+1} \times \RP^{n-k+1} \to
\RP^{n-k+1} \times \RP^{n-k+1}$, which permutes the coordinates, outside the diagonal. 
The deleted product is an open manifold.

\subsubsection{The mapping $c_0$ (\ref{c0}), $c$ (\ref{c}) as formal mappings with a formal self-intersection}

Denote by
$T_{\RP^{n-k+1} \times \RP^{n-k+1}}$, $T_{\R^n \times \R^n}$
the standard involutions in the spaces  $\RP^{n-k+1} \times \RP^{n-k+1}$, $\R^n \times \R^n$, which permute the coordinates.
Let
\begin{eqnarray}\label{dd2}
d^{(2)}: \RP^{n-k+1} \times \RP^{n-k+1} \to \R^n \times \R^n 
\end{eqnarray}
be an arbitrary $T_{\RP^{n-k+1} \times \RP^{n-k+1}}$, $T_{\R^n \times \R^n}$--equivariant mapping, which is transversal along the target diagonal. 
Denote the polyhedron
$(d^{(2)})^{-1}(\R^n_{diag})/T_{\RP^{n-k+1} \times \RP^{n-k+1}}$  by $\N(d^{(2)})$,
let us call this polyhedron a formal self-intersection polyhedron of the mapping 
$d^{(2)}$.  In the case when the formal mapping 
$d^{(2)}$ is a holonomic mapping and is the formal extension of a mapping $d$
(the cases $d=c_0$, $d=c$ are required) the polyhedron 
$\N_{\circ}(d^{(2)})$ coincides with the polyhedron
of self-intersection points of $d$, which is defined by the formula:
\begin{eqnarray}\label{Nd}
\N_{\circ}(d^{(2)}) =  \{ ([\x,\y]) \in \Gamma_{\circ} : \y \ne \x, d(\y)=d(\x) \},
\end{eqnarray}
where
$\Gamma_{\circ}$ is the deleted square (\ref{99}).
We shall use the short denotation $K_{\circ} \subset \N_{\circ}$, where $\N_{\circ} = \N_{\circ}(d^{(2)})$ (see subsubsections \ref{/N}, \ref{secc0})  in the case $d^{(2)}$ is the formal extension of $c$, or, of $c_0$. In a general case $d^{(2)}$ is the result of a small regular equivariant deformation $c^{(2)} \mapsto d^{(2)}$, or, of the deformation $c_0^{(2)} \mapsto d^{(2)}$. Such a deformation hase to keep the involution (\ref{hatcompstrat}), which is free on $K_{\circ}$, outside the antidiagonal of the polyhedron $\N_{\circ}(d^{(2)})$. When we get a deformation we have inclusions:  
\begin{eqnarray}\label{inc}
 \N_{\circ}(d^{(2)}) \subset \N_{\circ}(c^{(2)}),  
\end{eqnarray}
\begin{eqnarray}\label{inc0}
\N_{\circ}(d^{(2)}) \subset \N_{\circ}(c_0^{(2)}).
\end{eqnarray}

\subsubsection{The structuring mappings $\eta_{\N_{\circ}}: \N_{\circ} \to
K(\D,1)$ for the mappings $c_0$, $c$ }

Let us define the mapping
\begin{eqnarray}\label{stG}
\eta_{\Gamma_{\circ}}: \Gamma_{\circ} \to K(\D,1),
\end{eqnarray}
which is called the structuring mapping.

The inclusion
$\bar \Gamma_{\circ} \subset \RP^{n-k+1} \times
\RP^{n-k+1}$ induces an isomorphism of the fundamental groups, because the codimension of the diagonal $\Delta \subset \RP^{n-k+1}
\times \RP^{n-k+1}$ equals $n-k+1$ and satisfies the following inequality: 
$n-k+1 \ge 3$. Therefore, the following equation is satisfied:
\begin{eqnarray}\label{pi_1}
\pi_1(\bar \Gamma_{\circ}) = H_1(\bar \Gamma_{\circ};\Z/2) = \Z/2 \times \Z/2.
\end{eqnarray}

Let us consider the induced automorphism
$T'_{\ast}: H_1(\bar
\Gamma_{\circ};\Z/2) \to H_1(\bar \Gamma_{\circ};\Z/2)$, for the definition of $T'$ see (\ref{99}). This automorphism is not the identity. Let us fix the standard isomorphism $H_1(\bar \Gamma_{\circ};\Z/2)$ with $\I_{A \times \dot{A}}$, this isomorphism maps the generator of the first factor (correspondingly, the generator of the second factor) of the group $H_1(\bar \Gamma_{\circ};\Z/2)$ (see
(\ref{pi_1})) to the element  $A=ab \in \I_c \subset \D$
(correspondingly, to the element $\dot{A}=ba \in \I_c \subset \D$), 
which is represented in the representation of
$\D$ by the symmetry with respect to the second (correspondingly, with respect to the first) coordinate axis. 

It is easy to check that the automorphism of external conjugation of the subgroup
$\I_{A \times \dot{A}} \subset \D$ on the element $b \in \D \setminus \I_{A \times \dot{A}}$ (in this formula the element
$b \in \I_b$ is a generator), which is defined by the formula:  $x \mapsto bxb^{-1}$, is transformed by the considered isomorphism into the automorphism  $T'_{\ast}$. The fundamental group  $\pi_1(\Gamma_{\circ})$
is a quadratic extension of the group  $\pi_1(\bar \Gamma)$
by the element $b$, this extension is uniquely determined (up to an equivalence) by the automorphism $T'_{\ast}$. Therefore 
$\pi_1(\Gamma_{\circ}) \simeq \D$ and the mapping 
$\eta_{\Gamma_{\circ}}: \Gamma_{\circ} \to K(\D,1)$ is well defined.

The structuring mapping
$\eta_{\N_{\circ}}: \N_{\circ} \to
K(\D,1)$ 
is defined by the restriction of
$\eta_{\Gamma_{\circ}}$ on the open subpolyhedron (\ref{Nd}). The construction depends of
 $c_0$ (\ref{c0}), $c$ (\ref{c}).

\subsection{Stratifications}		
	
	In this subsection the polyhedrons (\ref{Nd}) for mappings  (\ref{c0}), (\ref{c}) 
	are described. 
	
	\subsubsection{The stratification of the sphere $ J $ \label{stratJ}}

Let us ordered 1-dimensional lens spaces (circles), which are generators of the join, by positive integers
$1, \dots, r$, and denote by $J(k_1, \dots, k_s) \subset J$
the subjoin, which is defined by the join of the subcollection of
circles
$S^{1}/\i$  with the numbers $1 \le k_1 < \dots < k_s
\le r$, $0 \ge s \ge r$.
The stratification is induced by the stratification of the standard $r$-dimensional simplex
$\delta^{r}$ by the projection $J \to \delta^{r}$.
The inverse images of the vertexes of the simplex are circles
$J(j)
\subset J$, $ J (j) \approx S^1$, $ 1 \le j \le r $, which are irreducible components of the join.
The circle looks like the base $S^1/-1$ of $2$-sheeted covering for the mapping (\ref{c}), and like the base $S^1/\i$ for $4$-sheeted covering for the mapping (\ref{c0}).

Let us define the space
$J_1^{[s]}$ as a subspace in $J$ as the union of all subspaces  $J(k_1, \dots, k_s) \subset J$.

Therefore, the following stratification 
\begin{eqnarray}\label{stratJ}
J^{(r)} \subset \dots \subset J^{(1)} \subset J^{(0)},
\end{eqnarray}
of the space $
J $ is well defined. 
For a given stratum the number
$ r-s $ of omitted coordinates is called the deep of the stratum.

Let us introduce the following denotations:
\begin{eqnarray}\label{strat[J]}
J^{[i]} = J^{(i)} \setminus J^{(i+1)}.
\end{eqnarray}
	
\subsubsection{The stratification of $\RP^{n-k}$ \label{stratRP}}

Let us define the stratification of the projective space
$\RP^{n-k+1}$.
Denote the maximal open cell
$
p^{-1}(J(k_1, \dots, k_s))$ by $U(k_1, \dots, k_s) \subset
S^{n-k+1}/-1$. 
This cell is called an elementary stratum of the deep $r-s$.
A point on an elementary stratum
$U(k_1, \dots, k_s) \subset S^{n-k+1}/-1$
is defined by the collection of coordinates
$(\check x_{k_1}, \dots, \check
x_{k_s}, l)$, where $\check x_{k_i}$ is the coordinate
on the circle, which is covered the circle of the join with the number $k_i$, $l$ is a coordinate on
$(s-1)$-dimensional simplex of the join. 
Points on an elementary stratum
$U(k_1, \dots, k_s)$ belong to the union of subcomplexes of the join
with  corresponding coordinates.

\subsubsection{The stratification of the polyhedron  
$\N_{\circ}$ (\ref{Nd}) for the mapping $c_0$ (\ref{c0})} \label{secc0}

Let us start with the  case $d=c_0$.
The polyhedron
$\N_{\circ}$, (\ref{Nd}) is the disjoin union 
of elementary strata, which are defined as inverse images of the strata
$(\ref{strat[J]})$.
Let us denote tis strata by
\begin{eqnarray}\label{compstrat}
K^{[r-s]}(k_1, \dots, k_s), \qquad 1 \le s
\le r. 
\end{eqnarray}
Let us consider exceptional antidiagonal strata of $\N_{\circ}$ inside the anti-diagonal
\begin{eqnarray}\label{antidiag}
\Delta_{anti}= \{ ([x,y]) \in \Gamma_0: x = \i y \}.
\end{eqnarray}
An open polyhedron, which is defined by the cutting-out of anti-diagonal strata from  $\N_{\circ}$ is denoted by $K_{\circ} \subset \N_{\circ}$. 

Let us describe an elementary stratum
$ K^{[r-s]} (k_1, \dots, k_s) $ by means of coordinate system.
For the simplicity of denotations, let us consider the case
$s=r$, this is the case of maximal elementary stratum.  
Assume a pair of points
$(\x_1$, $\x_2)$ determines a point on
$K^{[0]}(1, \dots, r)$, let us fixes for the pair a pair of points
$(\check \x_1, \check \x_2)$ on the covering sphere
$S^{n-k}$,  which is mapped into the pair $(\x_1$,
$\x_2)$ by the projection $ S^{n-k+1} \to \RP^{n-k+1} $.
With respect to constructions above, denote by
$ (\check x_{1, i}, \check x_{2, i}) $, $ i = 1, \dots, r $
the collection of spherical coordinates of each point. 
A spherical coordinate determines a point on the circle
with the same number
$i$, the circle is a covering over the corresponding coordinate lens 
$J(i) \subset J$ in the join. Recall, that the pair of coordinates
with a common index determines a par in a layer of the standard cyclic 
$\i$-covering  $ S^1
\to S^1 / \i $.

Collections of coordinates  $\{(\check x_{1,i}, \check x_{2,i})\}$
are considered up to a transformation of the full collection into the antipodal collection. Moreover, a non-ordered pair of points
$[\x_1,\x_2]$ is lifted into a pair of points   $(\bar \x_1,
\bar \x_2)$ on the sphere  $S^{n-k}$  not uniquely. 
The collection of coordinates of a lift is defined up to 
$8$ different transformations. (Up to transformations of the group  $\D$ of the order $8$.)

Analogous statement is well-formulated for points in deeper elementary strata
$ K^{[r-s]} (k_1, \dots, k_s) $, $ 1 \le s \le r $.
	
\subsubsection{The stratification of the polyhedron  
$\N_{\circ}$ (\ref{Nd}) for the mapping $c$ (\ref{c})}

The polyhedron
$\N_{\circ}$, (\ref{Nd}) is the disjoin union 
of elementary strata, which are defined as inverse images of the strata
$(\ref{strat[J]})$ analogously with the case in the subsection \ref{secc0}.
Let us denote this strata by (\ref{compstrat}) as in the construction for the mapping $c_0$.
In the considered case  exceptional antidiagonal strata  on $\N_{\circ}$ are absent.

Let us describe an elementary stratum
$ K^{[r-s]} (k_1, \dots, k_s) $ by means of coordinate system.
Assume a pair of points
$(\x_1$, $\x_2)$ determines a point on
$K^{[0]}(1, \dots, r)$, let us fixes for the pair a pair of points
$(\check \x_1, \check \x_2)$ on the covering sphere
$S^{n-k}$,  which is mapped into the pair $(\x_1$,
$\x_2)$ by the projection $ S^{n-k+1} \to \RP^{n-k+1} $.
Denote by
$ (\check \x_{1, i}, \check \x_{2, i}) $, $ i = 1, \dots, r $
the collection of spherical coordinates of each point. 
 Recall, that the pair of coordinates
with a common index determines a pair in a layer of the standard $2$-fold covering  $ S^1
\to S^1 / -1 $.

Collections of coordinates  $(\check x_{1,i}, \check x_{2,i})$
are considered up to antipodal transformations of the full collection into the antipodal collection. Moreover, pairs of points
$(\x_1,\x_2)$ are non-ordered and lifts of a point in $K$ to a pair of points  $(\bar \x_1,
\bar \x_2)$ on the sphere  $S^{n-k+1}$ are not uniquely defined. 
Therefore collection of coordinates are defined up to 
$8$ different transformations. (Up to transformation of the group  $\D$ of the order $8$.)

Analogous statement is well-formulated for points in deeper elementary strata
$ K^{[r-s]} (k_1, \dots, k_s) $, $ 1 \le s \le r $.
	
	\subsubsection{Coordinates on
		$K_{\circ} \subset \N_{\circ}$, the mapping $c_0$ (\ref{c0}) \label{resedues}}
	
	Let 
	$x \in  K^{[r-s]} (k_1, \dots, k_s)$ be a point on an elementary stratum.
	Let us consider collections of spherical coordinates 
	$\check x_{1,i}$, $\check x_{2,i}$, $k_1 \le i \le k_s$, $1 \le i \le r_0$ 
	of the point  $x$. For each $i$ the following cases are possible:
	the pairs of $i$-th coordinates coincides; antipodal; the second coordinate is the result of the transformation of the first coordinate by the generator $\i$ of the cyclic cover $S^1 \mapsto S^1/\i$. 
	Let us define for an ordered pair of coordinates
	$
	\check x_{1, i} $, $ \check x_{2, i} $, $ 1 \le i \le r $  a residue
	$v_i$ with values $+1$, $-1$, $+\i$, $-\i$:
	$ v_{k_i} = \check x_{1, k_i}(\check x_{2, k_i})^{-1} $.
	
	By transformation of the collection of coordinates 
	$\check x_2$ into the antipodal collection, the collection of residues
	are transformed by the antipodal opposition:
	$$\{(\check x_{1, k_i},\check x_{2, k_i})\} \mapsto \{(-\check x_{1, k_i},\check x_{2, k_i})\}, \quad \{v_{k_i}\} \mapsto \{-v_{k_i}\},$$ 
	$$\{(\check x_{1, k_i},\check x_{2, k_i})\} \mapsto \{(\check x_{1, k_i},-\check x_{2, k_i})\}, \quad \{v_{k_i}\} \mapsto \{-v_{k_i}\}.$$ 
	
	By transformation of the collections of coordinates by the renumeration:
	$(\check x_1,\check x_2) \mapsto (\check x_2, \check x_1)$,
	the collection of residues are transformed by the complex conjugated collection:
	into the antipodal collection, the collection of residues
	are transformed by the antipodal opposition:
	$$\{(\check x_{1, k_i},\check x_{2, k_i})\} \mapsto \{(\check x_{2, k_i},\check x_{1, k_i})\}, \quad \{v_{k_i}\} \mapsto \{\bar {v}_{k_i}\},$$ 
	where  $v \mapsto \bar {v}$  means the complex conjugation. 
	Evidently, the collection of residues remains fixed in the case a different  point on the elementary stratum is marked.
	
	Elementary strata  $K^{[r-s]} (k_1, \dots, k_s)$,
	accordingly with collections of coordinates, are divided into 3 types:
	$\I_{b},\I_{a \times \aa},\I_{d}$. In the case residues takes values in
	$\{+\i,-\i\}$ (correspondingly, in $\{+1,-1\}$), one say that an elementary stratum is of the type
	$\I_{b}$ (correspondingly, of the type $\I_{a \times \aa}$). 
	In the case there are residues of the complex type
	$\{+\i,-\i\}$, and the real type $\{+1,-1\}$ simultaneously, one shall say
	about an elementary stratum of the type $\I_{d}$. 
	It is easy to check that the restriction of the characteristic mapping
	$\eta : K_{\circ} \to K(\D,1)$
	on the elementary stratum of the type $\I_{b},\I_{a \times \aa},\I_{d}$ correspondingly, is a composition of a mapping into the space $K(\I_{b},1)$
	(correspondingly, into the space  $\I_{a \times \aa}$, or  $K(\I_{d},1)$) with
	the mapping $i_{b}: K(\I_{b},1) \to K(\D,1)$ (correspondingly with the mapping  $i_{a \times \aa}: K(\I_{a \times \aa},1) \to K(\D,1)$, or with $i_{d}: K(\I_{d},1)
	\to K(\D,1)$). 
	For strata of the first two types a reduction
	of the structuring mapping (up to homotopy) is not well defined; this reduction
	is well defined only up to the conjugation
	of the subgroup
	$\I_{b}$ (correspondingly, in the subgroup $\I_{a \times \aa}$).
	
	Recall, the polyhedron
	$\N_{\circ}$ is defined by addition elementary anti-diagonal strata to the polyhedron $K_{\circ}$.	On each elementary anti-diagonal stratum
	$ K (k_1, \dots, k_s) $ a residue of each coordinate equals $+\i$.

	Define the following  open polyhedrons:
	\begin{eqnarray}\label{Ib}
		K_{\I_b \circ} \subset K_{\circ} \subset \N_{\circ}
	\end{eqnarray}
	\begin{eqnarray}\label{Ia}
		K_{\I_{a \times \aa} \circ} \subset K_{\circ} \subset \N_{\circ}
	\end{eqnarray}
	\begin{eqnarray}\label{Id}
		K_{\I_d \circ} \subset K_{\circ} \subset \N_{\circ}
	\end{eqnarray}
	as polyhedrons, which are obtain as union of all elementary strata of the corresponding type.
	The anti-diagonal stratum (formally) is  a stratum of the type
	$\I_{b}$, with this convenient the inclusion (\ref{Ib}) is modified into 	
	\begin{eqnarray}\label{anti}
K_{\I_b} \mapsto K_{\I_b;\circ} \cup \Delta_{anti}  \subset \N_{\circ}.
	\end{eqnarray}
		\subsubsection{Coordinates on
		$\N_{\circ}$, the mapping $c$ (\ref{c}) \label{cc}}
	
		Let 
	$x \in  K^{[r-s]} (k_1, \dots, k_s)$ be a point on an elementary stratum.
	Let us consider collections of spherical coordinates 
	$\check x_{1,i}$, $\check x_{2,i}$, $k_1 \le i \le k_s$, $1 \le i \le r_0$ 
	of the point  $x$. For each $i$ the following cases are possible:
	the pairs of $i$-th coordinates coincides; antipodal. 
	Let us define for an ordered pair of coordinates
	$
	\check x_{1, i} $, $ \check x_{2, i} $, $ 1 \le i \le r $  residues
	$v_i$ take values in $\{+1,-1\}$:
	$ v_{k_i} = \check x_{1, k_i}(\check x_{2, k_i})^{-1} $.
	
	By transformation of the collection of coordinates 
	are defined as before in the construction for the mapping $c_0$ (\ref{c0}).
		Elementary strata  $K^{[r-s]} (k_1, \dots, k_s)$ are only of the type
	$\I_{a \times \aa}$. 
	 The following statement is clear, because coordinate system for corresponding polyhedra are explicitly defined.
	 
	 \begin{lemma}\label{cc0}
	 The polyhedron $\N_{\circ}$
	 for the mapping $c$ (\ref{c}) is PL-homeomorphic to the polyhedron (\ref{Ia}) for the mapping $c_0$ (\ref{c0}).
	 \end{lemma}

\subsubsection{A prescribed coordinate system   on an elementary stratum  \label{prescr}}

Let us recall that the polyhedron
$K_{\I_{b};\circ}$ is decomposed into a disjoint union of elementary strata
$K(k_1, \dots, k_s)$, $0
\le s \le r$, residues of angles coordinates take values $\{+\i,-\i\}$.
On each elementary stratum with odd (even) angles
$\alpha$ of the type $\I_{\b}$ let us define a coordinate system, which is called a prescribed coordinate system $\Omega(\alpha)$.

Assume that $s$ is odd. Let us call a coordinate system a prescribed system, if
a number of coordinates with the residues $+\i$ is greater then the number of coordinates with the residues $-\i$. 
In the case the residues are divided in the middle, let us define the prescribed coordinate 
system by the following convenience: the residue with the smallest number equals to $+1$.


An analogous definition for the polyhedron $K_{\I_{a \times \aa};\circ}$ is presented.
Let us call a coordinate system a prescribed system, if
a number of coordinates with the residues $+1$ is greater then the number of coordinates with the residues $-1$. 
In the case the residues are divided in the middle (this is possible only if $s$ is even), let us define the prescribed coordinate system such that the residues with the smallest index equals to $+1$. 




\subsubsection{Admissible boundary strata of elementary strata}

For an arbitrary elementary stratum
$\beta \subset \hat K^{[r-s,i]}(k_1, \dots, k_s)$ of the type
$\I_{\b}$ (correspondingly, of the type $\I_{\a \times \aa}$)
of
$\hat K_{\circ}$ let us consider a smallest stratum in its boundary $\alpha$ of the same type,  $\beta \ne \alpha$, this means that $\alpha$ in inside
the closure of $\beta$:
$\alpha \subset Cl(\beta) \subset Cl(\hat K^{[r-s,i]}(k_1, \dots, k_s))$. 

Prescribed coordinate system 
$\Omega(\beta)$ on the stratum
$\beta \subset \hat K(k_1, \dots, k_s)$ is restricted to the elementary stratum $\alpha \subset Cl( \beta)$ inside its boundary. We shall write:  $\alpha \prec \beta$.
In the case when the restriction of a prescribed coordinate system  $\Omega(\beta) \vert_{\alpha}$ coincide with the prescribed coordinate system  $\Omega(\alpha)$
on the smallest stratum, we shall call that the pair of strata
$(\alpha,\beta)$ is admissible. In the opposite case, the pair of strata $(\alpha,\beta)$ is not admissible.

In the case a pair
$(\alpha,\beta)$ is not admissible, the transformation  $\Omega(\beta) \vert_{\alpha}$ into $\Omega(\alpha)$ are described below for all types of strata.

For pair of strata of the type $\I_{\b}$ a non-admissible of
the pair means that the system of coordinates
$\Omega(\beta) \vert_{\alpha}$ are related with $\Omega(\alpha)$
by one of the following transformations:
\begin{eqnarray}\label{prim1}
	(\check{x}_1,\check{x}_2) \mapsto  (\check{x}_2,\check{x}_1),
\end{eqnarray}
\begin{eqnarray}\label{prim-1}
	(\check{x}_1,\check{x}_2) \mapsto  (-\check{x}_2,-\check{x}_1),
\end{eqnarray}
\begin{eqnarray}\label{prim2}
	(\check{x}_1,\check{x}_2) \mapsto  (-\check{x}_1,\check{x}_2),
\end{eqnarray}
\begin{eqnarray}\label{prim3}
	(\check{x}_1,\check{x}_2) \mapsto  (\check{x}_1,-\check{x}_2).
\end{eqnarray}

\subsection{An additional structure on the polyhedron $K_{\circ}$ of the mapping $c_0$ (\ref{c0})}
\subsubsection{ An involution $\tau_c$ \label{tauc}}

An elementary stratum, except the anti-diagonal stratum, is a 2-sheeted covering with respect to a free involution, denoted by 
$\tau_c$ (see the formula below)
over the corresponding elementary stratum
of a polyhedron $\hat K_{\circ} = K_{\circ}/\tau_c$, 
a stratum of this polyhedron is denoted by  
\begin{displaymath}\label{hatcompstrat}
\hat K^{[r-s]}(k_1, \dots, k_s), \qquad 1 \le s \le r.
\end{displaymath}

Define the involution $\tau_c$ on the space
$K_{\circ}$ by the formula:
\begin{eqnarray}\label{hatcompstrat}
\tau_c:([x,y],\lambda) \mapsto  ([ix,iy],\lambda).
\end{eqnarray}
This involution is free outside the antidiagonal (recall that by the definition of $K_{\circ}$ the antidiagonal is removed, see (\ref{antidiag})).
This statement is followed from the coordinate description of elementary strata of
$K_{\circ}$, an elementary strata is invariant by the involution $\tau_c$ and the quotient is well defined.
	
\begin{lemma}\label{1}
There exists a mapping
$\eta_c$, which is included into the following diagram, where $\D \subset C$ is a central quadratic extension by the element $c \in C \setminus \D$, $c^2=b^2 \in \D$
of the order $4$, vertical mappings are 2-sheeted coverings:
		\xymatrix{
			K_{\circ}\ar@{->}[r]^\eta   \ar@{->}[d]& B\D \ar@{->}[d]\\
		\hat{K}_{\circ}=	K_{\circ}/\tau_c \ar@{->}[r]^{\eta_c} &BC.
		}
	\end{lemma}
	
	\begin{proof}
In each elementary stratum, which consists of points $((x,y),\lambda)$ on the sphere, 
let us define the action of the dihedral group
$\D =\{a,b: [a,b]=b^2; b^4=e\}$,
by the formulas: 
\begin{align*}
a&\colon ((x,y),\lambda)\mapsto ((y,x),\lambda)\\
b&\colon ((x,y),\lambda)\mapsto ((y,-x),\lambda).
\end{align*}
Define the group
$C$ as the quadratic central extension of the group $\D$ by the element $c$,
where $c^2=b^2$. Define the action of $C$, generated by elements $a,b,c$,
by the formulas:
\begin{align*}
a&\colon ((x,y),\lambda)\mapsto ((y,x),\lambda)\\
b&\colon ((x,y),\lambda)\mapsto ((y,-x),\lambda)\\
c&\colon ((x,y),\lambda)\mapsto ((ix,iy),\lambda).
\end{align*}
This action is free and one may define $\hat{K}_{\circ}=K_{\circ}/\tau_{c}$.
The involution 
$\tau_c$, defined above, are given by the action on the element $c \in C \setminus \D$. 
\end{proof}

All the statements above are formulated for elementary strata of the polyhedron
$\hat{K}_{\circ}=K_{\circ}/\tau_{c}$.
Elementary strata of this polyhedron with the involution (\ref{hatcompstrat})  are divided into the following types: $\I_{b,c},\I_{a \times \aa,c},\I_{d,\c}$, respectively to the type of the subgroup in  $C$.  Surprisingly, strata of the type $\I_{b,c}$, 
$\I_{a \times \aa,c}$ are homeomorphic, see Proposition \ref{iso}.
(A homeomorphism in prescribed coordinate systems is given by the formula: 
$\Phi: ([x,y],\lambda) \mapsto ([ix,y],\lambda)$.
Homeomorphism
$\Phi$ induces an isomorphism of the corresponding structuring subgroups $\I_{b,c} \mapsto \I_{a \times \aa,\c}$, which is given on the generators by the formula: $b \mapsto ac$,
$c \mapsto c$.)

Analogously to (\ref{Ib}),(\ref{Ia}),(\ref{Id})
let us define the following open subpolyhedrons:  
\begin{eqnarray}\label{Ibc}
K_{\I_{b,c};\circ}/\tau_c = \hat K_{\I_{b,c}; \circ} \subset \hat{K}_{\circ},
\end{eqnarray}
\begin{eqnarray}\label{Iac}
K_{\I_{a \times \aa,c};\circ}/\tau_c = \hat K_{\I_{a \times \aa, c}; \circ} \subset \hat{K}_{\circ},
\end{eqnarray}
\begin{eqnarray}\label{Idc}
K_{\I_{d,c};\circ}/\tau_c =	\hat K_{\I_{d}; \circ}  \subset \hat{K}_{\circ}.
\end{eqnarray}
The polyhedra
$(\ref{Ibc})$, $(\ref{Iac})$ are homeomorphic. 

Definition of prescribed coordinate system and  admissible boundary strata for the subpolyhedrons (\ref{Ibc}), (\ref{Iac}) is analogous to (\ref{Ib}), (\ref{Ia}). 


\subsubsection{Resolutions  $C_{\Im} \mapsto \I_{b,c}$, $C_{\Re} \mapsto \I_{a \times \aa,c}$\label{r}}

Recall, the group $C$ is generated by  generators $\{a,b,c\}$.  This group  admits two subgroups  $\{b,c\}$, $\{a \times \aa,c\}$ described by the corresponding subcollection of generators. Let us construct epimorphisms $\I_{b,c}$ by 
$ C_{\Im} \longrightarrow \I_{b,c}$, $ C_{\Re} \longrightarrow \I_{a \times \aa, c }$, which are called resolutions.

Let us consider the following diagram:

\begin{equation}\xymatrix{
	K_{b;\circ}\ar@{->}[r]  \ar@{->}[d]& B\Z_4(b) \ar@{->}[d]\\
	K_{\I_b;\circ} \ar@{->}[r] &B\D,  
}\label{112}\end{equation}
where $K_{b;\circ} \to K_{\I_b;\circ} $ is a double covering over the polyhedron (\ref{Ib}), which is described by the subgroup $\{b\} \subset \D$.
this diagram is extended into the right:
\begin{equation}\xymatrix{
	K_{b;\circ}\ar@{->}[r]  \ar@{->}[d]& B\Z_4(b) \ar@{->}[d]\ar@{->}[r]  \ar@{->}[d]& {\rm *}  \ar@{->}[d]\\
	K_{\I_b;\circ} \ar@{->}[r] &B\D \ar@{->}[r] &B \D/\Z_4(b)\ar@{=}[r] & B\Z_2.
}\label{113}\end{equation}
the composition of horizontal mappings classifies the double covering $K_{b;\circ}\to K_{\I_b;\circ}$. The generator of the right-bottom group
$\Z_2$ is the image of an arbitrary element from the non-trivial residue class of the subgroup $\Z_4(b)$ in $B\D$. 
Let us take the element  $ \dot{A}$.

 The space $B\Z = S^1$ admits the positive generator in $\Z$, which is denoted by  $\hat A$.
 Then we have the following commutative square, the vertical mappings are epimorphisms $\hat{A} \mapsto \dot{A}$, the semi-direct product $\D=\Z_4(b)\rtimes \Z_2(\dot A)$ in the bottom left term determines the analogous semi-direct product $\Z_4(b) \rtimes \Z(\hat{A})$ in the upper left term of the diagram:
\eqq
\xymatrix{
 B\Z_4(b)\rtimes \Z(\hat A)\ar@{->}[r]  \ar@{->}[d]& B\Z(\hat A) \ar@{->}[d]\\
B\D=B\Z_4(b)\rtimes \Z_2(\dot A) \ar@{->}[r] &B\Z_2(\dot{A}).
}\label{eq:square1}\eeq

Let us relates the diagram with the structured group of the polyhedron
$\hat{K}_{\I_b;\circ} = K_{\I_{b,c};\circ}/\tau_c$.
For this reason let us remark that there exist diagrams,
which are analogous to~(\ref{112}) and~(\ref{113}):

\begin{equation}\xymatrix{
\hat{K}_{b,c;\circ}=	K_{b,c;\circ}/\tau_c\ar@{->}[r]  \ar@{->}[d]& B\I_{b,c} \ar@{->}[d]\ar@{->}[r]  \ar@{->}[d]& {\rm *}  \ar@{->}[d]\\
\hat{K}_{\I_{b,c};\circ} =	K_{\I_{b,c};\circ} /\tau_c \ar@{->}[r] &BC \ar@{->}[r] &B C/C_1\ar@{=}[r] & B\Z_2
}\label{eq:lift44}\end{equation}
moreover, the compositions of the horizontal bottom mappings is classified the left vertical double covering.
Recall, in the diagram above  $\I_{b,c}$ is a subgroup in  $C$, which is generated by elements $b$ and $c$. It is convenient
to write-down this subgroup as 
$\Z_4(b)\times \Z_2(cb^3)$. Then
$C$ can be represented as a semi-direct product $\Z_4(b)\times \Z_2(cb^3)\rtimes_{\theta_\Im}\Z_2(\dot A)$,
where the action $\theta_\Im$ on  $\Z_4(b)\times \Z_2(cb^3)$ is defined by the formula: $(x,y)\mapsto (-x+2y,y)$.
This fact is followed from the formulas:
\begin{align}
&\dot A \cdot b = b^3\cdot \dot A\\
&\dot A \cdot b^3c  = bc \cdot \dot A = b^2\cdot b^3c\cdot \dot A
\end{align}
Let us denote by $ C_\Im$ the group
$\Z_4(b)\times \Z_2(cb^3)\rtimes_{\theta_\Im}\Z(\dot A)$,
then we get the following Cartesian square:

\eqq\xymatrix{
B C_\Im\ar@{->}[r]  \ar@{->}[d]& B\Z(\dot A) \ar@{->}[d]\\
BC \ar@{->}[r] &B\Z_2(\dot A)
}\label{eq:square2}
\eeq

In Theorem \ref{t11} it is proved that the mapping $\hat{K}_{\I_{b,c};\circ} \to S^1=B(\dot A)$
exists. As a corollary we get the following theorem,  called the resolution construction.

\begin{theorem}\label{resol}
There exists a mapping (a resolution)
of $\hat{K}_{\I_{b,c};\circ} = K_{\I_{b,c};\circ}/\tau_c$ of the bottom left space of the diagram (\ref{eq:lift44}) into 
the space $BC_{Im}$ in the upper left vertex of the diagram (\ref{eq:square2}), such that 
this mapping commutes with the left vertical projection $BC_{Im} \longrightarrow BC$ in the diagram. 
\end{theorem}

Let us describe analogous diagrams to the diagrams  (\ref{eq:lift64}),(\ref{eq:square2}) for the space
 $\hat{K}_{\I_{a \times \aa};\circ}$ instead of $\hat{K}_{\I_{b};\circ}$.  
 \begin{equation}\xymatrix{
 		\hat{K}_{a \times \aa;\circ}=	K_{a \times \aa,c;\circ}/\tau_c\ar@{->}[r]  \ar@{->}[d]& B\I_{a \times \aa,c} \ar@{->}[d]\ar@{->}[r]  \ar@{->}[d]& {\rm *}  \ar@{->}[d]\\
 		\hat{K}_{\I_{a \times \aa};\circ} =	K_{\I_{a \times \aa,c};\circ} /\tau_c \ar@{->}[r] &BC \ar@{->}[r] &B C/C_1\ar@{=}[r] & B\Z_2
 	}\label{eq:lift64}\end{equation}
 
Let us denote by $ C_\Re$ the group
$\Z_2(a) \times  \Z_4(c\dot{a})\rtimes_{\theta_\Re}\Z(\dot A)$,
then we get the following Cartesian square:

\eqq\xymatrix{
	B C_\Re\ar@{->}[r]  \ar@{->}[d]& B\Z(\dot A) \ar@{->}[d]\\
	BC \ar@{->}[r] &B\Z_2(\dot A)
}\label{eq:square3}
\eeq
 
The following theorem is a direct analogue of Theorem \ref{resol} for the group $C_{Re}$. 
\begin{theorem}\label{resol2}
	There exists a mapping (a resolution)
	of $\hat{K}_{\I_{a \times \aa};\circ} = K_{\I_{a \times \aa,c};\circ}/\tau_c$ of the bottom left space of the diagram (\ref{eq:lift64}) into 
	the space $BC_{Re}$ in the upper left vertex of the diagram (\ref{eq:square3}), such that 
	this mapping commutes with the left vertical projection $BC_{Re} \longrightarrow BC$ in the diagram. 
\end{theorem} 
 
Let us prove that Theorem \ref{resol2} as a corollary of Theorem \ref{resol}, let us use the following proposition.

\begin{proposition}\label{iso}
The mapping $\Phi: \hat{K}_{\I_b \circ}= K_{\I_{b,c} \circ}/\tau_c \to \hat{K}_{\I_{a \times \aa};\circ} = K_{\I_{a \times \aa,c};\circ}/\tau_c$ between the upper left polyhedron in the diagrams (\ref{eq:square2}), (\ref{eq:square3}),
which is given by the formula $\Phi: ([x,y],\lambda)\mapsto  ([ix,y],\lambda)$, is a homeomorphism.
\end{proposition}

\begin{proof}
A correctness of the definition of $\Phi$ and existence of an inverse mapping is straightforward exercise. It is convenient 
to note that the formula for
$\Phi$ gives a well defined mapping only after a factorization by $\tau_c$.
\end{proof}

\subsubsection{Proof of Theorem \ref{resol2}}
By Theorem \ref{resol} we have the mapping 
\eqq
\hat{K}_{\I_b;\circ}\to B\Z(\dot{A})=S^1,
\eeq
which is defined by the resolution.
\label{key}
Thus, we get a mapping:
\eqq
\hat{K}_{\I_{a \times \aa};\circ}\to S^1,
\eeq
which is given by the composition:
\begin{eqnarray}\label{eeq222}
 \hat{K}_{\I_{a \times \aa,c};\circ}/\tau_c\stackrel{\Phi^{-1}}{\to}  \hat{K}_{\I_{b,c};\circ}/\tau_c \to S^1.
\end{eqnarray}

We have the following commutative diagram:
\eqq
\xymatrix{
\hat{K}_{\I_b;\circ} \ar@{->}[r]   \ar@{->}[d]_\Phi& BC \ar@{->}[d]^{\Psi}\\
\hat{K}_{\I_{a \times \aa};\circ} \ar@{->}[r]^{\eta_c} &BC
}\eeq
where the mapping $\Psi$ is induced by the corresponding automorphism of the group $C$, which will be denoted by the same:
\begin{align*}
\Psi&\colon C\to C\\
\Psi&\colon a\mapsto bc\\
\Psi&\colon b\mapsto ac\\
\Psi&\colon c\mapsto c
\end{align*}

Then there exist a commutative diagram:
\eqq
\xymatrix{
\hat{K}_{\I_{a \times \aa};\circ} \ar@{->}[r]^{\Phi^{-1}}   \ar@{->}[d]& \hat{K}_{\I_b;\circ} \ar@{->}[r]\ar@{->}[d]& B\Z(\hat A)  \ar@{->}[d]\\
BC\ar@{->}[r]^{\Psi^{-1}} &BC \ar@{->}[r] & B\Z_2(\dot A)
}\eeq
The kernel of the composition of the bottom mappings
coincides with 
$\Psi^{-1}(C_{Im})$, this kernel is generated by elements: 
$c$ and $ac^3$, this can be represented as a semi-direct product  $\Z_4(c)\times\Z_2(\dot a)$. The group
 $C$ itself is represented by a semi-direct product: 
$\Z_4(c)\times\Z_2(\dot a)\rtimes_{\theta_\Re} \Z_2(\dot A)$, where the action of the generator  $\dot A$ on
$\Z_4(c)\times\Z_2(\dot a)$ is well defined by the formula: $(x,y)\mapsto (x+2y,y)$.
This follows from the formula:
\begin{align}
&\dot A \cdot c = c\cdot \dot A\\
&\dot A \cdot \dot a  = \dot a \cdot c^2 \cdot \dot A 
\end{align}

Therefore the classifying mapping
$\hat{K}_{\I_{a\times\aa};\circ}\to BC$ is lifted into $BC_\Re = B\Z_4(c)\times\Z_2(\dot a)\rtimes_{\theta_\Re} \Z(\dot A)$.
 Proposition \ref{iso} is proved. \qed

\subsection{The polyhedron $\hat K_{\I_b;\circ}$, the resolution of its structuring mapping \label{resolution}}

Let us consider the polyhedron
$\hat K_{\I_b;\circ}$, whit coordinate system of the type $\I_{b,c}$, for short re-denote this polyhedron by
$\hat K$.
Consider in
$ C$ subgroups, which are generated by the prescribed elements:

\begin{eqnarray}\label{sq1}
\xymatrix{
\Z_4(b^2,c) \ar@{^{(}->}[rr]   \ar@{^{(}->}[dd]\ar@{^{(}->}[rd]&& \Z_2 \times \Z_4(b,c)\ar@{^{(}->}[dd]\\
& \Z_2\times \Z_4(A,\dot A,c) \ar@{^{(}->}[rd]\\
\Z_2 \times \Z_4(a,\dot a,c) \ar@{^{(}->}[rr] &&C=(a,b,c)
}
\end{eqnarray}
In this diagram all mappings are induced by inclusions of corresponding subgroups
of the index $2$. The datagram admits the following 2-sheeted covering diagram,
this covering diagram consists itself of 2-sheeted coverings.

\begin{eqnarray}\label{sq2}
\xymatrix{
\hat K_{b^2}\ar@{->}[rr]   \ar@{->}[dd]\ar@{->}[rd]&& \hat K_{b}\ar@{->}[dd]\\
& \overline K  \ar@{->}[rd]\\
\hat K_{a \dot a} \ar@{->}[rr] &&\hat K
}
\end{eqnarray}

In this section points of the space
$\hat{K}_{b^2}=K_{b^2,c}$, are denoted by   $[(x,y),\lambda]$, and the momentum
coordinate $\lambda$ should be omitted.
In the space 
$\hat K_{b^2}$ a point  $(x,y)$ coincides with points   $(\i x,\i y)$,$(-x,-y)$,$(-\i x,-\i y)$.

\subsection{Construction of equivariant morphism of  a bundle over $\hat K_{b^2}$ with layers  non-ordered pairs of oriented real line into a bundle with layers $\C$}

Let us divide the space
$\hat K_{b^2}$ into two spaces $\hat K_{b^2,R}$, $\hat K_{b^2,L}$ with the common
boundary:
\begin{eqnarray}\label{RL}
\hat{K}_{b^2} = \hat K_{b^2,R} \cup_{\Sigma} \hat K_{b^2,L}. 
\end{eqnarray}

This common boundary let us denote by
$\hat K_{b^2,R} \cap \hat K_{b^2,L} = \Sigma$. Define the subspaces in all elementary strata separately. Let us consider the momenta coordinate
$\lambda$, which takes value in the simplex
$\Delta^{r-j}$, where $j$ is the deep of the strata.  Let us consider two faces of simplex $\Delta^{r-j}$, which are defined by momenta, corresponded to angles with positive and negative imaginary residues. The simplex itself is the join of
this two faces, let us assume that a parameter of the join $t \in [0,1]$ 
takes the value $0$ on the negative face and the value $1$ to the positive face.

Let as call a simplex $\Delta^{r-j}$  regular, if   numbers of positive and negative coordinates are different. In the case the numbers of positive and negative coordinates coincide, a simplex is called  exceptional.  In particular, in the case $r-j$ is odd, a corresponding simplex is regular. For a regular simplex one may define a type of the simplex. Let us call a simplex is right (a R-type simplex), if the number of positive coordinates is greater then the number of negative coordinates. Let us call a simplex is left (a L-type simplex), if the number of positive coordinates  is less then the number of negative coordinates. A type of an exceptional simplex is not defined.

Let us define the decomposition (\ref{RL}) for a subpolyhedron $\hat{K}_{b^2}^1 \subset \hat{K}_{b^2}$.

\begin{definition}\label{st}
	Let us define the polyhedron $\hat{K}_{b^2}^1$, which consists of all strata in $\hat{K}_{b^2}$
	of the deep $0$ and $1$.
	Let us define the following $\tau_a$-involutive disjoint decomposition for maximal strata:
	\begin{eqnarray}\label{s,t}
	\hat{K}_{b^2}^0 =	\hat{K}_{b^2,R}^0 \cup \hat{K}_{b^2,L}^0.
	\end{eqnarray}
	We may extend this decomposition the decomposition (\ref{s,t}) and define the $\tau_a$-involutive polyhedrons:
	\begin{eqnarray}\label{s,t1}
	\hat{K}_{b^2}^1 =	\hat{K}_{b^2,R}^1 \cup_{\Sigma} \hat{K}_{b^2,L}^1.
	\end{eqnarray}

In (\ref{s,t1}) the common boundary polyhedron $\Sigma$ consists of all exceptional strata of the deep $1$.
Obviously, one may define the decomposition  (\ref{RL}), such that
	the decompositions (\ref{s,t}), (\ref{s,t1}) are restrictions  to the corresponding subpolyhedra in the left-hand side.
	\end{definition}


Over the space $\hat K_{b^2}^1 = \hat K_{b^2,t}^1 \cup \hat K_{b^2,s}^1$ let us define a pair of bundles equipped with involutions  $\tau_a,
\tau_b,\tau_i$ and an equivariant morphism between them. 
On the space, the base of the bundles, the involutions $\tau_a,
\tau_b$ are defined as the involutions on the square (\ref{sq2}), using the square of the group (\ref{sq1}).
The involution
$\tau_a$ changes the subspace  $\hat{K}_{b^2,t}^1$ into $\hat{K}_{b^2,s}^1$ in (\ref{s,t1}),  this involution translate a prescribed angle into its conjugated. 
The involution $\tau_i$ on the base is the identity.

We will use the involutions on the space (\ref{s,t1}), we assume that the involutions are extended by the involution in the source bundle $E_{\R}$.
A fiber of $E_\R$ is a disjoin union of the two real lines  $\R \sqcup \R$: $E_\R=\hat K_{b^2}^1 \times (\R \sqcup \R)$. The unit vectors in the fibers $\R \sqcup \R$ over $\hat K_{b^2,t}^1$ is denoted by $\pm t_+$,  
$\pm t_-$.  The basis vectors
$+1$ and $-1$ determine a decomposition of  the layers into two rays. 

For a $R$-stratum the lower superscript  $\pm$ of a line, which will be used in the formulas,
is associated with the inner positive-prescribed (negative-non-prescribed)  angles collection. For a $L$-stratum tht lower superscript $\pm$ of lines is associated with inner right-prescribed  (left-prescribed) angles in a $\tilde{\tau}_a$-conjugated $R$-stratum (recall the corresponding $\tilde{\tau}_a$-conjugated left $L$-angle to a right $R$-angle is collected with the opposite sign with respect to
the  angle in the $L$-stratum itself).
Geometrically, when we pass by a path from a maximal stratum $\alpha_1$ in $\hat K_{b^2,R}^0$ to a neighbouring
maximal stratum $\alpha_2$ in $\hat K_{b^2,L}^0$ trough a common exceptional stratum of deep $1$, the fibre $t_+$ is
transformed to $t_-$; because in a $t$-stratum $\alpha_1$ the fibre $t_+$ is associated with  inner right (prescribed) residue in the stratum; the extension of the same fibre in the neighbouring $s$-stratum $\alpha_2$ is denoted by $t_-$, because this fibre is associated with the inner left (non-prescribed) residue in the $\tau_a$-conjugated stratum.

Let us describe the target bundle of the morphism.
This bundle denote by 
$E_{\C}$, without equivariant $\tau_a$-structure this bundle is trivial with the complex line layer: $E_{\C} = \hat{K}_{b^2}^1 \times \C$. Assume that the trivialization of the bundle
$E_{\C}$ is described by the variable $z$ over  $\hat K_{b^2,R}^1$. Over $\hat K_{b^2,L}^1$ the trivialization  is given by the coordinate  $z$ denoted the same.
Let us describe equivariant structures on the constructed bundles.

Let us define $3$ involutive transformations of
$E_\R$ (the involution $\tau_a$ changes the component in (\ref{s,t1}) ) by the following formula:
\begin{align}\label{coort}
	\tau_a&:((x,y),t_{\pm})\mapsto  ((y,x), t_{\pm})\\
	\tau_b&:((x,y),t_{\pm})\mapsto  ((-y,x),t_{\pm})\\
	\tau_i&:((x,y),t_{\pm})\mapsto ((x,y),-t_{\mp}) 
\end{align}

By this formula 
$\tau_b$ is given by the trivial transformation of layers over the base $\hat{K}_{b^2}$. The involution
$\tau_i$ acts on the base  $\hat{K}_{b^2}$ by the identity, this transformation translates in the target bundle $E_{\R}$ lower superscripts and reveres directions of the layers. 

The action of the corresponding involutions on
$E_\C$ in the $t$-system is given by the formulas:
\begin{align}\label{tauz}
	\tau_a&:((x,y),z)\mapsto((y,x),\bar{z})\\
	\tau_b&:((x,y),z)\mapsto  ((-y,x),z)\\
	\tau_i&:((x,y),z)\mapsto ((x,y),\overline{z}).
\end{align}

The involutions defined above are pairwise commuted. 
The transformation  $\tau_a$  changes the components in (\ref{s,t}) and we may prove that
the formulas corresponds each other with respect to the identification of the fibres in a common point on $\hat{K}^1_{b^2}$. Two fibres of the bundle $E_{\C}$ are transformed by $z \mapsto z$, because the bundle $E_{\C}$ is non-trivial (is $\tau_a$-twisted) and the involution $\tau_{a}$ changes the $z$ coordinate. The transformation $\tau_a$  changes  lines in the fibre of the bundle $E_{\R}$. Because a line of $E_{\R}$ is equipped with different lower superscript, which are corresponded with respect to $\tau_a$, and the formulas for $\tau_a$ preserves the lower superscript of lines, geometrically we get the transposition of lines by $\tau_a$.

The result of the section is the following.

\begin{theorem}\label{t10}
	There exist 
	$(\tau_a,\tau_b,\tau_i)$-equivariant mapping $F$ 
	of the bundle $E_\R$ into the bundle $E_\C$, which is fibrewized monomorphism.
\end{theorem}

\begin{proof}
	The mapping $F$ is constructed by the induction over the codimension of simplex momenta. More precisely, let us start the construction of $F$ over centres of maximal simplexes, where the total collection of residues are defined.
	Over each maximal simplex in
	$\hat K_{b^2}$ this mapping is
	denote by $F^{(0)}$. Then let us extend the mapping on the segment, which join a pairs of centers of the simplexes
	and cross a codimension 1 face, or which join pairs of points
	of $\hat K_{b^2,R}$ and $\hat K_{b^2,L}$ in a common strata. 
	This extension is denoted by $F^{(1)}$.
	
	Additionally, to conclude the proof, one has to check that obstructions for equivariant extension of the mapping  
	$F^{(1)}$ on $\hat{K}_{b^2}^{(1)}$ to the hole complex
	$\hat K_{b^2}$ is trivial.

	\begin{lemma}\label{l10}
	1.	There exists a
		$(\tau_a,\tau_b,\tau_i)$-equivariant mapping 
		$F^{(0)}$ on simplexes of the complex in $\hat K_{b^2}$ of the deep $0$.
		\[  \]
		
		2. The equivariant morphism constructed by 1. is extended on the complex in $\hat K_{b^2}$ of the deeps $0$ and $1$.
	\end{lemma}

		\subsubsection*{Total $\theta$-residue,  inner-integer $\Theta$-residue}

	The base of the induction is the construction of the mapping
	$F^{(0)}$, the restriction of $F$ on the subspace, which consists of points $((x,y),\lambda)\in \hat K_{b^2}$, 
	for which all momenta coordinate $\lambda_j\ne 0$.
	Assume we get the case $\hat{K}_{b^2,R}^0$, the residues collection 
	$\{\delta_j\}$ exists for all numbers $j$ of coordinates. 
	Define in a $R$-stratum $\alpha_1$ the total residue by 
	\begin{eqnarray}\label{totalres}
	 \theta=\prod\limits_{j=1}^r\delta_j. 
	 \end{eqnarray}
	We associate this residue with the opposite base vector $-1 \in t_+$ on the line. 
	The number $r$  is odd, therefore $\theta$ is an imaginary. The conjugated total resedue is
	assotiated with the the base vector $+1 \in t_-$.

		Assume we get the case $\hat{K}_{b^2,L}^0$. 
	Define in a $L$-stratum $\alpha_2$ the total residue by the same formula (\ref{totalres}).
	We associate this residue with the opposite base vector $+1 \in t_+$ on the line. 
	The number $r$  is odd, therefore $\theta$ is an imaginary. The conjugated total residue is
	associated with the the base vector $-1 \in t_-$. 
	By the construction the total residue is 
	$\tau_a$-skew: 
	\begin{eqnarray}\label{skew}
	\bar{\theta}_{\alpha_1}=\overline{\prod\limits_{j=1}^r\delta_j\vert_{\tau_a(\alpha_2)}} = \prod\limits_{j=1}^r\delta_j\vert_{\alpha_2}=\theta_{\alpha_2}, \quad \alpha_1 = \tau_a(\alpha_2).
	\end{eqnarray}

	The following sums for a $R$-stratum $\alpha_1$:
	\begin{eqnarray}\label{Theta1}
	 \Theta(\alpha_1)=\sum\limits_{j=1}^r\delta_j  \in \i\Z
	 \end{eqnarray}
	  are integer lifts of $\theta$, let us call the sum the inner-integer  resedue. 
	
			For an $L$-stratum $\alpha_2$ the following sums: 
	\begin{eqnarray}\label{Theta2}
		 \Theta(\alpha_2)=\sum\limits_{j=1}^r(-\delta'_j) \in \i\Z.
	\end{eqnarray}		
	are integer lifts of the conjugated total residue $\bar{\theta}$, let us call the sum the inner-integer
	 residue. 

	The integer residue $\Theta$ is  $\tau_a$-symmetric:  we get for the corresponding two points $\x \in \alpha_1, \tau_a(\x) \in \tau_a(\alpha_1)$:
	$$\sum\limits_{j=1}^r\delta_j \vert_{\x}=\sum\limits_{j=1}^r(-\delta'_j) \vert_{\tilde{\tau}_a(\x)}.$$

	\subsubsection{A formula of morphism $F^{(0)}$, Lemma \ref{l10},1}
	
	In a maximal stratum of R-type
	on $\hat{K}_{b^2}$ let us define the morphism $F^{(0)}$ by the formula:
	\begin{equation}\label{1.1+}
		F^{(0)}: ((x,y),t_{+}) \mapsto ((x,y), -\theta  t).
	\end{equation}
	\begin{equation}\label{1.1-}
		F^{(0)}: ((x,y), t_{-}) \mapsto ((x,y), \bar{\theta} t).
	\end{equation}

	One may see, the first formula (\ref{1.1+}) defines the image of the vector   $-1 \in t_+$ by the value  $\theta$, this implies the minus in the right-hand side of the formula. The second formula determines the image of the base vector on the line $t_-$ by 
	the conjugated residue  $\bar{\theta}$. 
	
	Over a maximal stratum  of 
	$K_{b^2,L}$ the morphism $F^{(0)}$ is defined by the following formula:
	\begin{equation}\label{1.2+}
		F^{(0)}: ((x,y),t_{+}) \mapsto ((x,y), \bar{\theta}  t).
	\end{equation}
	\begin{equation}\label{1.2-}
		F^{(0)}: ((x,y), t_{-}) \mapsto ((x,y), -\theta t).
	\end{equation}
	
	We may assume that the first formula
	 (\ref{1.1+}) determines the image of the base vector  $+1 \in t_+$ by $\theta'=\bar{\theta}$, see the formula (\ref{skew}). 
	 Therefore the image of the base vector
	  $-1 \in t_+$ is defined by  $-\theta' = -\bar{\theta}$. 
	  The second formula defines the image of the  vector $-1 \in t_-$ by $\theta$. The image of the base vector  $+1 \in t_-$ is defined by  $-\theta$.

	The formulas (\ref{1.1+}), (\ref{1.1-}), (\ref{1.2+}), (\ref{1.2-}) are defined
	the morphism on the collection of maximal strata in the source. We check below that
	this morphism is $(\tau_a,\tau_b,\tau_i)$-equivariant.

\subsubsection{An interpolation of the morphism 	$F^{(0)}$ between pairs of maximal strata of a different type}
	
	Let us, briefly, formulate an interpolation rool of the constructed morphism on segments, which join
	two strata of a different type. This is the main point in the construction.

Let us consider a maximal $R$-stratum
$\alpha_1$ and  a neighbouring $L$-stratum  $\alpha_2$, take a stratum of the deep $1$ in $K^{(1)}$ between them. The total residue 
$\theta_1$ of the stratum $\alpha_1$ is well-defined, it takes value $+\i$. 

By an interpolation of the total residue we assume
the homotopy of $\Theta(\alpha_1)$ (\ref{Theta1}) into $-\Theta(\alpha_2)$ (\ref{Theta2}) 
along the segment, which joins the two maximal strata on different components
(\ref{s,t1}). We have the following formula: 
\begin{eqnarray}\label{relation12}
	\theta_1 = \bar{\theta}_2; \quad \Theta(\alpha_1) =\Theta(\alpha_2),
\end{eqnarray}
where $\Theta_1=\Theta(\alpha_1), \Theta_2 =\Theta(\alpha_2)$ for short.

To calculate 
$\theta_2$ let us consider a prescribed collection (see subsection \ref{prescr}) of residues in the $R$-stratum $\alpha_1$. Let us denote this collection by  $\{\delta_j\}_{\alpha_1}$. 
When we pass into a neighbouring $L$-stratum $\alpha_2$, we may consider a corresponding collection: $\{\delta'_j\}_{\alpha_2}$, which coincides with the conjugation of the collection  $\{\bar{\delta}_j\}_{\alpha_1}$ for all residues, but the only exception, which is deformed. 
The product $\theta_2=\prod\limits_{j=1}^r\bar{\delta'}_j$ в $\alpha_2$ coincides with   $\bar{\theta}_1=\prod\limits_{j=1}^r\bar{\delta}_j$ in $\alpha_1$, this proves the formula (\ref{relation12}). 

The secon formula is followed from   (\ref{Theta2}): the sum
$\sum\limits_{j=1}^r(-\delta'_j)$ in $\alpha_2$ coinsids with the sum 
$\sum\limits_{j=1}^r\delta_j$ in $\alpha_1$.

\begin{definition}\label{3pmod}
	
	Recall, $r\ge 3$, $\frac{n-k+2}{2}=r$,  $r \equiv-1\pmod{3}$. Assume the angles, which are corresponded to vertexes of the momenta simplex, are ordered.
	Take
	\begin{eqnarray}\label{bet}
	\beta \subset 	
	K_{b^2,R}^{(1)} \cap K_{b^2,L}^{(1)}
	\end{eqnarray}
	and arbitrary stratum of the deep $1$. Such a stratum has an odd 
	number  $\frac{r-1}{2}$ of angles with residues 
	  $+\i$ and the same odd number of angles with residues  $-\i$.  
	  Let us define a positive coorientation  
	  $\beta^{\uparrow}$ on $\beta$ by the direction from $R$-maximal stratum to $L$-maximal stratum, the maximal strata are in the boundary of $\beta$.
	  The inversion of the coorientation
	$\beta^{\uparrow} \mapsto \beta^{\downarrow}$ corresponds to the direction from $L$-stratum to $R$-stratum on the stratum $\beta$. Let us define a $\beta$-cooriented disorder  $\omega(\beta^{\uparrow}) \in \{0,1\}$ as a number  $\pmod{2}$ of transpositions of numbers of different  residues to get the total ordered partition of residues: the left group of   $\frac{r-1}{2}$ coordinates with residues $-\i$ and the right group of the last  $\frac{r-1}{2}$ coordinates
	with residues $+\i$.
	
	When we change the coorientation on $\beta$, the disorder
  $\omega(\beta^{\downarrow}) \in \{0,1\}$ is defined as the number  $\pmod{2}$
  of transpositions of coordinates with different residues to get the collection 
 of $\frac{r-1}{2}$ coordinates with residues $+\i$ on the left and the collection of 
 $\frac{r-1}{2}$ coordinates with residues $-\i$ on the right.		Obviously, we get:
 \begin{eqnarray}\label{omeg}
 	 \omega(\beta^{\uparrow})+\omega(\beta^{\downarrow})=1. 	
\end{eqnarray} 	 
\end{definition}

Assume the pair of strata $\alpha_1$, $\alpha_2$ are neighbouring strata in a common strata $\beta$ of the deep $1$, which has the disorder  $\omega(\beta^{\uparrow})=0$. 
Let us define a deformation of 
 $\Theta_1$ for $R$-stratum $\alpha_1$ into $-\Theta_2$ for $s$-stratum $\alpha_2$ along
 the segment $\beta$, the deformation defines the image of the line  $t_+$ into the coordinate plane $E_{\C}$. The deformation is the linear interpolation of the corresponding non-regular imaginary term (a positive elementary imaginary residue on $R$-stratum $\alpha_1$) in additive total residue $\Theta_1$ (\ref{Theta1}) into the opposite $-\Theta_2$ to  (\ref{Theta2}) (the corresponding negative elementary imaginary residue with the opposite sign in the formula   (\ref{Theta2})).
 Geometrically, over a point of a stratum of the deep $1$ we get a clockwise rotation of the vector 
$-1 \in t_+$ through the angle $\frac{\pi}{2}$
(algebraically, through the angle $-\frac{\pi}{2}$). 
In the case $\omega(\beta^{\uparrow})=1$ we get a contr-clockwise rotation of the vector $-1 \in t_+$ through the angle $\frac{\pi}{2}$.

We may define the back deformation from $L$-stratum to $R$-stratum we may use the previous
definition to define the back deformation. But, in fact, we may use the following alternative definition of the deformation, which gives the same.
Because in the coorientation on $\beta$ from $\alpha_2$ to $\alpha_1$ we have: $\omega(\beta^{\downarrow})=1$. This means that we get the  deformation of the total residue
 $\Theta_2$ for $L$-stratum $\alpha_2$ into the opposite residue $-\Theta_1$ for $R$-stratum $\alpha_1$ along the segment. This deformation corresponds to contr-clockwise rotation of the vector
 $-1 \in t_+$ trough the angle $-\frac{\pi}{2}$ on $E_{\C}$.
In the case $\omega(\beta^{\uparrow})=1$ we get a clockwise rotation of the vector $-1 \in t_+$ trought the angle $\frac{\pi}{2}$.

The  deformation from $\alpha_1$ to $\alpha_2$ geometrically coincides with the deformation fro $\alpha_2$ to $\alpha_1$, the deformation determines by the cooriented disorder of the strata $\beta$.

Below a key remark to prove that the interpolation an equivariant  is the following.
The involution $\tau_a$ translate the interpolation into itself. 
The involution 
 $\tau_a$ keeps the cooriented disorder: $\omega(\beta^{\uparrow}) = \omega(\tau_a(\beta)^{\downarrow})$ 
 and also keeps additive total residues. The same way, this involution changes 
 the line 
$t_+$ into $t_-$ and conjugate $z$-coordinate in the image  $E_{\C}$. 
This means the involution $\tau_a$ keeps directions of rotations of source lines in the target plane.

	\subsubsection{Proof of Lemma \ref{l10}} 
	
	Let us check that the mapping defined above by (\ref{1.1+}), (\ref{1.1-}) is
	$(\tau_a,\tau_b,\tau_i)$-equivariant.

	In fact, for the base vectors $t=+1$: 
	\begin{align*}
		&F^{(0)} \tau_i((x,y),t_+) = F^{(0)}((x,y),-t_-) = ((x,y),-\overline{\theta} )\\
		& \tau_i F^{(0)} ((x,y),t_+) = \tau_i((x,y),-\theta t) = ((x,y),-\overline{\theta} )
	\end{align*}

	\begin{align*}
		&F^{(0)} \tau_a ((x,y),t_+) = F^{(0)}((y,x),t_+) = ((y,x), -\bar{\theta} t)\\
		& \tau_a F^{(0)} ((x,y),t_+) = \tau_a((x,y),-\theta t) = ((y,x),- \bar{\theta} t)
	\end{align*}

	\begin{align*}
		&F^{(0)} \tau_b ((x,y),t_+) = F^{(0)}((-y,x),t_+) = ((-y,x),-\theta )\\
		& \tau_b F^{(0)} ((x,y),t_+) = \tau_b((x,y), -\theta t) = ((-y,x),-\theta )
	\end{align*}

	The third formula for  $\tau_b$ is obvious, the only base point  is translated inside an elementary strata. In the second formula in the right-hand side we get a point over $L$-strata, in the right hand-side we get a point over $R$-strata. The first formula is written for $R$-strata, the case of $L$-strata is analogous.

	Lemma \ref{l10}.1 is proved. \qed
	\[  \]

	Let us extend the mapping
	$F^{(0)}$ on segments in the origin $E_\R$ over $\hat K_{b^2}$, which join centres of maximal simplexes, the centres in the common strata  $\hat K_{b^2,R}$, or, in a common strata $\hat K_{b^2,L}$ (Case $I$).

	Let us consider two maximal strata $\alpha_1$, $\alpha_2$ in
	$\hat K_{b^2,R}$, which are distinguished by an only  non-degenerated coordinates  $j$. 
	Let us denote by  $\Theta_1$ total residue for $\alpha_1$ and by $\Theta_2$ the total residue for $\alpha_2$.
	Then for $R$-strata we get: $\theta'=\exp(\pi\hat\theta' /2)$,  $\theta''=\exp(\pi\hat\theta''/2)$.

	Let us consider a segment of the homotopy, which is inside 
	$\hat K_{b^2,R}$ (the case inside $\hat K_{b^2,s}$  is analogous),
	this is a subcase of Case $I$). On the segment which joins centres  $((x',y'),(y',x'))$, $((x'',y''),(y'',x''))$ of the first and second simplexes correspondingly, let us denote $F^{(1)}$ by the following formula, for the base vector on the line $t_+$ in the target: 
	
	\begin{align}\label{2.1+} 
	F^{(1)}:&((1-r)(x',y')+r(x'',y''),t_{+})\mapsto  \notag\\
	&\mapsto((1-r)(x',y')+r(x'',y''),-\exp(   ((1-r)\Theta_1+r(-\Theta_2))\pi/2     )  t)
\end{align}  
For $t_-$ the formula is following: 
\begin{align}\label{2.1-}
	F^{(1)}:&((1-r)(x',y')+r(x'',y''),t_{-})\mapsto  \notag\\
	&\mapsto((1-r)(x',y')+r(x'',y''),\overline{\exp(   ((1-r)\Theta_1+r(-\Theta_2))\pi/2     )}  t)
\end{align}

Let us consider a case, when a segment of the homotopy $F^{(1)}$ joins strata in
$\hat K_{b^2,R}$ with $\hat K_{b^2,L}$ 
(Case $II$). In this case the formula for $F^{(1)}$ was described above using (\ref{relation12}). 
A lift  $\Theta$ into the imaginary axis of the total residues $\theta$, which describes the target of the base vector $-1 \in t_+$ in $\alpha_1 \subset \hat K_{b^2,R}$ (subase $II;-t_+$).  The target of the  base vector of $-t_-$ in $[\alpha_1,\alpha_2]$ is the result of an extrapolation with the boundary condition $[\Theta = +\i,-\Theta = -\i]$, the side of the circle is defined by the formula for the differential: 
\begin{eqnarray}\label{dF}
dF(-t_+) = -\i (-1)^{\omega(\beta^{\uparrow})}.
\end{eqnarray}  

The image of the base vector  $+1 \in t_-$ on the second line is defined by the conjugated formula: $dF(t_-)=+\i(-1)^{\omega(\beta^{\uparrow})}$ (subcase $II;t_-$).

	Let us check that the formula gives
	$(\tau_a,\tau_b,\tau_i)$-equivariant mapping. Let us assume that the $R$-component in the source of $\tau_a$ and
	$L$-component in the target of   $\tau_a$ are considered (Case I). 
	
	Namely, for the base vectors $-1 \in t_+$ over $R$-strata is transformed  over $L$-stratum by the formula:
\begin{align*}
	F^{(1)}\tau_a:&((1-r)(x',y')+r(x'',y''),-t_+)=
	F^{(1)}((1-r)(y',x')+r(y'',x''),-t_-)\mapsto\\
	&\mapsto((1-r)(y',x')+r(y'',x''), \exp(((1-r)\overline{\Theta}_1+r\overline{\Theta}_2)\pi/2    )  ).
\end{align*}
Also we get:
\begin{align*}
	\tau_a F^{(1)}:&((1-r)(x',y')+r(x'',y''),-t_+)\mapsto\\
	&\mapsto\tau_a((1-r)(x',y')+r(x'',y''),\exp(   ((1-r)\Theta_1+r(-\Theta_2))\pi/2     )  )=\\
	&=((1-r)(y',x')+r(y'',x''),\overline{\exp(   ((1-r)\Theta_1+r(-\Theta_2))}\pi/2     )  ).
\end{align*}
	 Because
	$\overline{\exp(z)}=\exp(\overline{z})$,  the 
	$\tau_a$-equivariance  in Case $I$ is proved.

		In the subcase $II;-t_+$ the line   $t_+$   is deformed over  $\alpha_1 \subset K_{b^2,t}$
	and $\alpha_2 
	\subset K_{b^2,s}$, the deformation is defined by the target of the vector $-1 \in t_+$ with the condition  $d(F(-t_+)) = -\i(-1)^{\omega(\beta^{\uparrow})}$ over $[\alpha_1,\alpha_2]$. This homotopy is given by the interpolation   $+\i = \Theta_1$ into $-\i = -\Theta_2$, if $\omega(\beta^{\uparrow})=0$. 
	
The $\tau_a$-conjugated deformation
of the line $t_-$, from $R$ into $L$-stratum, starts over $\tau_a(\alpha_2) \subset K_{b^2,R}$ 
and ends over
$\tau_a(\alpha_1) \subset K_{b^2,L}$.
This conjugated deformation over
 $[\tau_a(\alpha_2),\tau_a(\alpha_1)]$ is defined by the interpolation of the image 
of
$+1 \in t_-$ by the boundary condition: $\Theta(\tau_a(\alpha_2)) =\overline{\Theta}_1$ into $-\overline{\Theta}_2$. 
We get:
 $d(F(t_-))=+\i(-1)^{\omega(\beta^{\uparrow})}$, and 
$d(\tau_a(F(-t_+))) = +\i(-1)^{\omega(\beta^{\uparrow})}$, because $\omega(\beta^{\uparrow})=\omega(\tau_a(\beta)^{\downarrow})$.
		
The homotopy
$F(\tau_a(-t_+))$ coincides with $F(-t_-)$; because 
$F(-t_-)$ is antipodal to the image of the base vector   $+1 \in t_-$, in the case $\omega(\beta^{\uparrow}) =0$, because of the relation: $dF(+t_-)=+\i$; we get: $dF(-t_-)=+\i$ and $d(F(\tau_a(-t_+)))=+\i$.

	Then let us consider the case of the involution $\tau_i$. Assume that the case  $-1 \in t_+$ is considered:
		\begin{align*}
		F^{(1)}\tau_i:&((1-r)(x',y')+r(x'',y''),-t_{+})=
		F^{(1)}((1-r)(x',y')+r(x'',y''),t_{-})\mapsto\\
		&\mapsto((1-r)(x',y')+r(x'',y''),\overline{\exp( ((1-r)\Theta_1+r\Theta_2)\pi/2     )}  ).
	\end{align*}
	С другой стороны, 
	\begin{align*}
		\tau_i F^{(1)}:&((1-r)(x',y')+r(x'',y''),-t_{+})\mapsto\\
		&\mapsto\tau_i((1-r)(x',y')+r(x'',y''),\exp(   ((1-r)\Theta_1+r\Theta_2)\pi/2     )  t)=\\
		&=((1-r)(x',y')+r(x'',y''),\overline{\exp(   ((1-r)\Theta_1+r\Theta_2)\pi/2     )  }).
	\end{align*}

	The deformation in Case $+1 \in t_-$ of the two lines are real-conjugated and starting targets of the two lines coincide.  This proves $\tau_i$-equivariant.

	The $\tau_b$-equivariance is evident. Lemma  \ref{l10}.2 is proved.
\end{proof}

Let us conclude with a plane of the proof of Theorem
\ref{t10}. The obstruction of an extension of the morphism, which is constructed in Lemma \ref{l10} on strata of the deep $3$ and greater are trivial. The obstruction to extend of the morphism on strata of the deep $2$ and greater are calculated below explicitly.

\subsubsection{A morphism of a line  bundle  $\tilde \tau_a$-skew-invariant  and	$\tilde \tau_b$-invariant into the trivial $\C$-bundle over  $\hat K_{b^2}$}

Let us consider a 
$(\tau_a,\tau_b,\tau_i)$-morphism  $F$ of the bundle 
$E_\R$ into the bundle  $E_\C$, which is constructed in Theorem \ref{t10}.
Let us extend this equivariant morphism by a tensor product
$(\tau_a,\tau_b,\tau_i) \mapsto (\tilde \tau_a,\tilde \tau_b,\tilde \tau_i)$, using an additional coordinate $w$ on the sphere  $S^0=\{\pm 1\}$.

Define  $(\tau_a,\tau_b,\tau_i)$-equivariant 
$S^0$-covering over $\hat K_{b^2}$ (a double covering), which is considered as a bundle $E_S$ over $\hat K_{b^2}$ with the sphere  $S^0=\{\pm 1\}$ in a layer. 
Without equivariant structure we get:
$E_S=\hat K_{b^2}\times S^0$.  The involutions are defined by the following formula: 
\begin{align}\label{S0}
	\tau_a&:((x,y),w)\mapsto ((-x,y),-w)\\
	\tau_b&:((x,y),w)\mapsto ((-y,x),w)\\
	\tau_i&:((x,y),w)\mapsto ((x,y),-w).
\end{align}

Involutions $\tau_b$, $\tau_i$ are now defined in an extended bundle $E_{\R} \otimes E_S$,
let us denote the involutions by  $\tilde{\tau}_b$, $\tilde{\tau}_i$ correspondingly. The involution
$\tau_a$ is also defined in the extended bundle, let us denote  $\tilde{\tau}_a = \tau_i \circ \tau_b \circ \tau_a$, more detailed, let us define  $\tilde{\tau}_a$ by the formulas: 
$$\tilde{\tau}_a:((x,y),t_{\pm})\mapsto  ((-x,y),-t_{\mp})\\$$
$$ \tilde{\tau}_a:((x,y),z) \mapsto ((-x,y),z) $$
and by the formula:
$$\tilde{\tau}_a:((x,y),w)\mapsto ((x,y),w),$$
which is followed from  (\ref{S0}),(79) and from the formula below   (\ref{coort}).
Let us note that after re-denotation, if  $(x,y) \in \hat K_{b^2,R}$, then $(-x,y) \in \hat K_{b^2,L}$, but we use $R$-system and the formula (\ref{coort}).

Involutions $(\tilde{\tau}_b, \tilde{\tau}_a, \tilde{\tau}_i)$ are defined in the bundle  $E_{\C}$
by a natural way. The involutions  $\tilde{\tau}_b$, $\tilde{\tau}_a$ on $\C$--bundle 
are trivial, the involution
$\tilde{\tau}_i$ acts in each layer by the conjugation. The involution   $\tau_a$ 
invariantly transforms each oriented layer of the line subbundle
$t_+$ in the bundle $E_\R$ into the oriented layer of another line subbundle  
$t_-$ over the corresponding point of the base, by an orienting preserved morphism of the layers.
The same way, the involution
$\tilde{\tau}_a$ preserves the orienting layers of the bundle  $E_\R$ and changes the orientation of the base vectors. Involution in extended bundle allows to define an equivariant morphism   $\tilde{F}$, as a natural extension of the equivariant morphism $F$. Thus, we define 
a  $(\tilde{\tau}_a,\tilde{\tau}_b,\tilde{\tau}_i)$-equivariant morphism  $\tilde{F}$.

In the collection of the involutions only the involution $\tilde{\tau}_i$ is non-trivial on $E_S$ and conjugates the complex structure in
$E_{\C}$. The involution 
$\tilde{\tau}_i$ is an exceptional, because this involution changes the $w$-coordinate (the coordinate of the covering).

Let us formulate the result in the following lemma.

\begin{lemma}\label{mor}
	Consider the $(\tilde{\tau}_a,\tilde{\tau}_b,\tilde{\tau}_i)$-equivariant tensor product  $E_{\R} \otimes E_S$, $E_{\C}\otimes E_S$. 
	
	1.The formula of $F$ is defined
	$(\tilde{\tau}_a,\tilde{\tau}_b,\tilde{\tau}_i)$-equivariant morphism 
	$$\tilde{F}: E_\R \otimes E_S \mapsto E_\C \otimes E_S. $$
	
	2. The $\tilde{\tau}_i$-equivariant factormorphism
	$$F^{\downarrow}: E_\R \otimes_{\tilde{\tau}_a,\tilde{\tau}_b} E_S \mapsto E_\C \otimes_{\tilde{\tau}_a,\tilde{\tau}_b} E_S $$
	is well defined.
\end{lemma}

\subsubsection*{Proof of Lemma \ref{mor}}

This is well-known fact from linear algebra.  \qed

\subsubsection{A calculation of the factor-bundles}

\begin{lemma}\label{lemma2}
	The bundle
	$ S(E_\C)\otimes_{(\tilde{\tau}_a,\tilde{\tau}_b)} E_S$ is isomorphic to the trivial $\tilde{\tau}_i$-equivariant bundle  
	$K\times S^1\to K$.
\end{lemma}

\begin{proof}
	Define the mapping  $F_\C:S(E_\C)\otimes_{(\tau_a,\tau_b,\tau_i)} E_S$ into $K\times S^1$
	by the formula:
	$$
	F_\C: ((x,y),z,w) \mapsto ([x,y],c_w(z)).$$
	In this formula $z\in \C, |z|=1, w\in S^0=\{\pm 1\}$.
	Let us recall that $(x,y)\in K_{b^2} $  и $(x,y)=(-x,-y)$,
	$[x,y]\in K$, where $[x,y]=a[x,y]=b[x,y]$.
	Additionally, here we have used  
	$$c_w(z)= \left\{\begin{array}{ll}  z ,& \text{if} w=1;\\
		\overline{z},& \text{if} w=-1.  \end{array}\right.$$
	It is clear that  $c_w(z) = c_{-w}(\overline z)$.
	
	Let us check that $F_\C$ is well defined. In fact,
	\begin{align*}
		\tau_a((x,y),z,w)=((y,x),\overline{z},-w)\stackrel{F_\C}{\mapsto} ([y,x],c_{-w}(\overline z))\\
		\tau_b((x,y),z,w) =((-y,x),\overline{z},-w)\stackrel{F_\C}{\mapsto} ([-y,x],c_{-w}(\overline z))\\
		\tau_i((x,y),z,w) = ((x,y),\overline{z},-w)\stackrel{F_\C}{\mapsto} ([x,y],c_{-w}(\overline z))
	\end{align*}
	In the result in all cases we get the same point
	$([x,y],c_w(z))\in K\times S^1$.

	The inverse mapping  $G_\C: K\times S^1 \to S(E_\C)\otimes_{(\tau_a,\tau_b,\tau_i)} E_S$
	is given by the formula:
	$$G_\C: ([x,y],z)\mapsto ((x,y),z,1),$$
	where $z\in S^1\subset \C$. To check that the formula is well defined it has to be checked that 
	when a representative of the point $([x,y],z)\in K\times S^1 $ is replaced, the point  $((x,y),z,1) \in S(E_\C)\otimes_{(\tau_a,\tau_b,\tau_i)} E_S$ remains the same. A gauge
	is the following: $([y,x],z) = ([-y,x],z)=([x,-y],z)$.
	It is easy to check that the points $((y,x),z,1) $, $((-y,x),z,1)$, $((x,-y),z,1)$ coincide to
	the point
	$((x,y),z,1) \in S(E_\C)\otimes_{(\tau_a,\tau_b,\tau_i)} E_S$.
	
	Additionally, it is evident that $F_\C$  and  $G_\C$ are commuted with the projections onto the base $K$ and that this morphisms are inverse.
\end{proof}

\begin{lemma}\label{lemma1}
	The bundle $ S(E_\R)\otimes_{(\tilde{\tau}_a,\tilde{\tau}_b,\tilde{\tau}_i)} E_S$  isomorphic to the wedge covering  $K_b\to K$. 
\end{lemma}

\begin{proof}
	The bundle
	$ S(E_\R)\otimes_{(\tau_a,\tau_b,\tau_i)} E_S$ after a gauge $(\tau_a,\tau_b,\tau_i) \mapsto (\tilde{\tau}_a,\tilde{\tau}_b,\tilde{\tau}_i)$ is splitted into a $\tilde{\tau}_a$ equivariant wedge  $(\tilde{\tau}_b,\tilde{\tau}_i)$ of two $\tilde{\tau}_i$-conjugated  non-trivial $\tilde{\tau}_a$-equivariant line bundles.   
\end{proof}

\subsubsection{The convolution of the tensor product}

By Lemma \ref{lemma1} and Lemma \ref{mor}, Statement 2 the factormorphism $F^{\downarrow}$
is the pair of morphisms of line bundles:
\begin{eqnarray}\label{F+}
	F_{+1}: \hat{K}_b^1 \otimes E_S \to S^1 \otimes E_S, 
\end{eqnarray}
\begin{eqnarray}\label{F-}
	F_{-1}: \hat{K}_b^1 \otimes E_S \to S^1 \otimes E_S. 
\end{eqnarray}
The morphism (\ref{F+}), restricted to the subbundle $\hat{K}_b^1 \otimes \{+1\} \subset \hat{K}_b^1 \otimes E_S$, takes values in $E_S \otimes \{+1\}$. This morphism is the required  in Theorem \ref{resol} mapping, on the subpolyhedron consists of all strata of deeps $0$ and $1$. The morphism (\ref{F-}), restricted to the subbundle $\hat{K}_b^1 \otimes \{-1\} \subset \hat{K}_b^1 \otimes E_S$ in the target, is the second (conjugated) copy of (\ref{F+}), it takes values in
$E_S \otimes \{-1\}$.

\subsubsection*{Extension of the morphism (\ref{F+}) over $\hat{K}^{(1)}$ to a morphism over $\hat{K}^{(2)}$ and over $\hat{K}$ \label{ext}}

Let us construct an extension of the morphism (\ref{F+}) to a morphism over $\hat{K}^{(2)}$
with strata of the deep $2$. Because  a boundary of a strata of the deep $2$ inside
 $K_{b}^{(1)}$ is inside $4$-uple of maximal strata, the problem of the extension is local.
 
 Let us consider the morphism $F$ by (\ref{F+}) and let us prove that this morphism
 has the trivial monodromy along a closed path around strata of the deep $2$.
Take a stratum  $\gamma^{(2)}$ of the deep $2$ for which angles coordinates
 $j_1$, $j_2$, $j_1 < j_2$ are cancelled.  Neighbouring strata  $\alpha^{(0)}_{-\i,-\i}$,
$\alpha^{(0)}_{-\i,+\i}$, $\alpha^{(0)}_{+\i,-\i}$, $\alpha^{(0)}_{+\i,+\i}$ 
are different by residues with numbers $j_1,j_2$.

Assume that on the stratum
$\gamma^{(2)}$ the number of angles with the residue $+\i$ is greater then angles with the residue $-\i$ by $1$. This case is not the only possible, but in the last cases, when the number of coordinates with the residues of the type $+\i$ is greater (is less) then the number of coordinates with the residues $-\i$  reasoning are evident.
In the case the number of angles with the residue $+\i$ is less then angles with the residue $-\i$ by $1$ is analogous.

Tips $R,L$ of maximal strata near the stratum  $\gamma^{(2)}$ is defined by $j_1,j_2$-residues.
With our denotation  the stratum
$\alpha^{(0)}_{-\i,-\i}$ is an $L$-stratum, because the both residues of angles with numbers $j_1$, $j_2$ equal to 
$-\i$, therefore on $\alpha^{(0)}_{-\i,-\i}$ the number of coordinates with the residue $-\i$ is greater then the number of coordinates with the residue $+\i$ by $1$. The strata
$\alpha_{-\i,+\i}^{(0)}, \alpha_{+\i,-\i}^{(0)}$ are  $R$-strata, for the strata  the number of coordinates with the residue $+\i$ is greater then the number of coordinates with the residue $-\i$ by $1$.  The stratum
$\alpha_{+\i,+\i}^{(0)}$ is  $R$-stratum, for the strata  the number of coordinates with the residue $+\i$ is greater then the number of coordinates with the residue $-\i$ by $3$.

Let us consider $4$ strata of the deep $1$ in a neighbourhood of $\gamma^{(2)}$.
Denote by $\beta_{-\i,0}^{(1)}$, $\beta_{0,-\i}^{(1)}$ strata of the deep $1$ 
which are inside the intersection of $R$ and $L$ as in the formula (\ref{bet}). The 
lower index $0$ of strata corresponds to degenerated coordinate.
Assume, for the beginning, that $j_1=j, j_2=j+1$ are neighbouring positive integers.
In this case the residues on the maximal stratum $\alpha^{(0)}_{-\i,-\i}$ are 
$\delta(j)=-\i$, $\delta(j+1)=-\i$. 
On $\alpha^{(0)}_{+\i,-\i}$ residues are: $\delta(j)=+\i$, $\delta(j+1)=-\i$;
on $\alpha^{(0)}_{-\i,+\i}$ residues are: $\delta(j)=-\i$, $\delta(j+1)=+\i$.

The following formula is satisfied:
\begin{eqnarray}\label{+1}
	\omega(\beta_{-\i,0}^{\downarrow}) = \omega(\beta_{0,-\i}^{\uparrow}) + 1.
\end{eqnarray}

Because exceptional numbers are neighbouring, the ordered collection of regular residues for the strata $\beta_{-\i,0}$, $\beta_{0,-\i}$ coincide. Because on the  boundary loop 
\begin{eqnarray}\label{loop}
(+\i,+\i) \mapsto (+\i,-\i) \mapsto (-\i,-\i)  \mapsto (-\i,+\i) \mapsto (+\i,+\i) 
\end{eqnarray}
 of $\gamma^{(2)}$ the stratum $\beta_{0,-\i}^{\uparrow}$ is positively cooriented $R \mapsto L$, and $\beta_{-\i,0}^{\downarrow}$ is negatively cooriented $L \mapsto R$, the calculation of the disorders are based on the common two ordered collections of residues and the formula (\ref{+1}) is analogous to  the formula (\ref{omeg}).  

In the case, when exceptional coordinates are not neighbouring, let us prove that the formula
 (\ref{+1})  is also satisfied. Assume a residue with a number
 $j_3$, $j_1 < j_3 < j_2$ is equal to
$-\i$, the terms $\omega(\beta_{-\i,0}^{\downarrow})$,$\omega(\beta_{0,-\i}^{\uparrow})$ in (\ref{+1}) remains unchanged. 
Because on the $R$-stratum
$\alpha^{(0)}_{+\i,-\i}$  and on the $L$-stratum
$\alpha^{(0)}_{-\i,-\i}$ the residue $j_2$ is negative imaginary, in the calculation $\omega(\beta_{0,-\i}^{\uparrow})$ the both residues $j_3, j_2$ is translated into the left
and the transposition with the numbers $(j_3,j_2)$ does not exist.
Because on the $R$-stratum
$\alpha^{(0)}_{-\i,+\i}$  and on the $L$-stratum
$\alpha^{(0)}_{-\i,-\i}$ the residue $j_1$ is negative imaginary, in the calculation $\omega(\beta_{-\i,0}^{\downarrow})$ the both residues $j_1,j_3$ is translated into the right
and the transposition with the numbers $(j_1,j_3)$ does not exist.

 Assume a residue with a number
$j_3$, $j_1 < j_3 < j_2$ is equal to
$+\i$, the terms $\omega(\beta_{-\i,0}^{\downarrow})$, $\omega(\beta_{0,-\i}^{\uparrow})$ in (\ref{+1}) are changed simultaneously. 
Because on the $R$-stratum
$\alpha^{(0)}_{+\i,-\i}$  and on the $L$-stratum
$\alpha^{(0)}_{-\i,-\i}$ the residue $j_2$ is negative imaginary, in the calculation $\omega(\beta_{0,-\i}^{\uparrow})$ the  residue $j_3$  is translated into the right
 the  residue $j_2$  is translated into the left
and the transposition with the numbers $(j_3,j_2)$ exists.
Because on the $R$-stratum
$\alpha^{(0)}_{-\i,+\i}$  and on the $L$-stratum
$\alpha^{(0)}_{-\i,-\i}$ the residue $j_1$ is negative imaginary, in the calculation $\omega(\beta_{-\i,0}^{\downarrow})$ the  residue $j_3$ is translated into the left,
the  residue $j_1$ is translated into the right
and the transposition with the numbers $(j_1,j_3)$ exists.

The formula   
(\ref{+1}) is proved for an arbitrary $j_1,j_2$ and for an arbitrary collection of residues. 

The formula of the morphism $F$ along $\beta_{0,-\i}^{\uparrow}$ and $\beta_{-\i,0}^{\downarrow}$ is written using integer residues $\Theta$.
In particular, residues
  $\Theta(\alpha^{(0)}_{-\i,+\i})$, $\Theta(\alpha^{(0)}_{+\i,-\i})$ on $R$-strata coincide. 
 The stratum  $\alpha^{(0)}_{-\i,+\i})$ is a common maximal stratum for 
 $\beta_{0,-\i}^{\uparrow}$, $\beta_{-\i,0}^{\downarrow}$
 By the formula (\ref{dF}) the number of rotation of the base vector $t_+$ along the corresponding two segments of the loop (\ref{loop}) are cancelled. 
In particular, if we assume
$\omega(\beta_{0,-\i}^{\uparrow})=0$, from the formula (\ref{+1}) we get $\omega(\beta_{-\i,0}^{\downarrow})=1$. 

The formula of the morphism $F$ on the segments $[\alpha^{(0)}_{-\i,+\i},\alpha^{(0)}_{+\i,+\i}]$, $[\alpha^{(0)}_{+\i,+\i},\alpha^{(0)}_{+\i,-\i}]$
is defined by (\ref{1.1+}),(\ref{1.1-}). One may see that the target homotopies of the vector $t_+$
on the segments are opposite. This proves that the monodromy along the path (\ref{loop})
is the identity. 

In the case all maximal strata in a neighbourhood of $\gamma^{(2)}$ are $R$-strata, or $L$-strata the monodromy of the vector $t_+$ is calculated from the formula (\ref{1.1+}) and is the identity. Theorem
 \ref{t10} is proved. \qed

Let us conclude the result of the subsection by the following theorem.

\begin{theorem}\label{t11}
	There exists a mapping 
	\begin{eqnarray}\label{piF}
		\hat{\pi}: \hat{K}_{\I_b;\circ} \to S^1 = B\Z(\dot A),
	\end{eqnarray}
which is coincides with the required mapping (\ref{key}).
\end{theorem}

\subsubsection{Proof of Theorem \ref{t11}}

Take the equivariant  morphism $F$ constructed in Theorem \ref{t10}. Let us apply Lemma \ref{mor} and calculations in Lemma \ref{lemma2}. \qed

\subsection{Additional properties of the resolution (\ref{piF})}

Recall, $dim(\hat{K}_{\I_b;\circ})=n-k+1, \quad n-k+1 \equiv 1 \pmod{2}$.
The boundary of the polyhedron $\hat{K}_{\I_b;\circ}$ consists of points on all regular strata, which are closed to antidiagonal (\ref{antidiag}).
 Let us consider  of the open polyhedron,
a thin regular neighbourhood of the boundary of $\hat{K}_{\I_b;\circ}$, which is denoted by
$\hat{U}_{\circ} \subset \hat{K}_{\I_b;\circ}$. Analogously, define $U_{\circ} \subset K_{\I_b;\circ}$.

Consider the mapping $\hat{U}_{\circ} \to S^1$, which is the restriction of 
the mapping
$\hat\pi: \hat{K}_{\I_b;\circ} \to S^1$, the quotient of the equivariant mapping (\ref{piF}). Because the structuring group
$\hat{U}_{\circ}$ is reduced to the subgroup  $(b,c) \subset C$,
the mapping
$\hat{U}_{\circ} \to S^1$ contains values in the even index 2 subgroup $2[\pi_1(S^1)] \subset \pi_1(S^1)$. 

\begin{lemma}\label{222}

1. There exists a subpolyhedron
$\hat{U}_{reg,\circ} \subset  \hat{U}_{\circ}$, such that the restriction
of the mapping  $\hat{\pi}$ to the complement 
$\hat{U}_{\circ} \setminus \hat{U}_{reg,\circ}$ is homotopic to the mapping into a point, and the polyhedron itself admits the following mappings: 

\begin{eqnarray}\label{mappings}
\hat{U}_{reg,\circ} \stackrel{\hat{\eta}_{reg}}{\longrightarrow} \hat{P} \stackrel{\hat{\zeta}_{reg}}{\longrightarrow} K(\{b,c\},1),
\end{eqnarray}
where 
$\hat{P}$ is a polyhedron of the dimension  $\frac{(3n-3k+6)}{4}$ (assuming $k \equiv 0 \pmod{4}$)
(i.e. of the dimension a little greater then
$\frac{3}{4}dim(\hat{K}_{\I_b;\circ}) = \frac{3(n-k+1)}{4}$); \ $\hat{\eta}_{reg}$, $\hat{\zeta}_{reg}$ are  mappings,
which are called the control mapping and the structuring mapping, correspondingly, for $\hat{P}$.

2. The covering subpolyhedron
$U_{reg,\circ} \subset  U_{\circ}$, such that the restriction
of the mapping  $\pi$ to the complement 
$U_{\circ} \setminus U_{reg,\circ}$ is homotopic to the mapping into a point, and the polyhedron itself admits the following mappings: 

\begin{eqnarray}\label{mappingsP}
	U_{reg,\circ} \stackrel{\eta_{reg}}{\longrightarrow} P \stackrel{\zeta_{reg}}{\longrightarrow} K(\{b\},1),
\end{eqnarray}
where 
$P$ is a polyhedron of the dimension  $\frac{n-k+2}{2}$, i.e. of the dimension a little greater then
$\frac{1}{2}dim(K_{\I_b;\circ}) = \frac{n-k+1}{2}$, $\eta_{reg}$, $\zeta_{reg}$ are  mappings,
which are called the control mapping and the structuring mapping, correspondingly, for $P$.

\end{lemma}

\subsubsection*{Proof of Lemma \ref{222}}

Proof of Statement 1. Consider in the polyhedron
$\hat{U}_{reg,\circ} \subset \hat{K}_{\I_b;\circ}$: the polyhedron of elementary strata with prescribes coordinate system,
see susbsubsection (\ref{prescr}). Let us define a "nice" coordinate system on each stratum of the polyhedron. A  "nice"  coordinate system is a coordinate system, which is attached to the anti-diagonal boundary, in which  coordinates with residues $-\i$ are singular and associated momenta are small, and coordinates with residue $+\i$ are regular and  associated momenta is defined a hyperface of the corresponding hyperface of the $\lambda$-simplex.

The prescribed coordinate system on an elementary stratum could be different with respect to the "nice" coordinate system. Such a stratum let us call regular.
Let us consider all strata for which the "nice" coordinate system coincides with prescribed coordinate system. This implies that for a regular stratum the number of "nice" residue $+\i$ less or equals to the number of  "nice" angles with the residue
$-\i$. Denote the polyhedron of all "nice" strata by $\hat{Q}_{nice} \subset \hat{U}_{reg,\circ}$ Let us define a $l$-polyhedron, denoted by $(\hat{Q}^{(l)} \cap \Delta_{antidiag}) \subset \Delta_{antidiag}$, which is defined as the closure of all maximal strata in $\hat{U}_{reg,\circ}$, with the antidiagonal  boundary
a full subpolyhedron strata of the deep $l$. Obviously,
\begin{eqnarray}\label{bigcup}
 \bigcup_{l, l \le \frac{n-k+2}{4}} (\hat{Q}^{(l)} \cap \Delta_{antidiag}) \supset (\hat{Q}_{nice} \cap \Delta_{antidiag}). 
\end{eqnarray}

Each polyhedron in (\ref{bigcup}) admits a mapping onto a polyhedron of the dimension $\frac{3n-3k+6}{4}$. By this map $\frac{n-k-2}{4}$ momenta coordinates with
the residue $-\i$ are omitted, but angels with residue $-\i$ are not omitted. Momenta in $\hat{Q}^{(l)}$ are possible only for angles coordinates with the residue $+\i$.

The mapping
$\hat{\pi}$ (\ref{piF}), restricted on   
$\hat{K}_{\I_b;\circ} \setminus \hat{Q}_{nice}$, is homotopic to constant mapping, because
maximal strata have only admissible boundaries. Lemma \ref{222} Statement 1 is proved.

Proof Statement 2. Consider in the polyhedron
$U_{reg,\circ} \subset K_{\I_b;\circ}$ elementary strata with prescribes coordinate system,
see Susbsubsection (\ref{prescr}). Let us define a "nice" coordinate system on each stratum of the polyhedron as in Statement 1.
Let us consider all maximal strata for which the "nice" coordinate system coincides with prescribed coordinate system. This implies that for a regular stratum the number of "nice" residue $+\i$ less or equals to the number of  "nice" angles with the residue
$-\i$ in the full collection of resedues. For a regular stratum the number of "nice" residue $+\i$ less, or equals to the number of all coordinates of the stratum.

Denote by $P \subset \Delta_{antidiag}$ the subpolyhedron, consists of boundaries of all regular strata.  
The polyhedron $P$ is equipped by the projection $U_{reg;\circ} \longrightarrow P$ as a projection, which keeps all angles with the residue $+\i$ and all associated momenta.
By definition $P$ is a subpolyhedron in $S^{n-k+1}/\i$. This implies that the structured mapping 
$\zeta_{reg}$ exists. 

The covering $\pi$ over the mapping
$\hat{\pi}$ (\ref{piF}), restricted on   
$U_{reg,\circ} \setminus \hat{Q}_{nice}$, is homotopic to constant mapping, because
maximal strata have only admissible boundaries.
 Lemma \ref{222} Statement 2 is proved.
\qed

Let us consider the mapping
$c_0$ and consider the polyhedrons  $\N_{\circ}$.  Accordingly to  (\ref{Ib}),
(\ref{Ia}), (\ref{Ib}) a decomposition of this polyhedron into subpolyhedrons is well defined. 
This decomposition corresponds to reductions of the structured group of elementary strata
to the corresponding subgroups in  $D$. 
For  $k' \ge 2$ the statement about two disjoint components on the polyhedrons by a formal
deformation of equivariant holonomic mapping
$c_0^{(2)}$, is deduced from the following lemma.

\begin{lemma}\label{223}
The classifying mapping 
\begin{eqnarray}\label{5}
	\pi_d = p_{A,\dot{A}}  \circ \eta_{\d}:\hat{K}_{\I_d} \to K(\D,1) \to K(\Z/2,1)
\end{eqnarray} 
for the canonical covering over
 $\hat{K}_{\I_d}$, defined by (\ref{Idc}), where  $\eta_d$ is the structuring mapping, 
$p_{A,\dot{A}}$ is the projection on the residue class of the subgroup  $(A,\dot{A}) \subset \D$,  is lifted to a mapping:
\begin{eqnarray}\label{barpi} 
 \bar{\pi}_{d} : \hat{K}_{\I_d} \to S^1, 
\end{eqnarray}
where on each elementary strata of the polyhedrons the mapping $\bar{\pi}$ is homotopic to a constant mapping. 
\end{lemma}

\subsubsection*{Proof of Lemma \ref{223}}

Let us assume
 $N >> n$ ($N=n+4$ is sufficient) and let us consider a "stabilization" of the mapping
 $c_0$, by a replace in the definition of this mapping the parameter  $n \mapsto N$. 
Then an inclusion of polyhedrons  
\begin{eqnarray}\label{in}
\hat{K}_{\I_d}(n) \subset \hat{K}_{\I_d}(N).
\end{eqnarray}
is well defined. Let us consider a subpolyhedron in
$\hat{K}_{\I_d}(N)$, which is denoted  the same, consists of elementary strata with a 
complete collection of residues af all types $+1,-1,+\i,-\i$.
It is clear, that the inclusion
 (\ref{in}) is well defined after by this "stabilization". Moreover, the inclusion
(\ref{in}) is naturally with respect to structuring mappings, therefore the statement of Lemma is sufficiently to prove for the polyhedrons  $\hat{K}_{\I_{d}}(N)$.

The projections  $p_{b}: \hat{K}_{\I_d}(N) \to \hat{K}_{\I_b}(N)$, $p_{a \times \aa}: \hat{K}_{\I_d}(N) \to \hat{K}_{\I_{a \times \aa}}(N)$ are well defined,
 moreover, the classifying mapping  $\bar{\pi}_{\d}$ is the sum of the two mappings (the mappings into the circle are summed in the right hand side of the formula): 
$$ \bar{\pi}_d(N) = \pi_{a \times \aa} \circ p_{a \times \aa} + \pi_{b} \circ p_{b}, $$
where  $\pi_{a \times \aa}$ is obtained from  (\ref{piF}),
$\pi_{b}$ is  (\ref{piF}), $\pi_{a \times \aa}$ is analogous mapping, constructed in Theorem \ref{resol2}. Lemma \ref{223} is proved. \qed

Let us conclude investigations of properties of polyhedrons
 $\N_{\b}$, $\N_{\bf{a}\times \aa;\circ}$ by the following calculation of characteristic numbers, which is required for the next subsection and Theorem
 \ref{T11}. The fundamental class of the polyhedrons $\N_{\I_{a \times \aa};\circ}$, $\N_{\I_{a \times \aa};\circ}$ are well defined, because the polyhedrons is a union of maximal strata, and each maximal strata is equipped with the fundamental class.
 The same is satisfied for an arbitrary  submanifold of an elementary strata, which is generic with respect to the boundary of its maximal strata.

\begin{lemma}\label{l22}
1. The Hurewich image of the fundamental class of the polyhedrons $\N_{\I_{a \times \aa};\circ}$
in the group $H_{n-2k+1}(K(\D,1))$, $dim(\N_{\circ})=n-2k+1$ 
is trivial. The image of the fundamental class of the closed polyhedrons  $\N_{\I_b}$ $(\ref{anti})$
by the structuring mapping  $\eta$ is translated into the generator in  $H_{n-2k+1}(K(\D,1);\Z/2)$.

2. The Hurewitcz image of the fundamental class of an arbitrary codimension $2s$ subpolyhedron in $\N_{\I_{a \times \aa};\circ}$ (which is
invariant with respect to the involution  of the double covering  $\N_{\I_{a \times \aa};\circ} \longrightarrow \hat{\N}_{\I_{a \times \aa};\circ}$ ($K_{\I_b;\circ} \to \hat{K}_{\I_b;\circ}$) in the group $H_{n-2k-2s+1}(K(\D,1))$  is trivial.
\end{lemma}

\subsubsection*{Proof of Lemma \ref{l22}}
Proof of Statement 1. The polyhedron $\N_{\b}$ is decomposed into a collection of maximal strata, each stratum is a manifold and contains the fundamental class. The structure of maximal strata 
is agree to the structure of the polyhedron $K$. In $K$ there exists the only anti-diagonal stratum (with residues $\i$) and some number of maximal strata with a mixed structure of the imaginary residues. 

It is sufficiently to consider the transfer of the structuring mapping $\eta$ over each stratum separately and calculate the image the the fundamental class by the transfer
 $\eta_{A \times \dot{A}}^!$.  For the antidiagonal maximal strata of the polyhedron  $\N_{\b;\circ}$ we get that the image of the fundamental cycle in 
$K(A \times \dot{A},1)$ by the projection into 
$K(A \times \dot{A},1) \to K(A,1)$
is translated into the fundamental class in 
 $H_{n-k}(K(A,1);\Z_2)$, $k \equiv 0 \pmod{2}$. The generic elementary stratum
(which is not the antidiagonal maximal stratum) is the double covering 
over the corresponding stratum of the polyhedron  $\hat{K}_{b}$. 
The transfer with respect to the subgroup
 $(A,\dot{A}) \subset \D$ is shrieked into the covering  in the tower of the subgroups $(A,\dot{A}) \subset C \to \D$, therefore the fundamental; class of the stratum is mapped by  $p_{A\times \dot{A} \to A} \circ \eta^!_{A,\dot{A}}$  in
 $H_{n-k}(K(A,1);\Z_2)$ to the trivial homology class.  
 
 Statement 2 is proved analogously. Lemma \ref{l22} is proved. \qed

\subsection{Proof of Lemma \ref{7} \label{lem7}}
Let us use the construction for  the self-intersection polyhedrons of the formal extension $c^{(2)}$ of the mapping $c$, described in Subsubsection \ref{cc}. By Theorem \ref{resol2} there exists
a mapping (\ref{barpi}), the covering over this mapping gives the integer lift of the structuring mapping, Claim --1, Definition \ref{abstru}. The control condition, Claim --2, Definition \ref{abstru}, is a corollary of Theorem \ref{222} Claim --1 and Proposition \ref{iso}.

We may apply the lemma for the case $d^{(2)}$ is an equivariant generic vertical deformation of $c^{(2)}$. We get an inclusion (\ref{inc}), which keeps strata.
 Lemma \ref{7} is proved. \qed

	\section{Classifying space for iterated self-intersections of $\I_{\a \times \aa}$-framed immersions \label{iterated}}

\subsection{Preliminary constructions and definitions \label{III}}

		Let $n_1$, $n_2$ be two odd positive integers, $n_1 + n_2 = n$,
	$k_1,k_2$ be two even positive integers, $n_1>k_1$, $n_2>k_2$.
	Consider the Cartesian product
	 \begin{eqnarray}\label{barc0c1}
	 	\begin{array}{c}
	G=c_{0,1} \times c_{0,2}: \RP^{n_1-k_1} \times \RP^{n_2 - k_2} \to \\
	 S^{n_1-k_1} \times S^{n_2 - k_2} \subset \R^{n_1} \times \R^{n_2} = \R^n \\
	 \end{array}
	\end{eqnarray}
	of two  mappings (\ref{c0}) in the corresponding dimensions. 
	 Let us denote by
	$I_1 \times I_2$ the Cartesian product of the involutions by the $\i$-multiplication on the factors:
	$$ I_1: \RP^{n_1-k_1} \to \RP^{n_1-k_1}; \quad I_2: \RP^{n_2-k_2} \to \RP^{n_2-k_2}$$ on the space
	$\RP^{n_1-k_1} \times \RP^{n_2-k_2}$; the involutions commute. The involutive transformation on $\RP^{n_1-k_1} \times \RP^{n_2 - k_2}$, which permutes the factors, is denoted by $T^{int}$. We may assume that $n_1-k_1 = n_2 - k_2$ is wery great. With this assumption
	the transformation   $T^{int}$ is an involution and is free outside the diagonal. We use the same construction for the target $\R^{n_1} \times \R^{n_2}$. The following dyagrame is commutative:
	\begin{eqnarray}\label{Tint}
		\begin{array}{ccc}
		\RP^{n_1-k_1} \times \RP^{n_2-k_2} & \longrightarrow & \R^{n_1} \times \R^{n_2} \\
		  \downarrow T^{int}   &         &   \downarrow T^{int}  \\
		 \RP^{n_2-k_2} \times \RP^{n_1-k_1} & \longrightarrow & \R^{n_2} \times \R^{n_1}. \\ 
		 \end{array}
	\end{eqnarray}	
	
	By the pair of involutions  $(I_1,I_2)$  let us	  define another pair of involutions
	\begin{eqnarray}\label{II}
	  (\tilde{I},\hat{I}).
	 \end{eqnarray}
The first involution  $\tilde{I}=diag(I_1,I_2)$ is the diagonal involution.
	 Denote $\hat{I} = I_1 \pmod{\tilde{I}}$. Obviously, $I_1 \pmod{\tilde{I}} = I_2 \pmod{\tilde{I}}$ and one may consider the involution $\hat{I}$ as an involution on the factorspace of $\tilde{I}$-involution. The pair of the involution $(\tilde{I},\hat{I})$, where $\hat{I}$ is a bottom (secondary) involution, $\tilde{I}$ in a top (primary) involution, depends not of an order of the Cartesian coordinate and describes the  factor.  The mapping (\ref{barc0c1}) is a $\tilde{I}$-double covering over the following mapping:

	 \begin{eqnarray}\label{c0c1}
		\begin{array}{c}
			\tilde{G}: (\RP^{n_1-k_1} \times \RP^{n_2 - k_2})/\tilde{I} \to \\
			S^{n_1-k_1} \times S^{n_2 - k_2} \subset \R^{n_1} \times \R^{n_2} = \R^n. \\
		\end{array}
	\end{eqnarray}

The mapping (\ref{c0c1}) is the $\hat{I}$-double covering over the following mapping:
\begin{eqnarray}\label{hatc0c1}
	\begin{array}{c}
		\hat{G}: ((\RP^{n_1-k_1} \times \RP^{n_2 - k_2})/\tilde{I})/\hat{I} \to \\
		S^{n_1-k_1} \times S^{n_2 - k_2} \subset \R^{n_1} \times \R^{n_2} = \R^n, \\
	\end{array}
\end{eqnarray}
where the involution in the target in the diagrams (\ref{c0c1}), (\ref{hatc0c1}) are the identity.	
	
Because $$((\RP^{n_1-k_1} \times \RP^{n_2 - k_2})/\tilde{I})/\hat{I} \cong S^{n_1-k_1}/\i \times S^{n_2-k_2}/\i$$ is the product of the two $\Z/4$-lenses, the mapping (\ref{hatc0c1})  is rewritten in the following more clear form:

\begin{eqnarray}\label{hatc0c12}
	\begin{array}{c}
		\hat{G}: S^{n_1-k_1}/\i \times S^{n_2 - k_2}/\i \to \\
		S^{n_1-k_1} \times S^{n_2 - k_2} \subset \R^{n_1} \times \R^{n_2} = \R^n. \\
	\end{array}
\end{eqnarray}

\subsubsection{Configuration space \label{chi}}

Recall, $\Gamma_{\circ}(X) = (X \times X \setminus diag(X))/T_X$, $T_X$ is the involution, the permutation  of the coordinated in the Cartesian product, see (\ref{99}) for a particular case
$X = \RP^{n-k}$. We consider the case $X=\RP^{n_1-k_1} \times \RP^{n_2-k_2}$ and denote
by  $\bar{\Gamma}_{\circ}(\RP^{n_1-k_1} \times \RP^{n_2-k_2})$ the canonical $2$-sheeted covering over  the deleted product $\Gamma_{\circ}(\RP^{n_1-k_1} \times \RP^{n_2-k_2})$
(an unordered pair of points is covered by  two оrdered pairs of points). 
There is a $(T_{\RP^{n_1-k_1} \times \RP^{n_2-k_2}};T_{\R^n \times \R^n})$-equivariant mapping: 
$$ \bar{G}^{(2)}: \bar{\Gamma}_{\circ} (\RP^{n_1-k_1} \times \RP^{n_2-k_2}) \to \bar{\Gamma}_{\circ}(\R^{n} \times \R^n),$$
which is a formal extension of the mapping $G$ (\ref{barc0c1}).

In a more detailed form this mapping looks as following:
\begin{eqnarray}\label{barG2}
	\begin{array}{c}
\bar{G}^{(2)}: ((\RP^{n_1-k_1} \times \RP^{n_2-k_2}) \times (\RP^{n_1-k_1} \times \RP^{n_2-k_2})) \setminus diag \\
\to ((\R^{n} \times \R^n) \times (\R^{n} \times \R^n)) \setminus diag.
\end{array}
\end{eqnarray}
The involutions $T^{int}=(T_{\RP^{n_1-k_1} \times \RP^{n_2-k_2}};T_{\R^n \times \R^n})$ permute the factors in the source deleted product an in the target. 
The involution $T^{int}$ on the 2-sheeted covering 
$$\Gamma_{\circ}(\RP^{n_1-k_1} \times \RP^{n_2-k_2}) \to
\Gamma_{\circ}(\RP^{n_1-k_1} \times \RP^{n_2-k_2})$$
 is well defined; the analogous denotation for the target. The factormapping $\bar{G}^{(2)}/T^{int}$ is denoted by $G^{(2)}$.
 On the deleted product $\bar{\Gamma}_{\circ}(\RP^{n_1-k_1} \times \RP^{n_2-k_2})$ the  involutions $(\tilde{I},\hat{I},T^{int})$ are defined, the definition of  $(\tilde{I},\hat{I})$ is the extensions of (\ref{barc0c1}) to  the deleted product.
An involution and its corresponding extension are denoted by the same for short. 
The pair of involutions $(\bar{I},\hat{I})$ corresponds to the involutions defined in subsection \ref{III}.

The Cartesian products  $G^{(2)}$ by (\ref{barG2})
is  $(\tilde{I},\hat{I},T^{int})$-equivariant mapping, where the involutions $\tilde{I},\hat{I}$ on the target of the diagram is the identity.	.

\subsubsection{Iterated self-intersection polyhedron of $\bar{G}$ (\ref{barc0c1}), diagram description}  

On the secondary deleted product
$\bar{\Gamma}_{\circ}(\bar{\Gamma}_{\circ} (\RP^{n_1-k_1} \times \RP^{n_2-k_2}))$, we get  the following collection
\begin{eqnarray}\label{collectinv}
 (\tilde{I},\hat{I},T^{int};T^{ext})
\end{eqnarray}
 of  involutions, where the first two involutions are defined by (\ref{II}), $T^{int}$ is the interiore permutation, defined by (\ref{Tint}),
 $T^{ext}$ is analogously defined using the exteriore permutation.

	Let us define an  $(\bar{I},\hat{I},T^{int},T^{ext})$-equivariant, $(\bar{I},\hat{I},T^{int})$-equivariant polyhedra
	\begin{eqnarray}\label{barKK}
		\overbrace{\NN_{\circ}} \subset \bar{\Gamma}_{\circ}\bar{\Gamma}_{\circ}(\RP^{n_1-k_1} \times \RP^{n_2-k_2}),
	\end{eqnarray}
\begin{eqnarray}\label{barbarKK}
	\overline{\NN_{\circ}} \subset \bar{\Gamma}_{\circ}\bar{\Gamma}_{\circ}(\RP^{n_1-k_1} \times \RP^{n_2-k_2})/T^{int},
\end{eqnarray}
	 which is called the iterated self-intersection of $(I_1,I_2)$-equivariant mapping  $(\ref{barc0c1})$. 
	 
	 At the first step of the definition, let us define the $T^{int}$-equivariant polyhedron of of formal self-intersection point of the top mapping (\ref{barc0c1}): 
	 \begin{eqnarray}\label{barK}
	 	\bar{\N}_{\circ} \subset \bar{\Gamma}_{\circ}(\RP^{n_1-k_1} \times \RP^{n_2-k_2}),
	 \end{eqnarray}
where the "bar" means that this polyhedron is  $T^{int}$-equivariant:
\begin{eqnarray}\label{NNd}
	\begin{array}{c}
	\bar\N_{\circ} = \{ (\x_1 \times \y_1;\x_2 \times \y_2) \in \bar{\Gamma}_{\circ} : \x_1 \times \y_1 \ne \x_2 \times \y_2, \\
	 G(\x_1 \times \y_1)=G(\x_2 
	\times \y_2) \}.
	\end{array}
\end{eqnarray}
The definition is analogously to the formula (\ref{Nd}).

	 At the second step of the definition, let us consider the equivariant polyhedron  (\ref{barK}) as the origin of the mapping inside the external (secondary) deleted product.  Define the (\ref{collectinv})-equivariant  polyhedron (\ref{barKKdef}) with the quotient (\ref{KK}). We use the formula (\ref{Nd})-type again: we restrict the mapping $\bar{G}^{(2)}$ (\ref{barG2}) to the polyhedron (\ref{barK}) at the origin and consider an additional secondary self-intersection condition in the third line of the formula (\ref{barKKdef}):
\begin{eqnarray}\label{barKKdef}
	\begin{array}{c}
	\overbrace{\NN_{\circ}} = \\
	 \{((\x_{1,1} \times \y_{1,1};\x_{1,2} \times \y_{1,2}) \times (\x_{2,1} \times \y_{2,1};\x_{2,2} \times \y_{2,2})) \subset \bar{\Gamma}_{\circ}\bar{\Gamma}_{\circ}(\RP^{n_1-k_1} \times \RP^{n_2-k_2}), \\
	(\x_{1,1} \times \y_{1,1};\x_{1,2} \times \y_{1,2}) \ne (\x_{2,1} \times \y_{2,1};\x_{2,2} \times \y_{2,2}); \\
	G^{(2)}(\x_{1,1} \times \y_{1,1};\x_{1,2} \times \y_{1,2}) = G^{(2)}(\x_{2,1} \times \y_{2,1};\x_{2,2} \times \y_{2,2}) \}.
	\end{array}
\end{eqnarray}	 
	 Of course, the primary conditions (\ref{NNd}) in the formula (\ref{barKKdef})  for points in  $\bar{\Gamma}_{\circ}(\RP^{n_1-k_1} \times \RP^{n_2-k_2})$ is satisfied.
	 Recall, $T^{int}$, $T^{ext}$ are the involutions in the iterated deleted product, we get the quotient:
	 
	\begin{eqnarray}\label{KK}
	\NN_{\circ} \subset \bar{\Gamma}_{\circ}(\bar{\Gamma}_{\circ}(\RP^{n_1-k_1} \times \RP^{n_2-k_2})/T^{int})/T^{ext}.
\end{eqnarray}	 
The polyhedron (\ref{barKK}) is the $(\bar{I},\hat{I})$-equivariant $(T^{int};T^{ext})$-covering over (\ref{KK}), which is an analog of the cannonical covering .	 
In the formulas (\ref{barKKdef}), (\ref{quadruple}) below, we associate the first lower superscript with $T^{ext}$ and the second lower superscript with $T^{int}$.
	 
\subsubsection{Iterated self-intersection polyhedron of $\bar{G}$ (\ref{barc0c1}), coordinate description} 	 
	 
The coordinate description 	 
	  is the following. Using the formula (\ref{barKKdef}), let us consider all ordered quadruples in 	$\overbrace{\NN}$:
\begin{eqnarray}\label{quadruple}
	\begin{array}{c} ((\x_{1,1},\x_{1,2}),(\y_{1,1},\y_{1,2});(\x_{2,1},\x_{2,2}),(\y_{2,1},
		\y_{2,2}));\\
		 \x_{p,q}  \in \RP^{n_1-k_1}; \quad  \y_{p,q}  \in \RP^{n_2-k_2}, \quad p=1,2; \ q= 1,2
		 \end{array}
	\end{eqnarray} 
  with conditions:
$$G(\x_{1,1},\y_{1,1}) = G(\x_{2,1},\y_{2,1})  =
G(\x_{1,2},\y_{1,2}) = G(\x_{2,2},\y_{2,2}) \in \R^{n_1} \times \R^{n_2}.$$

The space (\ref{barKK}) is the space of all ordered quadruples  (\ref{quadruple}) with the condition:
$\x_{1,1} \neq \x_{2,1}$, or $\y_{1,1} \neq \y_{2,1}$; and $\x_{1,2} \neq \x_{2,2}$, or $\y_{1,2} \neq \y_{2,2}$ (the internal diagonal condition).
$\x_{1,1} \neq \x_{1,2}$, or $\y_{1,1} \neq \y_{1,2}$; and $\x_{2,1} \neq \x_{2,2}$, or $\y_{2,1} \neq \y_{2,2}$
(the external diagonal condition for external 1-1;2-2 coordinates);
$\x_{1,1} \neq \x_{2,2}$, or $\y_{1,1} \neq \y_{2,2}$; and $\x_{2,1} \neq \x_{1,2}$, or $\y_{2,1} \neq \y_{1,2}$
(the external diagonal condition for external 1-2;2-1 coordinates).
The formulas (\ref{quadruple}), (\ref{barKKdef}) coincide.
	 
\subsection{Stratifications}

	\subsubsection{The stratification of the Cartesian product of two spheres $ J_1\times J_2$, $dim(J_1)=n_1-k_1$, $dim(J_2)=n_2-k_2$}

Let $n_1$, $n_2$ be odd, $k_1, k_2$ be even positive number, $n_1+n_2=n$, $k_1+k_2=k$.
This stratification is the Cartesian product of the two stratifications of the factors, described by (\ref{stratJ}).
Let us define the space 
$J_1^{[s]} \times J_2^{[t]}$ as a subspace in $J$ as the union of all subspaces  $J(x_{\k_1}, \dots, x_{\k_s}) \times J(y_{\k_1}, \dots, y_{\k_t})\subset J_1 \times J_2$.

Therefore, the following double stratification 
\begin{eqnarray}\label{stratJ1}
	J^{(r_1)}_1\times J^{(r_2)}_2  \subset \dots \subset J^{(0)}_1 \times J^{(0)}_2,
\end{eqnarray}
of the space $
J_1 \times J_2 $ is well defined. 
For a given stratum the number
$ r_1-s + r_2 - t$, $2(r+s)-1=n-k$, of omitted coordinates is called the deep of the stratum.

\subsubsection{The stratification of $\RP^{n_1-k_1} \times \RP^{n_2-k_2}$}

Let us define the stratification of the product $\RP^{n_1-k_1} \times \RP^{n_2-k_2}$  of projective spaces is defined by the product of the  stratifications of factors defined in subsubsection \ref{stratRP}.
Denote the maximal open cell
$
p^{-1}(J(x_1, \dots, x_{r_1}) \times J(y_1, \dots, y_{r_2})$ by 
\begin{eqnarray}\label{stprod}
U(x_1, \dots, x_{r_1}) \times U(y_1, \dots y_{r_2}) \subset
S^{n_1-k_1}/-1 \times S^{n_2 - k_2}/-1. 
\end{eqnarray}
This cell is called an elementary stratum of the deep $0$.
A point on an elementary stratum of an arbitrary deep $(r_1-s,r_2-t)$
is defined by the collection of spherical coordinates
$(\check x_{\k_1}, \dots, 
	\check x_{\k_s}, \check y_{\k_1}, \dots \check y_{\k_t}; l_1 \times l_2)$, where $\check x_{\k_i}$,  $\check y_{\k_j}$ are  coordinates
on a circle, which is covered the circles of the join with the numbers $\k_i,\k_j$; $l_1 \times l_2$ are coordinates on
$(s-1)\times(t-1)$-dimensional subsimplex of the product of the joins.

\subsubsection{The stratification of the polyhedron  
	$	\overbrace{\NN_{\circ}}$ (\ref{barKK})\label{seccKK}}

The polyhedron
$	\overbrace{\NN_{\circ}}$ is the disjoin union 
of elementary strata, the strata of polyhedron corresponds to  the strata
(\ref{stprod}) in the target.
Let us denote this strata by
\begin{eqnarray}\label{compstrat}
	\begin{array}{c}
	\overline{KK}^{[r_1-s]}(\k_{1,1},\k_{2,1} \dots, \k_{1,s},\k_{2,s}) \times \overline{KK}^{[r_2-t]}(\k_{1,1},\k_{2,1}, \dots, \k_{1,t},\k_{2,t}) =\\
		\overline{KK}^{[r_1-s;r_2-t]}(\k_{1,1},\k_{2,1} \dots, \k_{1,s},\k_{2,s};\k_{1,1},\k_{2,1}, \dots, \k_{1,t},\k_{2,t}) \\
	1 \le s_i
	\le r_i, i=1,2; \quad r_1+r_2=r.  
	\end{array}
\end{eqnarray}

Let us describe an  elementary stratum   (\ref{compstrat}) by means of coordinate system.
For the simplicity of denotation, let us consider the case
$s_1=r_1$, $s_2=r_2$ this is the case of maximal elementary stratum.  Analogous formula exists for points on deeper elementary strata (\ref{compstrat}).

Assume a quadruple  of points
$((\x_{1,1}, \x_{1,2}),(\y_{1,1},\y_{1,2});(\x_{2,1}, \x_{2,2}),(\y_{2,1},\y_{2,2}))$ determines a point on (\ref{compstrat}), let us fix for a quadruple a lift 
$((\check \x_{1,1}, \check \x_{1,2}),(\check \y_{1,1}, \check \y_{1,2});(\check \x_{2,1}, \check \x_{2,2});(\check \y_{2,1}, \check \y_{2,2}))$ on the covering sphere
$S^{n_1-k_1}$,  each point is mapped into the corresponding point of the quadruple  by the projection $ S^{n_1-k_1} \times S^{n_2-k_2} \to \RP^{n_1-k_1} \times \RP^{n_2 - k_2}$.
With respect to constructions above, denote by
\begin{eqnarray}\label{collection}
	\begin{array}{c}
 (\check x_{1,1; i}, \check x_{1,2; i}), (\check y_{1,1; j}, \check y_{1;2; j});  \\
(\check x_{2,1; i}, \check x_{2,2; i}) ; (\check y_{2,1; j}, \check y_{2,2; j}), \\
 i = 1, \dots, r_1; \quad j=1, \dots, r_2 \\ 
\end{array}
\end{eqnarray}
the collection of spherical coordinates of the corresponding point. 
A spherical coordinate determines a point on the circle
with the same number
$i,j$, the circle is a covering over the corresponding coordinate  
$J_1(i) \subset J_1$ ($J_2(j) \subset J_2$) in the join. 

Recall, that a coordinate in a quadruple 
with a prescribed number $i$ ($j$) of the corresponding angle coordinate determines a point in a prescribed fiber of the standard cyclic 
$\i$-covering  $ S^1
\to S^1 / \i $.

Collections of spherical coordinates  (\ref{collection}) of points 
are considered with respect to common elementary  transformations.  A  common transformation of $x,y$-collections with respect to the first and second superscripts is possible, this transformation is described by a monodromy over a closed path
and describes the representation of the fundamental group of the configuration space, which is called the structuring group. Recall, that each $\x$ and $\y$ coordinate of a point is defined up to undependable multiplication by $-1$.  Denote by $\D\D$ the group of the order $32$ which is
the product of the two dihedral groups $\D\D=\D_{\x} \times \D_{\y}$ with a common residue class of the permutation preserving subgroup $A \times \dot{A} \subset \D$ in the each factor.   
Collection of spherical coordinates of a point on the configuration space  is transformed up to 
$8\cdot 4 \times 8 \cdot 4 \times 2=128 \cdot 16$ different transformations. (Up to transformation of the group  $\D\D \rtimes \D\D$ of the order $2048$.)

\subsubsection{A classification of strata of the polyhedron  
	$	\overbrace{\NN_{\circ}}$ (\ref{barKKdef}) by the residues collection \ref{seccKK}}
Elementary strata  of the polyhedron (\ref{barKKdef}),
accordingly with collections of coordinates, are divided into several types, a full description of  subtypes of a type is not required. Let us introduce types
\begin{eqnarray}\label{type}
\{	\I_b, \I_{a \times \aa},  \I_{d} \}
\end{eqnarray}	
of strata of the polyhedron 
(\ref{barKKdef}).

For an arbitrary elementary strata  (\ref{compstrat}) a residues collection of pairs (a number of an $x$ or $y$  residues in a pair is defined by a first superscript and a pair itself in the collection is considered over a second superscript):
\begin{eqnarray}\label{resedu}
	\begin{array}{c}
\{v_{1,\k_i}=\check x_{1,1;\k_i} \check x^{-1}_{1,2;\k_i}, v_{2,\k_i}=\check x_{2,1;\k_j} \check x^{-1}_{2,2;\k_j}; \\
w_{1,\k_j}=\check y_{1,1;\k_j} \check y^{-1}_{1,2;\k_j} ,w_{2,\k_j}=\check y_{2,1;\k_j} \check y^{-1}_{2,2;\k_j}\},\\
 1 \le \k_i \le r_1-k_1, \quad 1 \le \k_j \le r_2-k_2.
 \end{array}
\end{eqnarray}
is well defined. 

The residues collection (\ref{resedu}) for an elementary stratum of the polyhedron  (\ref{KK}) are well defined up to:

$\bullet$ a common conjugation $\{v_{p;\k_i}\},\{w_{p;\k_j}\}$-collection with a given first superscript $p=1,2$; this transformation of residues corresponds to one of the translations
	\begin{eqnarray}\label{gaugep2}
	(\x_{p,1},\y_{p,1};\x_{p,2},\y_{p,2}) \mapsto (\x_{p,2},\y_{p,2};\x_{p,1},\y_{p,1}), \quad p=1, {\rm{or}}, \- p=2;
	\end{eqnarray}

$\bullet$ up to the translation $p \mapsto p+1$, which corresponds to the translation
	\begin{eqnarray}\label{gauge12}
		(\x_{1,1},\y_{1,1};\x_{1,2},\y_{1,2}) \mapsto (\x_{2,2},\y_{2,2};\x_{2,1},\y_{2,1});
		\end{eqnarray}

$\bullet$ up to the multiplication of each element in the $\{v\}$, or, $\{w\}$-collection with a given superscript $p$ by $-1$  undependably. 
	 
Let
 In the case  residues collection for all coordinates of an elementary stratum  takes values in
$\{+\i,-\i\}$ (correspondingly, in $\{+1,-1\}$), one say that an elementary stratum is of the type
$\I_b$ 
(correspondingly, of the type $\I_{a \times \aa}$).  
In the case there are at least $v$, or $w$ residues in the collection of the complex type
$\{+\i,-\i\}$, and of the real type $\{+1,-1\}$ simultaneously, one shall say
about an elementary stratum of the type $\I_{d}$. In particular, in the case $v_{1,1} \in \{\pm \i\}$, $w_{1,1} \in \{\pm 1\}$, we get a stratum of the type $\I_d$.

Let us introduce subtypes 
\begin{eqnarray}\label{subtypes}
	\{\I\I_{b}, \ \I\I_{a \times \aa}, \ \I\I_{d} \} 
\end{eqnarray}
of strata of the type $\I_b$.
Let us introduce the collection 
\begin{eqnarray}\label{crosscoll} 
 \{vv_{1;\k_i},vv_{2;\k_i} ww_{1;\k_j}  ww_{2;\k_j} \} 
 \end{eqnarray}
of cross-residues:
\begin{eqnarray}\label{crossresedu}
	\begin{array}{c}
		\{vv_{s;\k_i}=\check x_{1,s;\k_i}^{-1} \check x_{2,s;\k_i}, \\
		ww_{s;\k_j}=\check y_{1,s;\k_j}^{-1} \check y_{2,s+1;\k_j}\},\\
		s = 1,2; \quad 1 \le \k_i \le r_1-k_1, \quad 1 \le \k_j \le r_2-k_2. \\
	\end{array}
\end{eqnarray}

The collections $\{v_{p;\k_i}\},\{w_{p;\k_j}\}$, $\{vv_{s;\k_i}\},\{ww_{s;\k_j}\}$ are  dependent, because we get a relation:
\begin{eqnarray}\label{eq11}
 v_{1;\k_i}v_{2;\k_i} = vv_{1,\k_i}vv_{2,\k_i}, \qquad w_{1;\k_j}w_{2;\k_j} = ww_{1,\k_j}ww_{2,\k_j}. 
\end{eqnarray}

A priory for an elementary strata the collection (\ref{crosscoll}) is defined up to a general permutation  (\ref{gaugep2}), (\ref{gauge12}) and by the $-1$-multiplication. 
 In particular, the gauge (\ref{gaugep2}) changes real-imaginary types of residues.  
  An appropriate covering over a special 2-sheeted covering polyhedron, which is called the regularization, see subsubsection \ref{regularization}, is defined. A cross-residues in the collection (\ref{crossresedu}) for the regularization is a priori imaginary, or, real.

Let us define subtypes (\ref{subtypes}) of elementary strata the type $\I_b$. In the case cross-residues in $vv$-collection  take values in
$\{+\i,-\i\}$ and cross-residues in $ww$-collection  take values in
$\{+1,-1\}$, or, if  cross-residues in $vv$-collection  take values in
$\{+1,-1\}$ and cross-residues in $ww$-collection  take values in
$\{+\i,-\i\}$, one say that an elementary strata of the type $\I_b$ is of the subtype $\I\I_b$.
In the case cross-residues in $vv$-collection  take values in
$\{+\i,-\i\}$ and cross-residues in $ww$-collection  take values in
$\{+\i,-\i\}$, or, if  cross-residues in $vv$-collection  take values in
$\{+1,-1\}$ and cross-residues in $ww$-collection  take values in
$\{+1,-1\}$, one say that an elementary strata of the type $\I_b$ is of the subtype $\I\I_{a \times \aa}$.
In the case cross-residues in $vv$-collection, or in $ww$-collection  take values in
$\{+\i,-\i\}$ and  in
$\{+\i,-\i\}$ simultaneously, one say that an elementary strata of the type $\I_b$ is of the subtype $\I\I_{d}$.
A special case, when an elementary (a non-maximal) stratum contains only $x$-coordinates ($y$-coordinates are cancelled), 
or an elementary stratum contains only $y$-coordinates ($x$-coordinates are cancelled) is possible.
A subtype (\ref{subtypes}) of an arbitrary elementary strata of the type $\I_b$ (\ref{type}) is defined.

The polyhedron (\ref{KK}), the factorpolyhedron of (\ref{barKKdef}),  is decomposed into disjoin union of the following three polyhedra, the superscript denotes the type:
\begin{eqnarray}\label{IIIb}
	KK_{\I_b} \subset KK_{\circ},
\end{eqnarray}
\begin{eqnarray}\label{IIIa}
	KK_{\I_{a \times \aa; \circ}} \subset KK_{\circ},
\end{eqnarray}
\begin{eqnarray}\label{IIId}
	KK_{\I_d; \circ} \subset KK_{\circ}.
\end{eqnarray}

The subpolyhedron (\ref{IIIb}) is a disjoint union of the following subsubpolyhedra, using the corresponding secondary  subtype (\ref{subtypes}):
\begin{eqnarray}\label{Ibsub}
	KK_{\I\I_b} \subset KK_{\I_b},
\end{eqnarray}
\begin{eqnarray}\label{Iasub}
	KK_{\I\I_{a \times \aa} } \subset KK_{\I_b}, 
\end{eqnarray}
\begin{eqnarray}\label{Idsub}
	KK_{\I\I_d} \subset KK_{\I_b}. 
\end{eqnarray}
 
\subsection{Proprties of subpolyhedra (\ref{IIIb}),(\ref{IIIa}),(\ref{IIId}) }\label{propr}

\subsubsection{A regularization of the polyhedron (\ref{Ibsub}) \label{regularization}}
According subsubsection \ref{seccKK}, one may see that the structuring 
 group $\D\D \tilde\times \D\D$ of the polyhedron (\ref{KK}) is of the order $2^{11}$. 
This group is defined as the semi-direct product 
 \begin{eqnarray}\label{DDDD}
 \D\D \tilde\times \D\D \cong (\D_{A \times \dot{A}} \D) \tilde\times (\D_{A \times \dot{A}} \D)
 \end{eqnarray}
 
 In this formula the group $\D\D \cong \D_{A \times \dot{A}} \D$
is the structuring group of polyhedron, which is the secondary $T^{ext}$-covering over (\ref{Ibsub}).
This group is an index $2$ subgroup in $\D \times \D$ (the subgroup of the order $32$ in the group of the order $64$),
which  is given by the formula $(x_1 \times x_2), x_1 \cong x_2 \pmod{A \times \dot{A}}$, $A \times \dot{A} \subset \D$ (in the case coordinates in a pair  $(x_1,x_2)$ are permuted, the coordinates in the pair $(y_1,y_2)$ are permuted).

The group (\ref{DDDD}) is defined as the skew-product of the two subgroups
$\D\D$, a permutation of the factor corresponds to the the generator $T^{ext}$  of the external double covering.  

Over the subpolyhedron (\ref{Ibsub}) the structuring group (\ref{DDDD}) admits an index $2^3$ reduction, this is the structuring group of the order $2^{8}$:
\begin{eqnarray}\label{R1}
	R_1 \subset \D\D \tilde \times \D\D.
\end{eqnarray}
The group (\ref{DDDD}) keeps  the configuration (\ref{quadruple}). Let us consider a different  configuration of non-ordered pairs as the following:

\begin{eqnarray}\label{quadrupleBIS}
	\begin{array}{c} ((\x_{1,1},\x_{1,2}),(\y_{1,1},\y_{2,1});(\x_{2,1},\x_{2,2}),(\y_{2,1},
		\y_{2,2}));\\
		 \x_{p,q}  \in \RP^{n_1-k_1}; \quad  \y_{p,q}  \in \RP^{n_2-k_2}, \quad p=1,2; \ q= 1,2
		 \end{array}
	\end{eqnarray} 
\begin{eqnarray}\label{quadrupleBISS}
	\begin{array}{c} ((\x_{1,1},\x_{1,2}),(\y_{1,1},\y_{2,2});(\x_{2,2},\x_{2,1}),(\y_{2,1},
		\y_{1,2}));\\
		\x_{p,q}  \in \RP^{n_1-k_1}; \quad  \y_{p,q}  \in \RP^{n_2-k_2}, \quad p=1,2; \ q= 1,2
	\end{array}
\end{eqnarray} 
Consider transformations of (\ref{DDDD}), which also keeps the configurations (\ref{quadrupleBIS}), (\ref{quadrupleBISS}). 
This is an additional symmetry, which is described by the index $4$ subgroup (\ref{R1}).
Define the  subgroup (\ref{R1}) of transformations of (\ref{quadruple}), which keeps
the configurations (\ref{quadruple}), (\ref{quadrupleBIS}), (\ref{quadrupleBISS}) and, additionally, permutes points in the pairs $(\x_{1,1},\x_{1,2})$,  $(\x_{2,2},\x_{2,1})$ simultaneously. The analogous condition for the pairs $(\y_{1,1},\y_{1,2})$,  $(\y_{2,2},\y_{2,1})$ is also satisfied, but gives no additional restrictions.

Over the subpolyhedron (\ref{Ibsub}) we get a monodromy of an each coordinate, which is
a common value $\pmod{\pi} \in S^1/-1$ for the both coordinates in a pair
$(x_{1,\k(i)}),(x_{2,\k(i)})$,$(y_{1,\k(j)}),(y_{2,\k(j)})$. This proves that the monodromy keeps real-imaginary types of cross-residues:
$vv_{1,\k_i},vv_{2,\k_i},ww_{1,\k_j}ww_{2,\k_j} \pmod{\pm 1}$.

Recall, for the 2-sheeted covering polyhedron over $KK_{\I\I_b}$ (\ref{Ibsub}), associated with the subgroup (\ref{R1}) resedues (\ref{resedu}) are immaginary and cross-residues
(\ref{crosscoll}) are real. 
Take the index $4$ subgroup $R \subset R_1$ by the following 
equations:
\begin{eqnarray}\label{eq13}
v_{1;\k_i}=v_{2;\k_i}; \quad w_{1;\k_j}=w_{2,\k_j}. 
\end{eqnarray} 

Because of the equation (\ref{eq11}),
the equations (\ref{eq13}) can be presented in the following equivalent form:
\begin{eqnarray}\label{eq14}
vv_{1;\k_i}=vv_{2;\k_i}; \quad ww_{1;\k_j}=ww_{2,\k_j}. 
\end{eqnarray}  

Let us define the index $4$ subgroup of the order $2^{6}$
\begin{eqnarray}\label{R}
R \subset R_1,
\end{eqnarray}
which consists of transformations of $R_1$ (\ref{R1}), which preserves the equations (\ref{eq13}).

The subgroup (\ref{R}) keeps pies of resedues $v_{1,\k_i}$, $v_{2,\k_i}$ the two resedues is translated into conjugated simulataneously. The same property for the resedues  $v_{1,\k_i}$, $v_{2,\k_i}$.

Define the regularization over (\ref{Ibsub})
\begin{eqnarray}\label{RIbsub}
RKK_{\I\I_b} \mapsto KK_{\I\I_b}
\end{eqnarray}
as the quadruple covering with respect to the index $2^5$ subgroup (\ref{R}) in (\ref{DDDD}). The  denotation
\begin{eqnarray}\label{barRIbsubb}
\pi_{RKK}: \overbrace{RKK_{\I\I_b}} \mapsto \overbrace{KK_{\I\I_b}}
\end{eqnarray}
for the  canonical iterated $4$-sheeted covering over the covering (\ref{RIbsub}) is used.
Transformations of the covering space (\ref{barRIbsubb}) is described by the group of the order $2^4$.

\begin{lemma}\label{RK}
	There existe a codimentsion $2$ subpolyhedron in $KK_{\I\I_b}$ ower which the the covering (\ref{RIbsub}) is trivial.
\end{lemma}

\subsubsection*{A geometrical description of the regularization mapping (\ref{RIbsub}), proof of Lemma \ref{RK}) \label{re}}

The regularization mapping (\ref{RIbsub}) is defined as a covering, associated with the index $4$ subgroup (\ref{R}). Geometrically, the regularization mapping is an inclusion of a subpolyhedron and is defined as the following. This geometrical regularization is a more general, because the construction is well-defined over the ambient polyhedron (\ref{IIIb}).

Take the projection $p_1 \times p_2$ (\ref{p_i}), $i=1$, see below.
Because the stabilization keeps the structuring group of (\ref{IIIb}), see subsubsection 
\ref{stabilizz}, one may assume that the images of the $\x-,\y-$-coordinate projections $p_1, p_2$ are inside the polyhedra (\ref{Ib}). Take the composition with the resolution mappings 
(\ref{piF}) for the coordinates. We get the following mapping:
\begin{eqnarray}\label{piFF}		
		(p_1 \times p_2) \circ (\pi_1 \times \pi_2): KK_{\I_b} \to K_{\I_b;1} \times K_{\I_b;2} \to S^1 \times S^1 = B\Z(A) \times B\Z(\dot A).
\end{eqnarray}
Take the self-intersection points of the target mapping of the composition (\ref{piFF}) and define the subpolyhedron of
the self-intersection by $RKK_{\I\I_b} \subset KK_{\I\I_b}$. We may see that the subpolyhedron
admits the required reduction of the structured group as described by (\ref{RIbsub}). Take a closed loop in $RKK_{\I\I_b}$ and assume that the inverse image of the loop by $p_1 \times p_2$ is con-connected. Then the each covering loop projects  by $p_i \circ \pi_i, \quad i=1,2$  to a common element on the circle $B(\Z)$. This determines the required reduction (\ref{R}) of the structuring group. \qed

\subsubsection{A central extension of the structuring group and the involution $\tau\tau_c$}\label{tautau}
Recal, the antidiagonal the polyhedron $K_{\I_b}$ is a double covering over $\widehat{K}_{\I_b}$, where the structuring group of the covering is described by the central extension of $\D$ by the element $c$, $c^2=-1$, see subsubsection \ref{tauc}.

We define the analoug of the involution on $RKK_{\II_b}$, deneted by $\tau\tau_c$.
This construction may be apply to $x$ and $y$ coordinates independently. The lift of the
involution $\tau_c$ on $K_{\I_b}$ to the involution $\tau\tau_c$ on $KK_{\I_b}$ is free
outside the thin diagonal $\Delta_{anti}$ (see subsection \ref{anti} ). The involution
is defined by the translation of ciclic coordinates of strata. By this translation the collection of primery resedues remains unchanged. 
The collection of secondary resedues $vv_i \in \{\pm1\}$ is translated into the opposite.
By this trick the involution $\tau\tau_c$ over $x$- and $y$- coordinates changes strata into corresponding pairs.

\subsubsection{Antidiagonal strata}\label{anti}

For the polyhedron (\ref{IIIb}) 
let us define exceptional (thin) antidiagonal strata by the formula:
\begin{eqnarray}\label{antidiag}
	\begin{array}{c}
		\Delta_{anti}^{[0]}= \{(\x_{1,1}, \x_{1,2}),(\y_{1,1},\y_{1,2});(\x_{2,1}, \x_{2,2}),(\y_{2,1},\y_{2,2})\}: \\
		\i\x_{1,1} = \x_{1,2}, \ \i\x_{2,1} = \x_{2,2}, \ \i\x_{1,1} = \x_{2,1}, \\ 
		\i \y_{1,1} = \y_{1,2}, \  \i \y_{2,1} = \y_{2,2}, \ \i\y_{1,1} = \y_{2,2}.\\	
	\end{array}
\end{eqnarray}
The relation on $\Delta_{anti}$ (comp. with the formula (\ref{crossresedu}) for cross-resedues) shows that  of $\x$, $\y$-coordinates are not symmetric with  primary $1$ and $2$-coordinates
on the configuration space.
For thin antidiagonal deeper strata formulas are analogous. 

The union of all antidiagonal strata (see below for the definition \ref{antidiagT} of the thick antidiagonal strata $\Delta\Delta_{anti}$). The thick antidiagonal strata  inside
$RKK_{\II_b}$ is defined with the condition (we get 3 equations in (\ref{antidiag}) instead of 2 equations in (\ref{antidiagT}) ):
\begin{eqnarray}\label{extra}
	vv_i=+1 (over \quad {\rm{x}}); \qquad ww_j = +1 (over \quad {\rm{y}}).
	\end{eqnarray}
The conditios for $X$- and $y$- coordinates may be considered independently.

The union of all thick antidiagonal strata determines the polyhedron (\ref{Ibsub})
 and the following sequence of polyhedrons:
\begin{eqnarray}\label{antidia}
\Delta_{anti} \subset \Delta\Delta_{anti} \subset  RKK_{\I\I_b} \subset KK_{\I\I_b} \subset KK_{\I_b}. 
\end{eqnarray}

A general structuring group $\D\D \rtimes \D\D$ over the subpolyhedron (\ref{antidiag}) admits
a reduction into the diagonal subgroup $\{b \times 1; 1 \times \dot{b}\} = \Z/4 \times \Z/4$, $b \in \I_{b}, \bb \in \I_{\bb}$. 
Below we prove that the structure group over the subpolyhedrons $RKK_{\I\I_b}$ (\ref{RIbsub}) outside the thin antidiagonal is 
$[\Z/2 \times (\I_b \rtimes \Z(A))] \times [\Z/2 \times (\I_{\bb} \rtimes \Z(\dot{A}))]$,
where $\I_b \rtimes \Z(A)$, $\I_{\bb} \rtimes \Z(\dot{A})$ are the groups as in the upper left vertex of the diagram (\ref{eq:square2}).

\subsubsection{A stabilization of $KK_{\circ}$ \label{stabilizz}}


By a stabilization of the polyhedron  $\NN$ we means an inclusion of the polyhedron of the size $r_1 \times r_2$ into very large polyhedron of a size 
$R$, when by a stabilization an arbitrary strata of $KK_{\I_d; \circ}$ is included into a strata of a large polyhedron with a full collection of residues $\{v_i,w_j\}$.
 The stabilization for an each coordinate is analogous to (\ref{in}).

 The mapping (projection)
\begin{eqnarray}\label{p_i}
p_i: \overbrace{\NN}/T^{int} \to K_1(r_1) \times K_2(r_2), \quad i=1,2
\end{eqnarray}
is well defined. This projection 
 for $i=1$ omits the second superscript (for $i=2$ omits the first superscript) of a pair of points are well defined. The mapping (\ref{p_i}) is compressed with respect to 
 $\overbrace{\NN}/T^{int} \to \NN$. There is an inclusion:
\begin{eqnarray}\label{inclu}
 K_1(r_1) \times K_2(r_2) \subset K(r_1+2) \times K(r_2+2), 
\end{eqnarray}
which is an analog of the stabilization inclusion (\ref{in}). By the stabilization the image of an elementary simplex of $K_1(r_1) \times K_2(r_2)$. The construction gives the required stabilisation of $p=p_1(\overline{\NN}/T^{int})$. The stabilization preserves the structuring group of
(\ref{IIId}).

We may define the stabilization of the polyhedron $\overline{\NN}/T_{i+1}$ taking the partition (\ref{quadrupleBIS}) and define the projections $q_i: \overline{\NN}/T_{i+1} \to K_1(r_1) \times K_2(r_2) \subset K_1(R) \times K_2(R)$ analogously to  $p$. 
The thick antidiagonal consists of points:
\begin{eqnarray}\label{antidiagT}
\begin{array}{c}
\Delta\Delta_{anti}^{[0]}= \{(\x_{1,1}, \x_{1,2}),(\y_{1,1},\y_{1,2});(\x_{2,1}, \x_{2,2}),(\y_{2,1},\y_{2,2})\}: \\
\i\x_{1,1} = \x_{2,1}, \ \i\x_{1,2} = \x_{2,2}, \\ 
\i \y_{1,1} = \y_{2,2}, \  \i \y_{1,2} = \y_{2,1}.\\	
\end{array}
\end{eqnarray}
inside the subpolyhedron (\ref{IIIb}). This thick antidiagonal (\ref{antidiagT}) contains the antidiagonal (\ref{antidiag}). Formally speaking, the thick antidiagonal contains secondary diagonal strata by the formulas:
\begin{eqnarray}\label{diagT}
 \{(\y_{1,1},\y_{1,2})=\emptyset, (\y_{2,1},\y_{2,2})=\emptyset\}, \quad or, \quad
 \{(\x_{1,1}, \x_{1,2})=\emptyset;(\x_{2,1}, \x_{2,2})=\emptyset\}.
 \end{eqnarray}
This subpolyhedron (\ref{diagT}) is delated from (\ref{barKK}). One may see, formally, this
subpolyhedron is inside the boundary of (\ref{Idsub}). By the resolution,  this polyhedron is 
eliminated, we may assume that  the boundary of (\ref{Ibsub}) does not  contain this polyhedron.

The stabilization preserves the structuring group of
(\ref{Idsub}).
The stabilization of $\NN$ is used  to cancel the polyhedron (\ref{IIId}) by a codimension $2$ formal vertical deformation, comp. with Lemma  \ref{223}.  The stabilizations is used in Lemma \ref{l13} with a cancelation of the polyhedron (\ref{Idsub}).

\subsection{A resolution of $RKK_{\I\I_b}$}\label{rezol}

Denote $r_1^+ = r_1+2$, $r_2^+ = r_2+2$ 
Take the polyhedron $RKK_{\I\I_b}^+$, which is construct by the stabilization $\NN_{\circ}(r_1,r_2) \mapsto \NN_{\circ}(r_1^+,r_2^+)$ with additional $4$-coordinates. Take a generic inclusion
 $RKK_{\I\I_b}^+ \supset RKK_{\I\I_b}$, called the stabilization. After the stabilization the image of polyhedron  $RKK_{\I\I_b}$ contains only maximal strata
 with  full $x$-,$y$-coordinate collections $\{-\i,+\i\}$ of residues $v_{1;\k_i},w_{1;\k_j}\}$
 and full collection $\{+1,-1\}$ of cross-resedues. 
 \begin{eqnarray}\label{crossrel1}
 \{vv_{1;\k_i},  vv_{2;\k_i};  ww_{1;\k_j},   ww_{2;\k_j}\}, \quad 1 \le i \le r_1^+, 1 \le j \le r_2^+. 
\end{eqnarray}

Let us apply  Lemma \ref{222}, statement 2  for the $x$-,$y$-coordinates to get the resolution of $RKK_{\I\I_b}^+$.


Take the quotient $(\RP^{n_1-k_1} \times \RP^{n_2-k_2})/\{I_1,I_2\}\}$ (for the definition $I_1,I_2$, see subsubsection  \ref{III}), take the standard complex conjugation
involution of this quotient, denote it by
$$ Conj_1 \times Conj_2: (\RP^{n_1-k_1} \times \RP^{n_2-k_2})/\{I_1,I_2\} \to (\RP^{n_1-k_1} \times \RP^{n_2-k_2}).$$
Denote by 
$$C_{Conj_1 \times Conj_2} = ((\RP^{n_1-k_1} \times \RP^{n_2-k_2})/\bar{I}) \rtimes_{Conj_1 \times Conj_2} S^1\times S^1, $$
the torus of the involution $Conj_1 \times Conj_2$.

There exist a mapping:
\begin{eqnarray}\label{s12}
	s:  RKK_{\I\I_b} \setminus \Delta\Delta_{anti} \longrightarrow  C_{Conj_1 \times Conj_2}
\end{eqnarray}
with the source is outside the thick antidiagonal,   the coordinate mappings (\ref{piF}) near the antidiagonal are, generally speaking, disconnected
This equivariant mapping (using the stabilization (subsubsection \ref{stabilizz})  of the polyhedron we may assume tar all strata outside $\Delta_{anty}$ have mixed collextion of primery resedues $\{+\i, -\i\}$) is extended to the following mapping outside the this diagonal:
\begin{eqnarray}\label{s13}
	s:  RKK_{\I\I_b} \setminus \Delta_{anti} \longrightarrow  C_{Conj_1 \times Conj_2}
\end{eqnarray}
This mapping satisfies  special boundary conditions on the thin antidiagonal, anagously . 

\subsubsection{The canonical mapping (\ref{s12})}\label{can}
 The mapping for ais defined as following. Let us present the construction with the assumption that the polyhedron (\ref{IIIb}) is stable in the sense of subsubsection \ref{stabilizz} and, also we may assume that the numbers $r_1^+, r_2^+$ of angles coordinates of a maximal strata are equal $3 \pmod{4}$, see Definition \ref{3pmod}.
 The canonical mapping for an original polyhedron is determined using the composition with the stabilization mapping.
For an elementary strata the full subcollection of the cross-residues is divided into two subcollections: 
\begin{eqnarray}\label{VV+-}
V_+=\{v_{1;\k_i}=+\i\},  V_- =\{v_{1;;\k_i}=-\i\}, \quad V= V_+ \cup V_-;
\end{eqnarray}
\begin{eqnarray}\label{WW+-}
	W_+=\{w_{1;\k_i}=+\i\}, W_- =\{w_{1;;\k_i}=-\i\}, \quad W= W_+ \cup W_-.
\end{eqnarray}
Because of the equation (\ref{eq14}), the second collection of cross-residues is unrequired.

 This subcollection of the coordinates with the  momenta determines the coordinate projections (\ref{s12}) onto $(\RP^{n_1-k_1} \times \RP^{n_2-k_2})$; restricted on an elementary stratum by identity on  angles $V_+$ and by the conjugation on angles $V_-$.

The mappings (\ref{s12}) are defined using mappings (\ref{piF}), see Theorem \ref{secoequil},  for each $\x, \ i=1; \y, \  i=2$-projection. The mapping (\ref{piF}) considered for each coordinate along the circle coordinate $S^1$ is not well defined near antidiagonal, but the monodromy of this mapping takes even values. Along the circle generator the mapping $\pi_{i}$, $i=1,2$ changes the collection $V$ to its complement conjugated subcollection. This gives 
the mapping (\ref{s12}), where the projection on the torus $p_C: C_{Conj_1 \times Conj_2}  \to S^1 \times S^1$
corresponds to the composition $p_C \circ s_i = \pi_{i}$, $i=1,2$. 

The mappings $s_1 \times s_2$ are not covering mappings over a common mapping (\ref{RIbsub}).
Theorem \ref{secoequil} below defines a 
"secondary" \- mapping:
\begin{eqnarray}\label{sig12}
	\sigma:   RKK_{\I\I_b} \setminus \Delta\Delta_{antidiag}  \longrightarrow CC_{Conj^2}, \quad i=1,2,
\end{eqnarray}
where 
$$ Conj^2: \RP^{n_1-k_1}/I_1 \times \RP^{n_2-k_2}/I_2 \to \RP^{n_1-k_1}/I_1 \times \RP^{n_2-k_2}/I_2$$
is the pair of the complex conjugations over the coordinates, $CC_{Conj}$ is the torus of this two involutions.
The source of the mapping $(\ref{sig12})$  is considered  outside the thick antidiagonal (\ref{antidiagT}).

Denote by 
\begin{eqnarray}\label{primcov}
\pi:	\overbrace{KK_{\I_d;\circ}} \longrightarrow \overline{KK_{\I_d;\circ}}
\end{eqnarray}	
the canonical primery double covering (interiore in the formula  (\ref{KK}) over the secondary double covering (exteriore in the formula  \ref{KK})  over the polyhedron (\ref{IIId}).
Denote by 
\begin{eqnarray}\label{primcov}
	p:	\overbrace{RKK_{\I\I_d}} \longrightarrow RKK_{\I\I_d}
\end{eqnarray}	
the restriction of the canonical secondary double covering over the regularized (see subsubsection \ref{re}) polyhedron $(\ref{IIId})$.

The following theorem is the analoug of Lemma \ref{223}. This is required in Lemma \ref{l13}.

\begin{theorem}\label{rezbb}
	
	1.   The canonical primary covering $(\ref{primcov})$ over the polyhedron $\overline{KK_{\I_d;\circ}}$   is a pull-back of the standard covering over the circles.
	The formula of the covering is invariant with respect to the  involution $\chi: (\x,\y) \mapsto (\y,\x)$ (subsubsection \ref{subchi}).

	2.    The  canonical secondary covering $\overbrace{RKK_{\I\I_d}} \longrightarrow RKK_{\I\I_d}$ is a pull-back of  the product of the two coordinate standard coverings over the circle. The $\chi$-involution. The construction is skew-invariant with respect to the  involution $\chi: (\x,\y) \mapsto (\y,\x)$.

\end{theorem}

\subsubsection*{Proof of Theorem \ref{rezbb}}

Statements are  corollaries of Lemma \ref{223}. Apply a stabilization, described in subsubsection  \ref{stabilizz}. We may prove the lemma with the assumption that all elementary strata admits $\x$ and $\y$-residues of all different type. Let us consider the  projection 
$$ p: \overbrace{KK_{\I\I_d}} \longrightarrow \overline{K_{\I_d;\circ}} \subset \overline{K}_{\circ;1} \times \overline{K}_{\circ;2}, $$ 
on the polyhedrons  (\ref{Nd}), where $\overline{K_{\I_d;\circ,i}} \subset \overline{K}_{\circ,i}$, $i=1,2$ are  the subpolyhedrons, defined by the formula (\ref{Id}). 
 The canonical (primary) covering over $\overbrace{KK_{\I\I_d;i}}$ is the pull-back of
 the coordinate canonical coverings over  $\overline{K_{\I_d;\circ,1}} \times \overline{K_{\I_d;\circ,2}}$.
By Lemma  \ref{223} each canonical coordinate covering is the pull-back of the covering over the circle.

The construction of the covering is given by the primary residues (\ref{resedu}), which are invariant with respect to the involution $\x \mapsto \y$. The formula of the residues used two coordinates primary indexes, which are invariant.
Statement 1 is proved. 
 
	Proof of Statement 2. Consider the canonical secondary covering $\overbrace{RKK_{\I_b}} \longrightarrow RKK_{\I_b}$ and restrict this covering over the subpolyhedron $RKK_{\I\I_d}
	\subset  KK_{\I_b}$. 
	 Let us apply a stabilization and decompose, analogous with the proof Statement 1, the polyhedron as following:
	\begin{eqnarray}\label{dec}
		\overline{KK}_{\I\I_d} = \overline{KK}_{\I\I_d;1} \cup \overline{KK}_{\I\I_d;2},
		\end{eqnarray}
		where $\overline{KK}_{\I\I_d;i}$, $i=1,2$ is a polyhedron defined as the inverse  image of the subpolyhedron $\bar{K}_{\I_d;\circ} \subset \overline{K}_{\circ}$
		by the coordinate projection:  
		$$p_i: \overline{KK}_{\I\I_d} \to \bar{K}_{\I_d;\circ},$$
		using the collection of secondary residues 	and the stable condition.	
	

		Each polyhedron in the right-hand side of (\ref{dec}) is a double covering of the 
		projection on the first index $1$ in (\ref{c}).
		This decomposition is well defined, because for $\overline{KK}_{\I\I_b}$   real cross-residues
		$vv_{1,\k_i}$, $vv_{2,\k_i}$ are equal by (\ref{eq14}). This proves that the first index $1$ is global over the subpolyhedron (\ref{dec}) with external $T^{ext}$-equivariant structure (\ref{collectinv}). An analogous
		construction for the second coordinate  uses $ww_{1,\k_j}$, or,  $ww_{2,\k_j}$  cross-residues.

	The canonical (secondary) covering over $KK_{\I\I_d;i}$, $i=1,2$ is a pull-back of
	the canonical covering over  $K_{\I_d;\circ}$; by Lemma  \ref{223} this covering is the pull-back of the covering over the circle.
	The construction of the covering is given by the secondary cross-residues (\ref{crossresedu}), which are skew-invariant with respect to the involution $\x \mapsto \y$. The formula of the cross-residues used two secondary coordinates, the involution $\chi$ permutes the secondary coordinates.
	Statement 2 is proved. Theorem \ref{rezbb} is proved. \qed

\subsubsection{Properties of $\overline{RKK}_{\I\I_b}$}

Let us consider a sequence (\ref{antidia}). Let us define a splitting
\begin{eqnarray}\label{splitt}
SP: RKK_{\I\I_b} \setminus \Delta\Delta_{anti}	 \longrightarrow K_{\I_b,1} \times K_{\I_b,2} \setminus (\Delta_1 \times K_{\I_b,2} \cup K_{\I_b,1} \times \Delta_2).
\end{eqnarray}	 
The mapping (\ref{splitt}) (as usual, the formula is presented for a maximal strata) in the spherical coordinates is  given by (comp. subsubsection \ref{can}): 
\begin{eqnarray}\label{splitting}
\begin{array}{c}
 (\check x_{1,1; i}, \check x_{1,2; i}), (\check y_{1,1; j}, \check y_{1;2; j}); 
(\check x_{2,1; i}, \check x_{2,2; i}) ; (\check y_{2,1; j}, \check y_{2,2; j}) \mapsto \\
(\check x_{1,1; i}, \check x_{2,1; i}) \times (\check y_{1,2; j}, \check y_{2;2; j}).
\end{array}
\end{eqnarray}


Statement 1 of the following theorem is the analoug of Theorem \ref{t11}. Statement 2
is the analoug of Lemma  \ref{222}. Statement 3 is used in  Lemma \ref{l13}.

\begin{theorem}\label{secoequil}
\[ \]
1. There exists the $\hat{I},T^{ext}$-equivariant resolution mapping 
\begin{eqnarray}\label{Piff}
	\pi = \pi_1 \times \pi_2: \overline{RKK}_{\I_b} \setminus \Delta\Delta_{anti}  \to S^1 \times S^1,
\end{eqnarray}
which is the Cartesian product of  the two coordinate mappings (\ref{piF}).

2.  The mapping (\ref{Piff}), restricted to   $\Delta\Delta_{anty} \cap RKK_{\I_b}$ is continuous in  a complement of
a polyhedron  of dimension $[\frac{r_1 + r_2 -2}{4}]+2$.  


3. Assume we have a generic polyhedron $Q \subset RKK_{\I_b}$.
Let $RKK_{\I\I_b}(Q) \subset RKK_{\I\I_b}$ be a subpolyhedron which consists of quadruples
which are inside $Q$. Then the singular points of the restriction of the mapping (\ref{Piff}) to $KK_{\I\I_b}(Q)$ is a codimension $\frac{3(r_1+r_2-2)}{4}-2$ subpolyhedron  $SS(Q) \subset Q$.

\end{theorem}

\subsubsection*{Proof of Theorem \ref{secoequil}}
Proof 1. The required mapping (\ref{Piff}) outside the antidiagonal is induced
from coverings over the coordinate mappings (\ref{piF}) using the stabilization (\ref{stabilizz}) and the equivariant  mappings (\ref{splitt}).  The coordinate mappings corresponds to
to $\I_b$ and $\II_b$ cyclic structures, or,  to resedues and cross-resedues of $\x$ and $\y$ coordinates, subsubsection \ref{can}. When we take a generic cyclic  transformation of $x-$ coordinates, imaginary $W$ resedues remains unchanged, but the corresponding collection residues $V$  is translated into the conjugated. (The corresponding secondary $vv$-collection remain unchanged.)  This observation is required to well formulate boundary conditions along the antidiagonal.

Proof 2.  Let us investigate boundary conditions for the mapping (\ref{Piff}).
Take the polyhedron $RKK_{\I_b} \setminus \Delta\Delta_{anti}$. Denote by $S$ the  bondary of the polyhedron $RKK_{\I_b}$ near the thick antidiagonal.  The image of the momentum mapping $\lambda$ on $S$ is closed to the codimension $1$  skeleton in
$\Delta^{r_1-1} \times \Delta^{r_2-1}$. 

Denote by $\Pi_1 \subset RKK_{\I_b}$, $\Pi_2 \subset RKK_{\I_b}$ 
the polyhedron, where the mapping $\pi_1$, $\pi_2$ has ramifications.
The polyhedron $\Pi_1$ is described using the collection of primery resedues $\{v_1\} \in \pm\i$ and the collection of the secondary resedues  $\{vv_1\} \in \pm 1$ over the $x$-coordinate. The ramification is along a subpolyhedron in $\Delta\Delta_{anti}$, where
the collection of resedues $v_1$ is divided coordinates into the halph into two subcollections: $+\i$ and $-\i$, and the collection of secondary cross-resedues $vv_i$ in divided the collections of coordinates with $v_i=+\i$ and with $v_i=-\i$  into the two secondary subcollections  with $vv_i$ $\{+1\}$ and with $\{-1\}$. (Outside the subpolyhedron
of primery ramification monodromy of resedues is the constant, inside the subpolyhedron of primery ramification and outside the subsubpolyhedron of secondary ramification the monodromy is the constant.)

The polyhedron $\Pi_1$ is the union of $\Pi_1^{\pm\i}$ and $\Pi_1^{\pm1}$.  Define   $\Pi_1^{\pm\i}$ as the collection of strata of $RKK_{\I_b}$, for which the number of primery resedues $+\i$ equals to the number of primery resedues $-\i$ up to $\pm 1$. Define $\Pi_1^{\pm1}$ for which the number of secondary resedues $+1$ equals to the number of secondary resedues $-1$ up to $\pm 1$. 
Outside $\Pi_1=\Pi_1^{\i} \cup \Pi_1^{\pm1}$ the mapping (\ref{Piff})  (including the factor $\pi_2$) is homotopic to a constant. The resedues type is preserved along a closed path and a coordinate system
transforms by the identity.
The same construction holds for $\Pi_2=\Pi_2^{\pm\i} \cup \Pi_2^{\pm1}$.

Consider $\Pi_1 \cap \Pi_2 \subset RKK_{\I_b}$. The complement of this polyhedron 
 $\overline{\Pi_1} \cup \overline{\Pi_2}$ contains the common intersection 
  $\overline{\Pi_1} \cap \overline{\Pi_2}$.  The restrictions  of the mapping (\ref{Piff})
to each polyhedron and to each complement is constant. Because the intersection $\overline{\Pi_1} \cap \overline{\Pi_2}$ is connected, the restriction of the mapping (\ref{Piff}) to the union  $\overline{\Pi_1} \cup \overline{\Pi_2}$ is homotopic to a constant.

Then for the closure of $\Pi_1 \cup \Pi_2$ in $S$ we get 
 $dim(\Pi_1 \cup \Pi_2 \cap S) = [\frac{r_1+r_2}{4}]+2$.

Outside the polyhedron $\Pi_1 \cap \Pi_2$ the mapping $\pi$ is homopopic to a constant mapping. 

Proof 3. Because $Q \subset RKK_{\I_b}$ is codimension $4q$ generic, this implies that the polyhedron $KK_{\I\I_b}(Q)$ is generic with respect to  strata of $KK_{\I_b}$.
This means that each strata  of $KK_{\I\I_b}(Q)$ is a codimension $4q$ in the corresponding cell of $KK_{\I\I_b}$. The singular polyhedron consists of strata of the codimension $\frac{3(r_1+r_2-2)}{4}-2$ in $KK_{\I_b}$, then the codimension of $SS(Q) \subset Q$ is the same.

Theorem \ref{secoequil} is proved. \qed

Let us formulate an analog of Lemma \ref{l22}. 
An analogue of the involution $\tau_c$ from Subsection \ref{tauc} on the polyhedron $KK_{\I\I_b}$  is required. Consider the projection (\ref{splitt}). Coordinate two involutions $\tau_{c;i}$, $i=1,2$ (subsubsection \ref{tautau}) 
determine the involution on $KK_{\I\I_b}$ outside the thin antidiagonal. The diagonal of $\tau_{c;i}$-involutions is the involution on $KK_{\I\I_b;\circ}$, which is free outside the thick antidiagonal. Denote this involution by $\tau\tau_c$. On the polyhedron 
$KK_{\I\I_{a \times \aa}}$, the involution $\tau\tau_c$ is analogously defined.


\begin{lemma}\label{l24}
	
	1. The Hurewitcz image of the fundamental class of the antidiagonal polyhedrons $\Delta_{antydiag}$ (\ref{antidia}) in
	$KK_{\I\I_{a \times \aa};\circ}$
	in the group $H_{n_1+n_2-k_1-k_2}(K(\D \rtimes \D,1);\Z/2)$, $dim(KK_{\I\I_{a \times \aa}})=n_1+n_2-k_1-k_2$ 
	is the generator $t_{n_1-k_1} \otimes t_{n_2-k_2} \in H_{n_1+n_2-k_1-k_2}(K(\Z/4 \times \Z/4 ;\Z/2)) \to H_{n_1+n_2-k_1-k_2}(K(\D \rtimes \D,1);\Z/2)$.

	2. The Hurewicz image by the structured homomorphism of the fundamental class of an arbitrary codimension $2s$ subpolyhedron in $\NN_{\I\I_b}$ $($in $\NN_{\I\I_{a \times \aa}}$$)$, which is
	invariant with respect to the involution $\tau\tau_c$ of the double covering
	 $\NN_{\I\I_{b}} \longrightarrow \widehat{\NN}_{\I\I_{b}}$
	  $(\NN_{\I\I_{a \times \aa}} \longrightarrow \widehat{\NN}_{\I\I_{a \times \aa}})$  in the group $H_{n_1+n_2-k_1-k_2-2s}(K(\D \rtimes \D,1))$  is trivial.
\end{lemma}

\subsubsection*{Proof of Lemma \ref{l24}}

Proof  is analogous to Lemma \ref{l22}. \qed


Let us prove an analog of Lemma \ref{222} Statement 2 for polyhedrons $RKK_{\I\I_b}$.

Recall, $dim(KK_{\I\I_b})=n_1-k_1+n_2-k_2, \quad n_1,n_2 \equiv 1 \pmod{2}$, $k_1,k_2 \equiv 0 \pmod{2}$.
The boundary of the polyhedron $RKK_{\I\I_b}$ consists of points on all regular strata, which are closed to anti-diagonal (\ref{antidiag}).
Let us consider  of the open polyhedron,
a small regular neighbourhood of the thin antidiagonal polyhedron $\Delta_{anti} \subset RKK_{\I\I_b}$, which is denoted by
$U_{\circ} \subset RKK_{\I\I_b} \setminus \Delta_{anti}$.

Consider the mapping $U_{\circ} \to S^1 \times S^1$, which is the restriction of 
  the equivariant mapping (\ref{Piff}). Because the structuring group
$U_{\circ}$ is reduced to the subgroup  $\I_b \times \I_{\bb} = (b,\dot{b}) $,
the mapping
$U_{\circ} \to S^1$ contains values in the even index $2\times 2$ subgroup $2[\pi_1(S^1)]\times 2[\pi_1(S^1)] \subset \pi_1(S^1) \times \pi_1(S^1)$.

\subsubsection{The popyhedron $KK_{\I\I_b}$ with an  involution $\chi$  \label{subchi}}

In this subsection we construct a polyhedron $KK_{\I\I_b}$ as the union of coordinate subpolyhedra by the formula (\ref{unify}).
On this polyhedron 
the involution $T^{int}$ (\ref{Tint}) is invariant.

Take the union of the two polyhedra 
\begin{eqnarray}\label{unify}
\begin{array}{c}
KK_{\I\I_b}=KK_{\I\I_b}(r_1,r_2) \cup_{KK_{\I\I_b}{r_2,r_2)}} KK_{\I\I_b}(r_2,r_1), \\
r_1 + r_2 = r; \quad r_1 = -1 \pmod{4}; r_2 = -1 \pmod{4}.
\end{array}
\end{eqnarray}
In the case $r_1=r_2=r$ we get $KK_{\I\I_b}=KK_{\I\I_b}(r,r)$. 
On the polyhedron $KK_{\I\I_b}$ there is the involution
$$\chi: KK_{\I\I_b} \to KK_{\I\I_b},$$
which is defined by the permutation of the coordinates $(\x,\y) \mapsto (\y,\x)$. 
This involution is the diagonal of the transformation $T^{int}$,
Constructions of the subsection \ref{propr} are invariant with respect to the involution $\chi$.

Petr Akhmetiev, Troitsk, IZMIRAN \ pmakhmet@mail.ru

\end{document}